\documentclass[11pt,fleqn]{article}
\usepackage{amsfonts}
\usepackage{amssymb}
\usepackage{relsize}
\usepackage[leqno]{amsmath}
\usepackage{multicol}
\usepackage{blkarray}
\usepackage[export]{adjustbox}
\usepackage{array}
\usepackage{framed}

\makeatletter
\newcommand{\leqnomode}{\tagsleft@true}
\newcommand{\reqnomode}{\tagsleft@false}
\makeatother

\newcommand{\R}{\mathbb{R}}
\newcommand{\defin}{\mathrel{\mathop :}=}

\usepackage[UKenglish]{babel}
\usepackage{enumerate}
\usepackage{hyperref}
\usepackage{latexsym}
\usepackage{lmodern}

\usepackage{tikz}
\usetikzlibrary{arrows,backgrounds}
\usetikzlibrary{decorations.pathmorphing}

\usepackage{graphicx}

\usepackage{bbm}
\usepackage{caption}
\usepackage{subcaption}

\usepackage[margin=2.25cm]{geometry}
\usepackage{marvosym}

\usepackage{amsthm}

\theoremstyle{plain}
\newtheorem{theorem}{Theorem}[section]

\newtheorem{lemma}[theorem]{Lemma}
\newtheorem{corollary}[theorem]{Corollary}
\newtheorem{proposition}[theorem]{Proposition}
\newtheorem{defn}[theorem]{Definition}

\theoremstyle{remark}
\newtheorem{remark}{Remark}

\newtheorem{innercustomgeneric}{\customgenericname}
\providecommand{\customgenericname}{}
\newcommand{\newcustomtheorem}[2]{%
  \newenvironment{#1}[1]
  {%
   \renewcommand\customgenericname{#2}%
   \renewcommand\theinnercustomgeneric{##1}%
   \innercustomgeneric
  }
  {\endinnercustomgeneric}
}

\newcustomtheorem{customlemma}{Lemma}

\theoremstyle{plain}

\title{Generalized Ellipsoids}

\author{Amir Ali Ahmadi
\thanks{Princeton University, Operations Research and Financial Engineering. AAA was partially supported by the MURI award of the AFOSR, the Sloan Fellowship, the Princeton AI Lab Seed Grant, and the Princeton SEAS Innovation Grant. Email: {\tt\small aaa@princeton.edu}}
\ \ \ \ \ \ \ \ \
Abraar Chaudhry \thanks{Georgia Institute of Technology, H. Milton Stewart School of Industrial and Systems Engineering. AC was partially supported by the MURI award of the AFOSR. Email: {\tt\small achaudhry61@gatech.edu}}
\ \ \ \ \ \ \ \ \
Cemil Dibek\thanks{Ko\c{c} University, Turkey, Department of Industrial Engineering. Email: {\tt \small cdibek@ku.edu.tr}}
}

\date{}
\begin{document}
\maketitle

%%%%%%%%%%%%%%%%%%%%%%%%%%%%%%%%%%%%%%
\begin{abstract}
We introduce a family of symmetric convex bodies called \emph{generalized ellipsoids of degree $d$} (GE\nobreakdash-$d$s), with ellipsoids corresponding to the case of $d=0$. Generalized ellipsoids (GEs) retain many geometric, algebraic, and algorithmic properties of ellipsoids. We show that the conditions that the parameters of a GE must satisfy can be checked in strongly polynomial time, and that one can search for GEs of a given degree by solving a semidefinite program whose size grows only linearly with dimension. We give an example of a GE which does not have a second-order cone representation, but show that every GE has a semidefinite representation whose size depends linearly on both its dimension and degree. In terms of expressiveness, we prove that for any integer $m\geq 2$, every symmetric full-dimensional polytope with $2m$ facets and every intersection of $m$ co-centered ellipsoids can be represented exactly as a GE-$d$ with $d \leq 2m-3$. Using this result, we show that every symmetric convex body can be approximated arbitrarily well by a GE-$d$ and we quantify the quality of the approximation as a function of the degree $d$. Finally, we present applications of GEs to several areas, such as time-varying portfolio optimization, stability analysis of switched linear systems, robust-to-dynamics optimization, and robust polynomial regression.
\\

\end{abstract}

%%%%%%%%%%%%%%%%%%%%%%%%%%%%%%%%%%%%%%
\reqnomode
\section{Introduction} \label{sec:intro}

An \emph{ellipsoid} in Euclidean space $\R^n$ is a set of the type
\begin{equation}\label{eq:ellipsoid_defn}
\hspace{4.3cm}
\mathcal{E} = \{x \in \R^n \: | \:  (x-x_0)^TP(x-x_0) \leq 1\},
\end{equation}
where $P$ is a (symmetric) positive definite matrix and $x_0\in\mathbb{R}^n$ is a given vector. Ellipsoids are among the most prominent examples of convex sets in applied and computational mathematics. In optimization, they represent sublevel sets of objective functions of convex quadratic programs, feature in the description of celebrated algorithms such as the ellipsoid method and Dikin's method, and serve as primary examples of uncertainty sets in robust optimization. In control and robotics, they appear as sublevel sets of quadratic Lyapunov functions or in the description of the manipulability set of a robotic system. In convex geometry, they are used to approximate convex bodies with an approximation guarantee established by John's ellipsoid theorem. In probability and statistics, they appear as confidence regions for Gaussian or more generally elliptical distributions. 
%appear in the description of elliptical distributions or as confidence regions for
%appear in the description of elliptical distribution in probability, and define confidence regions in statistics, e.g., for Gaussian distributions.
These examples are a small sample among many. We refer the reader to \cite[Chap. 1]{Todd} for some reasons why ellipsoids are so ubiquitous in many areas.

% Ellipsoids are among the most prominent examples of convex sets in optimization and related fields. For example, they represent sublevel sets of objective functions in convex quadratic programming, feature in the description of optimization algorithms such as the ellipsoid method and Dikin's method, serve as primary examples of uncertainty sets in robust optimization, appear as sublevel sets of quadratic Lyapunov functions in dynamics and control, feature in fundamental theorems of convex geometry such as John's ellipsoid theorem, are used to describe the manipulability of a robotic system in robotics, appear in the description of elliptical distribution in probability, and define confidence regions in statistics, e.g., for Gaussian distributions.
% See, e.g., Chapter 1 of the book~\cite{Todd} for some reasons why ellipsoids are so ubiquitous in many areas of optimization.

% \textcolor{red}{Rewrite with no story}
% - algorithms -- ellipsoid method, Dinkin ellipsoid
% - robust optimization -- primary example of uncertainty sets
% - convex geometry -- John's ellipsoid theorem
% - dynamics and control -- sublevel sets of quadratic Lyapunov functions
% - convex quadratic programming - sublevel sets of the objective function
% - probability - confidence regions of Gaussian distributions

Since ellipsoids are sublevel sets of strictly convex quadratic polynomials, a very natural generalization of ellipsoids would be to consider sublevel sets of strictly convex polynomials of degree higher than two. However, unless P=NP, the set of convex (or strictly convex) polynomials of degree at least four does not admit a tractable description~\cite{AOPT}. This implies that algorithms based on this approach would in general not scale well with increasing dimension.

%in general not be computationally tractable. 

%[***AAA: can potentially remove this and comments on the remarks in the contributions***]
%This set can be approximated by a hierarchy of semidefinite programs using sum of squares based sufficient conditions for convexity. However, even for a fixed level of this hierarchy, the size of these semidefinite programs grows quickly with the dimension and the degree of the polynomial.

%Ellipsoids could potentially be generalized in several ways. In \cite{NieParSturm}, for example, the authors generalize ellipsoids considering the set of points whose sum of distances to $k$ given points is at most some fixed number. This generalization of ellipsoids, however, is primarily explored in dimension 2 within the paper, and with a focus on their algebraic properties rather than their optimization perspective or their representation power. Another natural approach to generalize ellipsoids would be to consider sublevel sets of (strictly) convex polynomials of higher degrees. Sets obtained through this generalization, however, do not have a tractable representation since testing convexity of a polynomial of degree at least 4 is NP-hard~\cite{AOPT}. In the literature, using sum of squares relaxations is suggested to obtain sublevel sets of the so-called sos-convex polynomials. Even though this leads to semidefinite programs, the size of these programs quickly becomes large with the increasing dimension and degrees. 

%In this work, we introduce \emph{generalized ellipsoids of degree $d$} as a generalization of ellipsoids 

In this work, we propose a different generalization of ellipsoids to sets that we call \emph{generalized ellipsoids of degree $d$} (GE-$d$s), with ellipsoids corresponding to the case of $d=0$.
%; see Definition~\ref{def:GE_d_defn}. 
%GE-$d$s 
Generalized ellipsoids (GEs) retain some key geometric and algebraic properties of ellipsoids, such as the properties of being convex and semialgebraic. Importantly, they are also algorithmically tractable to search for and optimize over. The reason for this tractability stems from the fact that our generalization (see Definition~\ref{def:GE_d_defn}) keeps the defining inequalities of a GE \emph{quadratic} in $x\in\mathbb{R}^n$, while adding a \emph{single} new variable on which these inequalities depend polynomially. Despite the univariate nature of this dependence, we show that GEs can approximate any $n$-dimensional symmetric convex body to arbitrary accuracy.

\subsection{Organization and main contributions}

The remainder of this paper is organized as follows. In Section~\ref{sec:GE_definition}, we give the definition of generalized ellipsoids, justify our definition, and provide some examples.

In Section~\ref{sec:recog_opt}, we focus on recognition of GEs and search for GEs. In Section~\ref{subsec:efficient_recog}, we show that the conditions that the parameters of a GE must satisfy can be checked in strongly polynomial time. In particular, we show that one can check if a univariate polynomial matrix is positive semidefinite over an interval (or the real line) in strongly polynomial time.
This result may be of independent interest.
In Section~\ref{subsec:efficient_search}, using the fact that certain low-degree sum of squares tests for positive semidefiniteness of polynomial matrices are exact, we show that one can search for GEs of a given degree by solving a single semidefinite program. In fact, the size of this semidefinite program grows only \emph{linearly} with dimension.

%we observe that the search for GEs can be done efficiently through semidefinite programming.

%The univariate nature of this dependence ensures that certain low-degree sum of squares tests for positive semidefiniteness of polynomial matrices are exact. This turns out to imply that one can both search for GE-$d$s and optimize over GE-$d$s by solving a single semidefinite program (as opposed to a hierarchy). In fact, the size of these semidefinite programs grows only linearly with dimension.

In Section~\ref{sec:conic_rep}, we investigate whether GEs can be represented as the feasible set of three increasingly expressive families of tractable conic programs. This is relevant to applications involving optimization over a GE. In Section~\ref{subsec:intersection}, we show that when $d \geq 2$, GE-$d$s cannot always be described by finitely many convex quadratic constraints. In Section~\ref{subsec:socp_rep}, we show that when $d \geq 16$, GE-$d$s do not always have a second-order cone representation. In Section~\ref{subsec:sdp_representable}, we show that every GE has a semidefinite representation whose size depends linearly on both its dimension and degree.

In Section~\ref{sec:rep_power}, we focus on the expressive power of GEs. We show that for any integer $m\geq 2$, every compact intersection of $m$ ``semiellipsoids'', and in particular every symmetric full-dimensional polytope with $2m$ facets and every intersection of $m$ co-centered ellipsoids, can be represented exactly as a GE-$d$ with $d \leq 2m-3$. 
A technical lemma that goes into this proof shows that for any dimension $m \geq 2$, there is a polynomial curve of degree $2m-3$ that lies within the unit simplex in $\R^m$ and visits every one of its corners.
We believe this statement and our game-theoretic proof of it may be of independent interest.
We then show that every symmetric convex body can be approximated arbitrarily well by a GE-$d$ and we quantify the quality of the approximation as a function of the degree~$d$.
In Section~\ref{subsec:GEPloarity}, we discuss how GEs can behave under polar duality.

In Section~\ref{sec:applications}, we present four applications involving GEs and provide some numerical examples. In Section~\ref{subsec:markowitz}, we consider a time-varying extension of the minimum-variance portfolio optimization problem in finance. 
%Our formulation searches for a portfolio which has minimum GE-$d$-norm, a norm obtained by GEs. We show with a numerical example that GE-based portfolios have lower variance than the portfolio coming from the solution to minimizing the worst-case variance with respect to the measurements. 
In Section~\ref{subsec:jsr}, we consider an application in dynamical systems and show that asymptotically stable switched linear systems always admit a GE as an invariant set. 
%Equivalently, we show that if the joint spectral radius of a set of matrices is less than one, then there exists a ``contracting'' GE-$d$-norm. 
In Section~\ref{subsec:RDO}, we show how GEs can provide inner approximations to feasible sets of robust-to-dynamics optimization problems when the dynamical system is uncertain.
%We demonstrate with a numerical example that GEs provide good inner approximations to these feasible sets. 
In Section~\ref{subsec:shift regression}, we present an application of GEs to the problem of polynomial regression in statistics when there is uncertainty in the measurements.
%describe how GEs arise when solving a version of the problem of fitting a univariate polynomial to some observations, in which the observations are assumed to be slightly shifted due to some errors in measurement. 
%We show with a numerical example that polynomials obtained through GEs are significantly more robust to a small shift in observations and are overall much smoother than the polynomials obtained through usual least-squares polynomial regression.

Finally, in Section~\ref{sec:future_directions}, we list a few questions for future research.

%%%%%%%%%%%%%%%%%%%%%%%%%%%%%%%%%%%%%%%%%%%%%%%%%%

\section{Definition of GEs} \label{sec:GE_definition}

We begin by establishing some notation and terminology. We denote the set of real symmetric $n \times n$ matrices by $S^n$. We write $M \succeq 0$ if $M \in S^n$ is positive semidefinite (psd) and $M \succ 0$ if $M$ is positive definite (pd). We denote the set of $n \times n$ psd (resp. pd) matrices by $S_+^n$ (resp. $S_{++}^n$). The kernel of $M$ is denoted by $\text{Ker}(M)$. We refer to a matrix with polynomial entries as a \emph{polynomial matrix}. 
%We denote the set of symmetric $n \times n$ univariate polynomial matrices of degree (at most) $d$ by $S_d^n[t]$. When $n=1$, we use the simplified notation $\R_d[t]$. 
We can now give the definition of our generalization of ellipsoids.

\begin{defn}\label{def:GE_d_defn}
A set $\mathcal{E}_d \subset \R^n$ is a \emph{generalized ellipsoid of degree $d$} (GE-$d$) if it can be written as 
\begin{equation}\label{eq:GE_d_defn}
\hspace{3cm}
\mathcal{E}_d = \{x \in \R^n \: | \:  (x-x_0)^TP(t)(x-x_0) \leq 1 \:\: \forall t \in [-1,1]\}
\end{equation}
for some vector $x_0 \in \R^n$ and some univariate polynomial matrix $P(t)$ of degree (at most) $d$ that satisfies
\begin{itemize}
\item $P(t) \succeq 0 \quad \forall t \in [-1,1]$, and 
\item $\bigcap\limits_{t \in [-1,1]} \text{Ker}(P(t)) = \{0\}$.
\end{itemize}
We refer to these two conditions as the ``psd condition'' and the ``kernel condition'', respectively. We say that a set is a \emph{generalized ellipsoid} (GE) if it is a GE-$d$ for some nonnegative integer $d$.
\end{defn}
Observe that ellipsoids correspond precisely to GE-$0$s. Indeed, when $d=0$, $P(t)$ is a constant matrix, say $P(t) = P$, and we have $P \succ 0$ if and only if $P \succeq 0$ and $\text{Ker}(P) = \{0\}$. It is straightforward to check that GEs are symmetric convex bodies (i.e., compact convex sets with non-empty interior that are symmetric around their center). Throughout this paper, without loss of generality, we assume that the center $x_0$ of our GEs is at the origin. Figure~\ref{fig:GE_examples} demonstrates a few examples of GEs, together with the corresponding polynomial matrices $P(t)$.

%A set $\Omega \subseteq \mathbb{R}^n$ is symmetric if (after possible translation) we have $x \in \Omega$ if and only if $-x \in \Omega$ for every point $x \in \R^n$.

\begin{figure}[ht]
\centering
\begin{subfigure}{.34\textwidth}
  \centering
  \includegraphics[width=.6\linewidth]{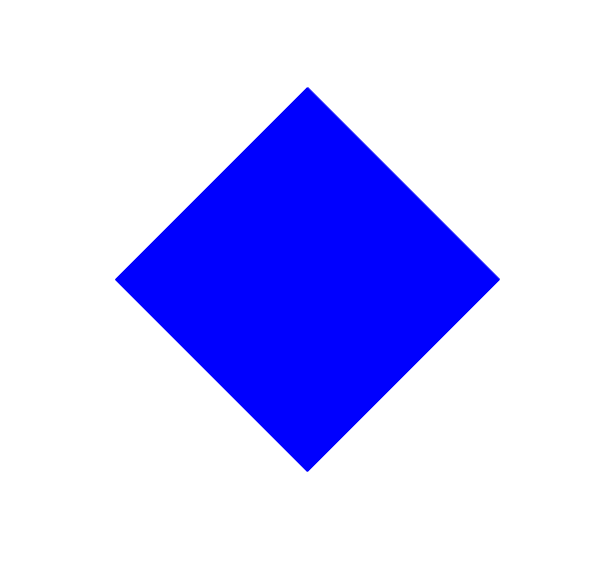}
  \vspace{-0.4cm}
  \caption{$P(t) = \begin{bmatrix} 1 & t \\ t & 1 \end{bmatrix}$}
  \label{fig:sub1}
\end{subfigure}%
\begin{subfigure}{.34\textwidth}
  \centering
  \includegraphics[width=.6\linewidth]{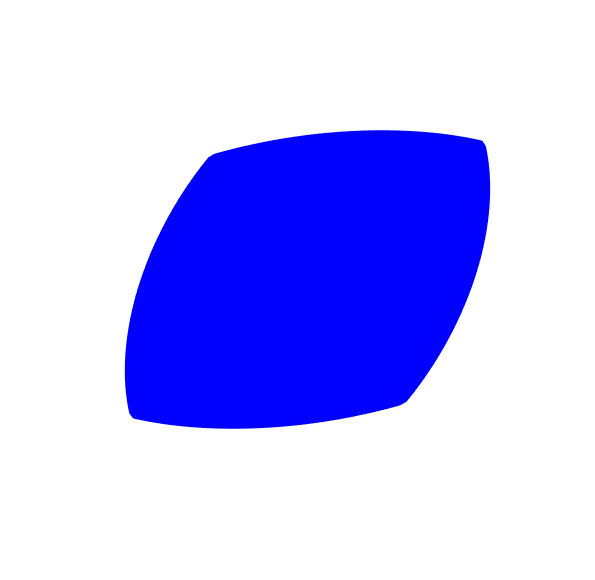}
  \vspace{-0.4cm}
  \caption{$P(t) = \begin{bmatrix} 4t+8 & -3t-1 \\ -3t-1 & -4t+12 \end{bmatrix}$}
  \label{fig:sub2}
\end{subfigure}%
\begin{subfigure}{.34\textwidth}
  \centering
  \includegraphics[width=.6\linewidth]{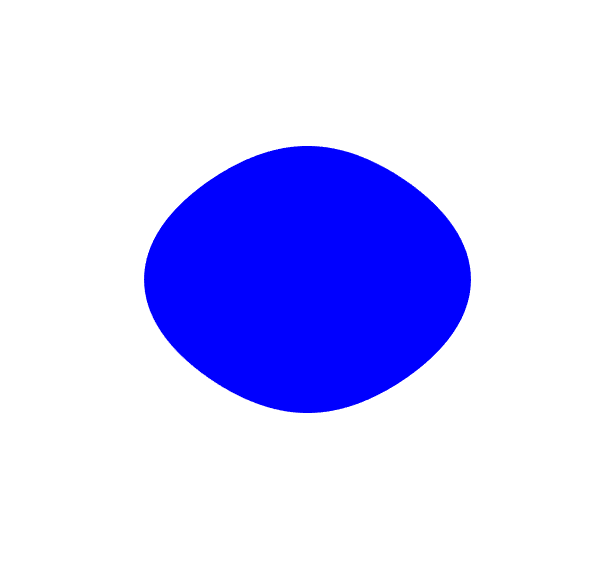}
  \vspace{-0.4cm}
  \caption{$P(t) = \begin{bmatrix} 2-t^2 & t \\ t & 3-t^2 \end{bmatrix}$}
  \label{fig:sub3}
\end{subfigure}
\begin{subfigure}{.45\textwidth}
  \centering
  \vspace{0.8cm}
  \includegraphics[width=.55\linewidth]{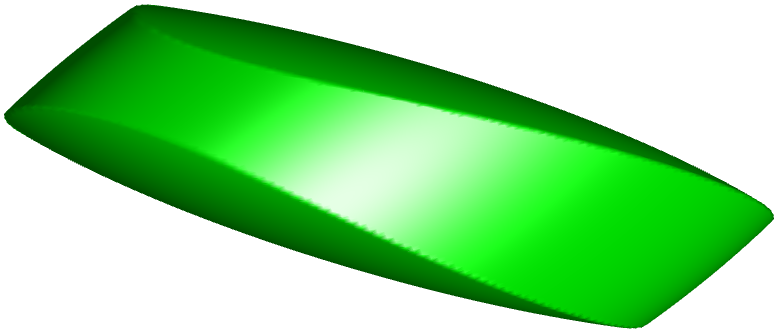}
  \vspace{0.4cm}
  \hspace{-0.8cm}
  \caption{$P(t) = \begin{bmatrix} 6-4t & 2-3t & 1-t \\ 2-3t & 3-t & -t \\ 1-t & -t & 2 \end{bmatrix}$}
  \label{fig:sub4}
\end{subfigure}
\begin{subfigure}{.45\textwidth}
  \centering
\vspace{0.8cm}
  \includegraphics[width=.5\linewidth]{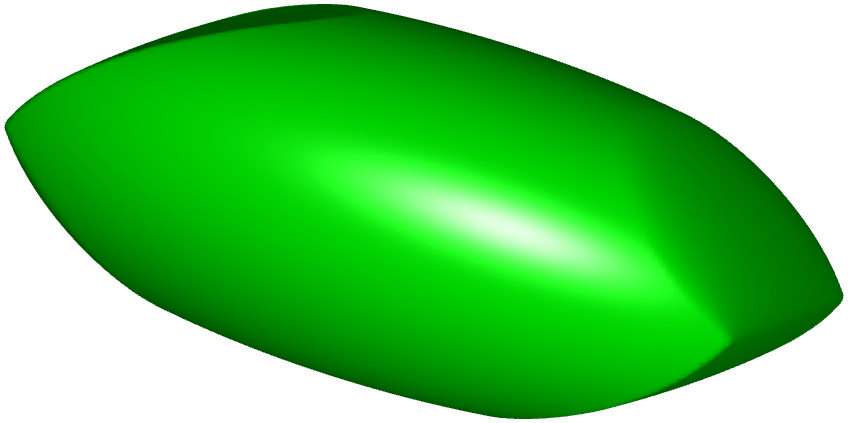}
  \vspace{0.4cm}
  \hspace{-1.5cm}
  \caption{$P(t) = 
  \begin{bmatrix} 
  4-t^2 & t+1 & 2-t^2 \\ 
   t+1 & 2t^2 + t + 2 & t^2-t+1 \\ 
  2-t^2 & t^2-t+1 & 2t^2+2t+4 
  \end{bmatrix}$}
  \label{fig:sub5}
\end{subfigure}
\caption{Some examples of GEs with the corresponding $P(t)$ matrices.}
\label{fig:GE_examples}
\end{figure}

Note that a GE-$d$ has a compact representation in terms of the coefficients of the polynomial matrix $P(t)$. The reason that the matrix $P(t)$ in the definition of GEs depends on a single variable $t$, and that this dependence is polynomial, is justified by algorithmic tractability purposes (see Section~\ref{sec:recog_opt} and Section~\ref{sec:conic_rep}). Yet, this family of sets is quite expressive (see Section~\ref{sec:rep_power}).
The domain for the variable $t$ in Definition~\ref{def:GE_d_defn} is chosen to be bounded since over an infinite interval, any non-constant polynomial matrix would go to infinity in norm, which would preclude the set $\mathcal{E}_d$ from having full dimension.
Among bounded intervals, the choice of $[-1,1]$ is without loss of generality since any other bounded interval can be transformed into $[-1,1]$ via shifting and scaling.

Now, let us just justify the kernel condition. When generalizing ellipsoids, one may find it more natural to consider positive definiteness of $P(t)$ for all $t \in [-1,1]$, or for some $t \in [-1,1]$. Suppose $P(t) \succeq 0 \:\: \forall t \in [-1,1]$ and consider the following three candidate conditions:

\vspace{0.2cm}
\noindent (i) $P(t) \succ 0 \:\: \forall t \in [-1,1]$,
\hspace{1cm}
(ii) $P(t) \succ 0$ for some $t \in [-1,1]$,
\hspace{1cm}
(iii) $\bigcap\limits_{t \in [-1,1]} \textup{Ker}(P(t)) = \{0\}$.

\noindent Clearly,  (i) $\Rightarrow$ (ii) $\Rightarrow$ (iii). However, (ii) $\not \Rightarrow$ (i), as seen, e.g., by
$P(t) =
\begin{bmatrix}
     1-t & 0 \\
     0 & 1+t
\end{bmatrix}$. Similarly, (iii)~$\not \Rightarrow$~(ii); as seen, e.g., by 
$$P(t) =
\begin{bmatrix}
     (1+t)^2 & (1-t)(1+t) \\
     (1-t)(1+t) & (1-t)^2
\end{bmatrix}.$$
Indeed, we have $\textup{Ker}(P(1)) \cap \textup{Ker}(P(-1)) = \{0\}$, but $P(t)$ is not pd for any $t \in [-1,1]$ as $\det(P(t))$ is identically zero. Hence, our choice of the kernel condition leads to a more inclusive definition.

We also note that a set of the type~\eqref{eq:GE_d_defn} with $P(t)\succeq 0$ $\forall t\in [-1,1]$ is a convex body if and only if the kernel condition is satisfied. We restate this claim in a lemma below in terms of norms that convex bodies define. Recall that any ellipsoid $\mathcal{E}$ as in~\eqref{eq:ellipsoid_defn} defines an ellipsoid (quadratic) norm by
$$||x||_{\mathcal{E}} = \sqrt{x^TPx}.$$
Similarly, for any generalized ellipsoid $\mathcal{E}_d$ as in~\eqref{eq:GE_d_defn}, we can define a \emph{generalized ellipsoid norm of degree $d$} (GE-$d$-norm) by
\begin{equation}\label{eq:GE_d_norm}
\hspace{5.4cm}
||x||_{\mathcal{E}_d} = \max\limits_{t \in [-1,1]} \sqrt{x^TP(t)x}.
\end{equation}
The proof of the following lemma is straightforward and hence omitted.

\begin{lemma}
\label{lem:norm_iff}
%Let $||x||_{\mathcal{E}_d} = \max\limits_{t \in [-1,1]} \sqrt{x^TP(t)x}$ where $P(t)$ is a univariate polynomial matrix of degree~$d$. Suppose $P(t) \succeq 0 \:\: \forall t \in [-1,1]$ (so that $||x||_{\mathcal{E}_d}$ is well-defined). Then, $||x||_{\mathcal{E}_d}$ is a norm if and only if $\bigcap_{t \in [-1,1]} \textup{Ker}(P(t)) = \{0\}$.
A function $f: \R^n \rightarrow \R$ of the type $f(x) = \max\limits_{t \in [-1,1]} \sqrt{x^TP(t)x}$ with $P(t) \succeq 0$ $\forall t \in [-1,1]$ is a norm if and only if $\bigcap_{t \in [-1,1]} \textup{Ker}(P(t)) = \{0\}$.
\end{lemma}

\section{Recognition of GEs and Search for GEs} \label{sec:recog_opt}

\subsection{Efficient recognition of GEs}\label{subsec:efficient_recog}

In this section, we show that GEs can be recognized in strongly polynomial time\footnote{We recall that a \emph{strongly polynomial time algorithm} is an algorithm such that (i) it consists of the elementary arithmetic operations:
addition, subtraction, comparison, multiplication, and division, (ii) the number of elementary operations depends polynomially on the dimension of the input to the algorithm, and (iii) the encoding length of the numbers occurring during the algorithm is bounded by a polynomial function of the encoding length of the input. We also recall that when we speak of polynomial time or strongly polynomial time algorithms, we are working in the Turing model of computation where the input to the problem consists of rational numbers and hence has finite encoding length. Here, the encoding length could for example be taken as the number of bits required in a binary representation of the input. See \cite{Grotschel1988} for more details. In our case, the input to the problem is the rational coefficients of the entries of $P(t)$ and the dimension of the input is the total number of coefficients which is $\left(\substack{n+1 \\ 2}\right) (d+1)$.}; i.e, given the coefficients of the entries of a univariate polynomial matrix $P(t)$, we can check both the kernel condition and the psd condition of Definition~\ref{def:GE_d_defn} in strongly polynomial time. In particular, the running time of the recognition procedure is polynomial in both the dimension $n$ and the degree $d$ of the GE-$d$, and its number of arithmetic operations does not depend on the encoding length of the coefficients of $P(t)$.

\begin{lemma}
\label{lem:kernel_cond}
Given a univariate polynomial matrix $P(t)$, one can check in strongly polynomial time whether $\bigcap_{t \in [-1,1]} \textup{Ker}(P(t)) = \{0\}$.
\end{lemma}

\begin{proof}
The kernel condition is equivalent to checking whether there is a vector $x \neq 0$ such that $P(t)x = 0$ $\forall t \in [-1,1]$. Let $P(t) = P_0 + P_1 t + P_2 t^2 + \dots + P_d t^d$ where $P_i \in S^n$ for $i=0,1,\dots,d$.  Since a univariate polynomial vanishes on $[-1,1]$ if and only if all of its coefficients are zero, the kernel condition is also equivalent to the existence of a vector $x \neq 0$ such that $P_0x = 0$, $P_1x = 0$, $\dots$, $P_dx = 0$. These equations can be collectively expressed as a single linear system $Ax = 0$, where
$$
A = \begin{bmatrix}
P_0 \\
P_1 \\
\vdots \\
P_d
\end{bmatrix} \in \mathbb{R}^{n(d+1) \times n}
$$
is the vertical concatenation of the matrices $P_0,P_1,\dots,P_d$. To check existence of a nonzero solution to the system $Ax=0$, we can equivalently check whether the rank of $A$ is less than $n$, which can be done in strongly polynomial time, e.g., by Edmonds' implementation of Gaussian elimination (see, e.g., \cite[Corollary~1.4.9.b]{Grotschel1988}).
\end{proof}

Next, we show that checking the psd condition in Definition~\ref{def:GE_d_defn} can be done in strongly polynomial time. The following theorem, and the lemma it relies on, may be of independent interest.

\begin{theorem}
\label{thm:psd strongly poly}
Given a univariate polynomial matrix $P(t)$, one can check in strongly polynomial time whether $P(t) \succeq 0$ $\forall t \in [-1,1]$.
\end{theorem}

To prove this theorem, we first establish the following lemma which generalizes a result in~\cite{real_stability_testing}.

\begin{lemma}\label{lem:psd strongly poly}
Given a univariate polynomial matrix $P(t)$, one can check in strongly polynomial time whether $P(t)\succeq 0$ $\forall t \in \R$.
\end{lemma}
\begin{proof}
%For $k=1,\dots,n$, let $P(t)_k$ denote the $k \times k$ leading principal submatrix of $P(t)$. For $k=1,\dots,n$, and $t \in \R$, define the univariate polynomial
% $$p_{k,t}(s) = \det (P(t)_k + sI_k),$$
% where $I_k$ is the $k \times k$ identity matrix.
% For $k=1,\dots,n$ and $i=0,\dots,k$ let $c_{k,i}(t)$ denote the coefficient of $s^i$ in $p_{k,t}(s)$.
Define the univariate polynomial
$$p_{t}(s) = \det (P(t) + sI),$$
where $I$ is the $n \times n$ identity matrix.
For $i=0,\dots,n$, let $c_{i}(t)$ denote the coefficient of $s^i$ in $p_{t}(s)$.
We first claim that for any $t \in \R$, $P(t) \succeq 0$ if and only if $c_{i}(t) \geq 0$ for $i=0,\dots,n$.

Fix $t \in \R$ and suppose first that $P(t)\succeq 0$.
For $j=1,\dots,n$, let $\lambda_j$ denote the (nonnegative) eigenvalues of $P(t)$.
Then we can write
$$
p_{t}(s) = \prod_{j=1}^n (\lambda_j + s).
$$
It is clear from this representation that all coefficients of $p_{t}(s)$ are nonnegative.

To see the converse, fix $t \in \R$ again and suppose $c_{i}(t) \geq 0$ for $i=0,\dots,n$.
For $j=1,\dots,n$, let $z_j$ denote the roots of $p_{t}(s)$.
Observe that the eigenvalues of $P(t)$ are precisely $-z_j$ for $j=1,\dots,n$.
Given that all coefficients of $p_{t}(s)$ are nonnegative, this polynomial cannot have any positive roots.
Thus $P(t)$ cannot have any nonnegative eigenvalues and thus, $P(t) \succeq 0$.
This completes the proof of the claim.

It remains to show that one can test nonnegativity of each function $c_{i}(t)$ in strongly polynomial time.
For the remainder of the proof, fix $i$ to be an integer between $0$ and $n$.
Observe that $c_{i}(t)$ is a polynomial in $t$ of degree $d(n-i)$.
To obtain the coefficients of $c_{i}(t)$, it suffices to evaluate this polynomial at $d(n-i) + 1$ distinct points and then solve a nonsingular linear system for the coefficients of the interpolating polynomial (recall one can solve a nonsingular linear system in strongly polynomial time; see, e.g., \cite[Corollary~1.4.9.a]{Grotschel1988}).
To evaluate $c_{i}(t)$ at a particular $t$, we need access to the coefficient of $s^i$ in $p_{t}(s)$.
To compute these coefficients, it similarly suffices to evaluate the polynomial $p_{t}(s)$ at $n+1$ distinct points and then solve an associated linear system.
Recall that the determinant of a constant matrix can be computed in strongly polynomial time (see, e.g., \cite[Corollary~1.4.9.c]{Grotschel1988}); hence, each evaluation of $p_{t}(s)$ can be done in strongly polynomial time via the determinantal definition of $p_{t}(s)$. Hence, overall, we can compute the coefficients of $c_{i}(t)$ by computing $(d(n-i) + 1)(n+1)$ determinants and solving $d(n-i) + 2$ linear systems.
With the coefficients of $c_{i}(t)$ at hand, one can check the nonnegativity of each $c_{i}(t)$ in strongly polynomial time via the method of \cite[Section~5]{real_stability_testing}.

\end{proof}

\begin{proof}[Proof of Theorem~\ref{thm:psd strongly poly}]
We claim that a polynomial matrix $P(t)$ of degree $d$ is psd for every $t \in [-1,1]$ if and only if the polynomial matrix $Q(t) = (t^2+1)^d P \left( \frac{t^2-1}{t^2+1} \right)$ is psd for every $t \in \R$.

%($\Rightarrow$)

Suppose first that  $P(t)\succeq 0$ $\forall t \in [-1, 1]$. For every $t \in \R$, we have $\frac{t^2-1}{t^2+1} \in [-1,1)$, and thus, $P \left( \frac{t^2-1}{t^2+1} \right) \succeq 0$.
Since $(t^2+1)^d$ is nonnegative, it follows that $Q(t) \succeq 0$ $\forall t \in \R$.

%($\Leftarrow$)
To see the converse, suppose for the sake of contradiction that $P(t) \not \succeq 0$ for some $t \in [-1,1]$. By the closedness of the psd cone, there is some $\hat{t} \in [-1,1)$ such that $P(\hat{t}) \not \succeq 0$.
Consider $\bar{t} = \sqrt{\frac{-1-\hat{t}}{-1+\hat{t}}}$.
Observe $Q(\bar{t}) = (\frac{-1-\hat{t}}{-1+\hat{t}} + 1)^d P(\hat{t})$.
Since $\hat{t} \in [-1,1)$, we have $\frac{-1-\hat{t}}{-1+\hat{t}} + 1 > 0$.
Therefore, $Q(\bar{t}) \not \succeq 0$, a contradiction.

Thus, to check that $P(t)$ is psd for every $t \in [-1,1]$, we can check if $Q(t)$ is psd for every $t \in \R$.
It is straightforward to see that the coefficients of $Q(t)$ can be derived from those of $P(t)$ in strongly polynomial time.
Hence, the result follows from Lemma~\ref{lem:psd strongly poly}.
\end{proof}

%%%%%%%%%%%%%%%%%%%%%%%%%%%%
\subsection{Efficient search for GEs}\label{subsec:efficient_search}

%In many applications, the matrix $P(t)$ in the representation of a GE may not fixed but instead serves as a decision variable. In this section,

In this section, we observe that the set of $n\times n$ polynomial matrices $P(t)$ of degree $d$ that satisfy $P(t) \succeq 0$ $\forall t \in [-1,1]$ has a semidefinite representation of size linear in $d$ (resp. in $n$) for fixed $n$ (resp. fixed $d$). As semidefinite programs (SDPs) can be solved in polynomial time to arbitrary accuracy~\cite{BoydSurvey}, this observation leads to efficient algorithms for applications where one needs to \emph{search} for GEs (see, e.g., Section~\ref{subsec:markowitz} and Section~\ref{subsec:RDO}). It can also reformulate the problem of checking if a given polynomial matrix $P(t)$ satisfies $P(t) \succeq 0$ $\forall t \in [-1,1]$ (i.e., the recognition question of the previous subsection) as a semidefinite programming feasibility problem. However, it is currently unknown whether semidefinite programming feasibility problems can be solved in polynomial time (let alone strongly polynomial time), and this is the justification for our alternative algorithm in the proof of Theorem~\ref{thm:psd strongly poly}.

The semidefinite programming formulation that we present here arises from a connection to sum of squares polynomials. 
%(see, e.g.,~\cite{Lasserre} and~\cite{Parrilo}) 
We recall that a polynomial $p: \R^n \rightarrow \R$ is \emph{nonnegative} if $p(x) \geq 0$ for all $x \in \R^n$ and a \emph{sum of squares} (sos) if there exist polynomials $q_1(x),\ldots,q_m(x)$ such that $p(x) = \sum_{i=1}^m q_i^2(x)$. A polynomial matrix $Y: \R^{n} \rightarrow S^{k}$ is said to be an \emph{sos-matrix} if $Y(x) = A(x)^TA(x)$ for some (not necessarily square) polynomial matrix $A(x)$. It is a classical fact in algebra that a univariate polynomial matrix $Y(t)$ is positive semidefinite for all $t\in \mathbb{R}$ if and only if it is an sos-matrix; see, e.g.,~\cite{Yakubovich,Youla} for some early proofs, \cite[Theorem~7.1]{ChoiLamRez80} for a short constructive proof that specifies the inner dimension in the factorization, and \cite{aylward2007explicit} for more context, a decomposition algorithm, and connections to other areas.
We note that this result no longer holds if one considers a matrix of more than a single variable (see, e.g., \cite{choi}).

For our purposes, we need a statement that applies to univariate polynomial matrices that are positive semidefinite over an interval. This variant appears for example in \cite[Theorem~2.5]{DetStud} and \cite[Theorem 6.11]{papp2013semidefinite}.

%; equivalently, if the polynomial $y^T P(x)y$ in the variables $(x, y)$ is sos. 

%we show that checking the psd condition can be reduced to a family of semidefinite programming conditions of polynomial size. Although we previously established that checking the psd condition can be performed in strongly polynomial time (Theorem~\ref{thm:psd strongly poly}), and despite the fact that semidefinite programming conditions do not yield even a polynomial-time method, our motivation for presenting this connection is twofold. First, it clarifies the relationship with sum of squares polynomials. Second, in many applications, the matrix $P(t)$ is not fixed but instead serves as a decision variable, i.e., we search for GEs (we search for $P(t)$'s). In such cases, this semidefinite programming-based approach remains viable.

%We make use of the fact that any univariate polynomial matrix that is positive semidefinite over an interval has a certain sum of squares representation. This representation can be obtained by leveraging the renowned relationship between sum of squares polynomials and semidefinite programs (see, e.g.,~\cite{Lasserre} and~\cite{Parrilo}). 

%A symmetric $m \times m$ polynomial matrix $P(x)$ in $n$ variables is said to be positive semidefinite, denoted by $P(x) \succeq 0$, if $P(x)$ is positive semidefinite for all $x \in \mathbb{R}^n$. It is straightforward to see that this condition holds if and only if the (scalar) polynomial $y^T P(x)y$ in $m+n$ variables $(x, y)$ is nonnegative.

\begin{theorem}
%(\cite[Theorem~2.5]{DetStud},
%\cite[Theorem~7.1]{ChoiLamRez80}, 
%\cite[Theorem 6.11]{papp2013semidefinite}, 
%\cite[Theorem 2]{aylward2007explicit})
\label{thm:univarite_decomp}
Let $P(t)$ be a symmetric univariate polynomial matrix of degree $d$. If $d$ is odd, then $P(t) \succeq 0$ $\forall t \in [-1,1]$ if and only if there exist sos-matrices $X_1(t)$ and $X_2(t)$ of degree $d-1$ such that
%(not necessarily square) polynomial matrices $B(t)$ and $C(t)$ of degree $\frac{d-1}{2}$ such that 
$$
%P(t) = (t+1)B(t)^TB(t) + (1-t) C(t)^TC(t)
P(t) = (t+1)X_1(t) + (1-t) X_2(t) \quad \forall t \in \R.
$$
Similarly, if $d$ is even, then $P(t) \succeq 0$ $\forall t \in [-1,1]$ if and only if there exist sos-matrices $X_1(t)$ and $X_2(t)$ of degree $d$ and $d-2$, respectively, such that
%(not necessarily square) polynomial matrices $B(t)$ and $C(t)$ of degree $\frac{d}{2}$ and $\frac{d}{2}-1$, respectively, such that
$$
%P(t) = B(t)^TB(t) + (1-t^2) C(t)^TC(t).
P(t) = X_1(t) + (1-t^2) X_2(t) \quad \forall t \in \R.
$$
\end{theorem}

%Obviously, if $P(x)$ is an sos-matrix, then $P(x)$ is positive semidefinite.

As we describe next, Theorem~\ref{thm:univarite_decomp} leads to a semidefinite representation of polynomial matrices that are psd on $[-1,1]$. This is essentially based on the following facts: (i) A polynomial matrix $Y: \R^{n} \rightarrow S^{k}$ is an sos-matrix if and only if the (scalar-valued) polynomial $y^T Y(x) y$ in the variables $(x_1,\dots,x_n, y_1, \ldots, y_k)$ is a sum of squares; (ii) A polynomial $p: \R^n \rightarrow \R$ of degree $2d$ is a sum of squares if and only if there exists a psd matrix $Q$ such that $p(x)=v(x)^TQv(x)$, where $v(x)$ is the vector of all monomials of degree up to $d$ (see, e.g., \cite{ChoiLamRez, ParriloThesis}).

%The latter condition is equivalent to existence of a positive semidefinite matrix $Q$ such that certain polynomial identity holds. An explicit proof of the following proposition can be found in~\cite{TV-SDP}.

%A polynomial $p: \R^n \rightarrow \R$ of degree $2d$ is a sum of squares if and only if there exists a psd matrix $Q$ such that $p(x)=v(x)^TQv(x)$, where $z(x)$ is the vector of all monomials of degree up to $d$. 

\begin{proposition}
\label{prop:sdp_transformation}
Let $P(t)$ be a symmetric univariate $n \times n$ polynomial matrix of degree $d$. For a positive integer $d'$, let 
$$v_{d'}(t,y) \defin \left (y_1,\dots,y_n, y_1t, \dots, y_nt, \dots, y_1t^{d'}, \dots, y_nt^{d'} \right)^T$$ be the vector of all monomials of the form $y_{\ell}t^k$ for $\ell=1,\dots,n$ and $k=0,\dots,{d'}$. 

If $d$ is odd, then $P(t) \succeq 0$ $\forall t \in [-1,1]$ if and only if there exist positive semidefinite matrices $Q_1, Q_2$ of size $(\frac{d+1}{2})n \times (\frac{d+1}{2})n$ that satisfy the following equations ($\forall t \in \R$ and $\forall y \in \R^n$):
\begin{itemize}
\itemsep0em
\item $P(t) = (t+1)X_1(t) + (1-t) X_2(t)$ ,
\item $y^T X_i(t) y = v_{\frac{d-1}{2}}(t,y)^T Q_i v_{\frac{d-1}{2}}(t,y)$ for $i=1,2$.
\end{itemize}
Similarly, if $d$ is even, then $P(t) \succeq 0$ $\forall t \in [-1,1]$ if and only if there exist positive semidefinite matrices $Q_1, Q_2$ of size $(\frac{d}{2}+1)n \times (\frac{d}{2}+1)n$ and $(\frac{d}{2})n \times (\frac{d}{2})n$, respectively, that satisfy the following equations ($\forall t \in \R$ and $\forall y \in \R^n$):
\begin{itemize}
\itemsep0em
\item $P(t) = X_1(t) + (1-t^2) X_2(t)$,
\item $y^T X_1(t) y = v_{\frac{d}{2}}(t,y)^T Q_1 v_{\frac{d}{2}}(t,y)$,
\item $y^T X_2(t) y = v_{\frac{d}{2}-1}(t,y)^T Q_2 v_{\frac{d}{2}-1}(t,y)$.
\end{itemize}
\end{proposition}

Since two polynomials are equal everywhere if and only if they have the same coefficients, Proposition~\ref{prop:sdp_transformation} reduces the task of checking the psd condition in Definition~\ref{def:GE_d_defn} to solving an SDP; see, e.g., \cite[Proposition 2]{TV-SDP} for a more explicit representation of this SDP. 

\begin{remark}
We make two remarks regarding computational considerations:
\begin{enumerate}
\itemsep0em
\item The sizes of the matrices $Q_1,Q_2$ in~Proposition~\ref{prop:sdp_transformation} grow only linearly with $d$ (resp. with $n$) for fixed $n$ (resp. fixed $d$). This is in contrast to the SDP hierarchies that arise in a search for convex homogeneous polynomials of degree $d$ in $n$ variables whose sublevel sets can also be considered as a natural generalization of ellipsoids. The size of the semidefinite constraint, even in the first level of this SDP hierarchy (see~\cite{PolyNorms}), grows at the rate $n\left( \substack{n +\frac{d}{2} -2\\\frac{d}{2}-1} \right)$.

\item For implementation purposes, many parsers (e.g., YALMIP~\cite{Yalmip} or SumOfSquares.jl~\cite{Julia}) directly accept sum of squares constraints on a polynomial or polynomial matrix with unknown coefficients and do the conversion to an SDP automatically. Therefore, for our purposes, one can use these parsers to directly work with the representation in Theorem~\ref{thm:univarite_decomp}. The resulting SDPs can be readily solved, e.g., by general-purpose interior point methods. While not the focus of this paper, we suspect implementation improvements may be possible by exploiting the univariate nature of the sum of squares constraints, for example using techniques from~\cite{LofPar,Papp,PappYildiz,LYP,HarrisParrilo}. 
\end{enumerate}
\label{rmk:computational}
\end{remark}

Finally, we note that by searching over polynomial matrices of degree $d$ that satisfy the psd condition in Definition~\ref{def:GE_d_defn}, we are searching over the closure of the set of polynomial matrices that satisfy both conditions in that definition. We can always check the kernel condition in Definition~\ref{def:GE_d_defn} via Lemma~\ref{lem:kernel_cond} as a post-processing step. This type of approach is common in applications of semidefinite and sum of squares programming where the optimization set of interest is not closed.

%that the set of polynomial matrices $P(t)$ of degree $d$ that satisfy the psd condition in Definition~\ref{def:GE_d_defn} but not the kernel condition forms a subset of the boundary of the set of polynomial matrices that satisfy both conditions. 

%%%%%%%%%%%%%%%%%%%%%%%%%%%%%%%%%%%%%%%%%%%%
\section{Conic Representation of GEs} \label{sec:conic_rep}
In this section, we study whether GEs admit a representation as the feasible set of different families of tractable conic programs. This is relevant for applications where one needs to optimize over a GE. For the sake of clarity, we stress that in contrast to Section~\ref{subsec:efficient_search}, where the representation was in the space of coefficients of polynomial matrices, the representation questions that we are concerned with in this section are in $x$ space and relate to a fixed GE-$d$ as defined in \eqref{eq:GE_d_defn}.
The results of Section~\ref{subsec:efficient_search} show that one can use semidefinite programming to search for polynomial matrices which define a GE, whereas the results of this section are concerned with the problem of optimizing over the set of vectors contained in a given GE.

It is clear that one cannot always optimize over a given GE using linear programming. Indeed, since ellipsoids in dimension two or higher have an infinite number of extreme points, already a GE-$0$ fails to be polyhedral. In our next three subsections, we ask if increasingly broader classes of convex sets can represent GEs.

%%%%%%%%%%%%%%%%%%%%%%%%%%%%%%%%%%%%%%%%
\subsection{Can GEs be described by finitely many convex quadratic constraints?}\label{subsec:intersection}

%In this section, we first show that GEs cannot always be represented as the intersection of finitely many ellipsoids. 

Considering the definition of a GE-$d$ in \eqref{eq:GE_d_defn} and the fact that a GE-$0$ is an ellipsoid, a natural question that comes to mind is whether a GE-$d$ can always be described by finitely many convex quadratic constraints.
The next proposition shows that this is the case only when $d=0$ or $d=1$.

% - for every GE-d there exists m s.t. GE-d is intersection m convex quadratic constraints
% - d=0 or d=1
% - every GE-d is intersection <=2 convex quadratic constraints

\begin{proposition}
\label{prop:finite_intersec}
The following are equivalent:
\begin{enumerate} [(i)]
\item for every GE-$d$ $\mathcal{E}_d$, there exists a nonnegative integer $m$ and matrices $P_1,\dots P_m \succeq 0$ such that $$ \mathcal{E}_d = \{ x \in \R^n \mid x^T P_i x \leq 1 \quad i=1,\dots,m\};$$
\item $d \in \{0,1\}$.
\end{enumerate}
% A GE-$d$ is necessarily a finite intersection of ellipsoids if and only if $d=0$ or $d=1$.
\end{proposition}

%\textcolor{red}{technically, a GE-1 can be the intersection of semiellipsoids rather than ellipsoids}

\begin{proof}
(ii) $\Rightarrow$ (i):
A GE-$0$ is precisely an ellipsoid so the claim is established with $m=1$.
We show that a GE-1 can always be described by two convex quadratic constraints.
Let $$\mathcal{E}_1 = \{x \in \R^n \: | \:  \max_{t \in [-1,1]} x^TP(t)x \leq 1\}$$ be a GE-1. Since $P(t)$ is a univariate polynomial matrix of degree 1, for a given $x \in \R^n$, $x^TP(t)x$ is an affine function in $t$, and therefore its maximum over $[-1,1]$ is attained at $t =1$ or at $t =-1$. Hence, 
\begin{equation*}
\hspace{3.3cm}
\begin{aligned}
\mathcal{E}_1 &= \{x \in \R^n \: | \:  \max\limits_{t \in \{-1,1\}} x^TP(t)x \leq 1\} \\
&= \{x \in \R^n \: | \: x^TP(1)x \leq 1\} \cap \{x \in \R^n \: | \: x^TP(-1)x \leq 1\}. 
\end{aligned}
\end{equation*}
This establishes the claim with $m=2$.\footnote{In fact, the intersection of any given two co-centered ellipsoids can be represented by a GE-1. This is a special case of Theorem~\ref{thm:semiellipsoid} in Section~\ref{sec:rep_power}.}

(i) $\Rightarrow$ (ii):
We show that for $d \geq 2$, there are GE-$d$s that cannot be described by finitely many convex quadratic constraints. Consider the set given by $\mathcal{E}_2 = \{x \in \R^2 \: | \:  \max_{t \in [-1,1]} x^TP(t)x \leq 1\}$ with 
$$P(t) = \begin{bmatrix}
     2-t^2 & t \\
     t & 3-t^2
\end{bmatrix}.$$
This set is depicted in Figure~\ref{fig:sub3}. It is easy to verify that $P(t) \succ 0 \:\: \forall t \in [-1,1]$, and therefore $\mathcal{E}_2$ is a valid GE-$2$. For any given $x \neq 0$, the function $x^TP(t)x = -(x_1^2+x_2^2) t^2 + 2x_1x_2t + 2x_1^2 + 3x_2^2$ is a concave quadratic polynomial in $t$ and is maximized at $t^* = \frac{x_1x_2}{x_1^2+x_2^2}$. Observe that $t^* \in [-1,1]$. It is then easy to verify that 
%$$S = \{x \in \R^n \: | \: 2x_1^4 + 6x_1^2x_2^2 + 3x_2^4 - x_1^2 - x_2^2 \leq 0\}.$$
$$\mathcal{E}_2 = \left\{x \in \R^2 \: \middle | \: \frac{x_1^2x_2^2}{x_1^2+x_2^2} + 2x_1^2 + 3 x_2^2 \leq 1 \right\}.$$
Let $h(x) \defin \frac{x_1^2x_2^2}{x_1^2+x_2^2} + 2x_1^2 + 3 x_2^2$ and assume for a contradiction that $\mathcal{E}_2$ can be described by finitely many (convex) quadratic constraints, i.e., $\mathcal{E}_2 = \{x \in \R^2 \: | \:  x^T P_i x \leq 1 \quad i=1,\dots,m\}$ for some (psd) matrices $P_1,\dots, P_m$. Let $g(x) \defin \max\limits_{i=1,\dots,m} x^TP_ix$. Since $g(x)$ and $h(x)$ have the same 1-sublevel set and the same degree of homogeneity, we have $g(x) = h(x)$ for every $x \in \R^2$.
It is straightforward to verify that $h(x_1,1)$ is not piecewise-quadratic while $g(x_1,1)$ is.
This is a contradiction.
\end{proof}

%%%%%%%%%%%%%%%%%%%%%%%%%%%%%%%%%%%%%%
\subsection{Are GEs SOCP-representable?} \label{subsec:socp_rep}

Since second-order order cone programs (SOCPs) are more expressive than convex quadratically constrained programs, it is natural to ask if GEs are SOCP-representable sets. For $n \geq 1$, recall the definition of the $n$-dimensional second-order cone:
$$
L_n \defin \left\{x \in \R^n \:\: \middle | \:\: \sqrt{\sum_{i=1}^{n-1} x_i^2} \leq x_n \right\}.
$$
A set $\Omega \subseteq \R^n$ is \emph{SOCP-representable} if for some nonnegative integers $k,m$, a cone $K \subset \R^m$ which is the product of second-order cones, some matrices $A \in \R^{m \times n}$ and $B \in \R^{m \times k}$, and some vector $b \in \R^m$, one can write
$$
\Omega = \{x \in \R^n \: \mid \: \exists u  \in \R^k \text{ s.t. } Ax + Bu + b \in K\}.
$$
See \cite{BTNem} for a reference on properties of SOCP-representable sets.
Since sets defined by convex quadratic constraints are SOCP-representable, we already know from Proposition~\ref{prop:finite_intersec} that GE-0s and GE-1s are SOCP-representable. The next theorem shows that this is not the case for all GE-$d$s.

\begin{theorem}\label{thm: not socp}
There exists a GE-16 in dimension 9 that is not SOCP-representable.
\end{theorem}
Our proof relies on the following result from \cite{fawzi}.
\begin{theorem}[Corollary~1 of \cite{fawzi}]\label{thm: poly not socp}
The set of nonnegative univariate polynomials of degree (at most) 4 is not SOCP-representable.
\end{theorem}

\begin{proof}[Proof of Theorem~\ref{thm: not socp}]
Let $\phi(t) \defin (1,t,t^2,\dots,t^8)^T$.
Define $P(t) \defin \phi(t)\phi(t)^T$ and 
$$\Omega \defin \{x \in \R^9 \mid x^T P(t) x \leq 1 \quad \forall t\in [-1,1]\}.$$
Observe that $\Omega$ is a valid GE as $P(t)\succeq 0$ for all $t\in[-1,1]$ and the kernel condition is satisfied. To see that the latter claim, suppose for the sake of contradiction that there exists a nonzero vector $y \in\mathbb{R}^9$ such that $P(t)y=0$ for all $t\in[-1,1]$. This would imply that the univariate polynomial $y^T\phi(t)$ vanishes for all $t\in[-1,1]$, which can only happen if all its coefficients are zero, hence contradicting the fact that $y$ was a nonzero vector.

We claim that $\Omega$ is not SOCP-representable. Let $\R_d[t]$ denote the set of polynomials with real coefficients in the variable $t$ of degree at most $d$. We identify a vector $x \in \R^9$ with the degree-$8$ polynomial $p(t) = x^T \phi(t)$.
Using this identification, we can write
% \begin{align*}
% \hspace{4.5cm}
% \Omega &= \{ p \in \R[t]^8 \mid p(t)^2 \leq 1 \quad \forall t\in [-1,1]\}\\
% &= \{ p \in \R[t]^8 \mid |p(t)| \leq 1 \quad \forall t\in [-1,1]\}.
% \end{align*}
$$\Omega = \{ p \in \R_8[t] \mid |p(t)| \leq 1 \quad \forall t\in [-1,1]\}.$$
Suppose for the sake of contradiction that $\Omega$ was SOCP-representable.
Consider the affine map $f:~\R_4[t] \rightarrow \R_8[t]$ which maps a degree-$4$ polynomial $q(t)$ to the degree-$8$ polynomial
$$
p(t) = 2(1-t^2)^4 q\left(\frac{2t}{1-t^2}\right) - 1.
$$
Since taking the inverse image of a set under an affine map preserves SOCP-representability (see \cite[2.3.D]{BTNem}), it follows that the set
$$
f^{-1}(\Omega) = \left\{q \in \R_4[t] \:\: \middle | \:\: \left|2(1-t^2)^4 q\left(\frac{2t}{1-t^2}\right) - 1\right| \leq 1 \quad \forall t\in [-1,1]\right\}
$$
is SOCP-representable.
Note that $t \rightarrow \frac{2t}{1-t^2}$ maps the interval $(-1,1)$ to the entire real line.
Furthermore, this map possesses an inverse, $s \rightarrow \frac{\sqrt{s^2+1}-1}{s}$.
Now, for any polynomial $q \in \R_4[t]$, we have
\begin{align*}
\hspace{4cm}
& \left| 2(1-t^2)^4 q\left(\frac{2t}{1-t^2}\right) - 1 \right| \leq 1 \quad \forall t\in [-1,1]\\
\Leftrightarrow \quad & 0 \leq (1-t^2)^4 q\left(\frac{2t}{1-t^2}\right) \leq 1 \quad \forall t\in [-1,1]\\
\Leftrightarrow \quad & 0 \leq \left(1-\left(\frac{\sqrt{s^2+1}-1}{s}\right)^2\right)^4 q(s) \leq 1 \quad \forall s\in \R \\
\Leftrightarrow \quad & 0 \leq q(s) \leq g(s) \quad \forall s\in \R,
\end{align*}
where $g(s) \defin (1-(\frac{\sqrt{s^2+1}-1}{s})^2)^{-4}$.
Hence,
$$
f^{-1}(\Omega) = \{q \in \R_4[s] \: \mid \: 0 \leq q(s) \leq g(s) \:\: \forall s\in \R\}.
$$
By \cite[Proposition~2.3.1]{BTNem}, it follows that the set
$$\left \{(q,M) \in \R_4[s] \times \R \:\: \middle | \:\: M > 0, \:\: 0 \leq \frac{q(s)}{M} \leq g(s) \:\: \forall s\in \R \right \}$$
is SOCP-representable. Since the map $(q,M) \rightarrow q$ is affine, it follows (see, e.g., \cite[2.3.C]{BTNem}) that the set
$$\{q \in \R_4[s] \: \mid \: \exists M > 0 \text{ s.t. } 0 \leq q(s) \leq M g(s) \:\: \forall s\in \R\}$$
is also SOCP-representable. 

One can check that $g(s) \geq \max(1,\frac{s^4}{16})$ for all $s \in \R$. Therefore, for all $q \in \R_4[s]$, there exists some $M > 0$ such that $q(s) \leq M g(s)$ for all $s\in \R$. Thus, the above set is equal to the set
$$
\{q \in \R_4[s] \: \mid \: q(s) \geq 0 \:\: \forall s\in \R\}.
$$
SOCP-representability of this set contradicts Theorem~\ref{thm: poly not socp}.
\end{proof}

\subsection{Are GEs SDP-representable?} \label{subsec:sdp_representable}

% A set $S$ is a \emph{spectrahedron} (or LMI-representable) if
% $$S = \{x \in \R^n \: | \: A_0 + \sum_{i=1}^{n} x_iA_i \succeq 0\}$$
% for some symmetric $m \times m$ matrices $A_0, A_1, \dots, A_n$.
%(If $S$ has an interior point, $A_0$ can be assumed to be positive definite without loss of generality.)
Since semidefinite programs are more expressive than second-order cone programs, it is natural to ask if GEs are SDP-representable sets. In this section, we answer this question in the affirmative. Recall that a set $\Omega \subseteq \R^n$ is \emph{SDP-representable} if
$$\Omega = \left\{x \in \R^n \: \middle | \: \exists y \in \R^k \textup{ s.t. } A_0 + \sum_{i=1}^{n} x_iA_i + \sum_{i=1}^{k} y_iB_i \succeq 0 \right\}$$
for some integer $k \geq 0$ and symmetric $m \times m$ matrices $A_0, A_1, \dots, A_n, B_1, \dots, B_k$.

% Let $E = \{x \in \R^n \: | \: (x-c)^TQ(x-c) \leq 1\}$ be an ellipsoid where $Q \in S_{++}^n$. A point $x$ is in $E$ if and only if it satisfies the LMI
% $$\begin{bmatrix} Q^{-1} & x-c \\ (x-c)^T & 1 \end{bmatrix} = \begin{bmatrix} Q^{-1} & -c \\ -c^T & 1 \end{bmatrix} + \sum_{i=1}^{n} x_i \begin{bmatrix} 0 & e_i \\ e_i^T & 0 \end{bmatrix} \succeq 0.$$
% Therefore, ellipsoids are spectrahedra, and thus SDP-representable. In this section, we show that GEs are SDP-representable. This means that optimizing over GEs can be done efficiently.

\begin{theorem}
\label{thm:sdp_representable}
Every GE is an SDP-representable set.
\end{theorem}

\begin{proof}
We provide two different SDP representations, one which explicitly uses the decomposition in Theorem~\ref{thm:univarite_decomp}, and one which implicitly uses the psd condition of Definition~\ref{def:GE_d_defn}. The second representation will typically be of smaller size, but the first one can be smaller when the polynomial matrix $P(t)$ has a low-rank decomposition. See Remark~\ref{rem:size_comparison} for a more precise comparison.

\vspace{0.2cm}

\noindent \emph{SDP representation 1:} Let $\mathcal{E}_d \subset \R^n$ be a GE-$d$ defined by the $n\times n$ polynomial matrix $P(t)$ of degree $d$. Since $P(t)\succeq 0$ for all $t\in[-1,1]$, recalling the definition of an sos-matrix, by Theorem~\ref{thm:univarite_decomp} we can write:
$$
P(t) = B(t)^TB(t) + (1-t^2) C(t)^TC(t)
$$
if $d$ is even and
$$
P(t) = (t+1)B(t)^TB(t) + (1-t) C(t)^TC(t)
$$
if $d$ is odd, where $B(t)$ and $C(t)$ are univariate polynomial matrices of respective sizes $r_1 \times n$ and $r_2 \times n$.
As shown in \cite[Theorem~7.1]{ChoiLamRez80}, there always exists a decomposition such that $r_1,r_2 \leq 2n$.
The degrees of $B(t)$ and $C(t)$ are respectively $\frac{d}{2}$ and $\frac{d}{2} - 1$ if $d$ is even, or both equal to $\frac{d-1}{2}$ if $d$ is odd.

Observe that by taking Schur complements, for any $x\in \R^n$ and $t \in [-1,1]$, we have
$$
x^T P(t) x \leq 1 \Leftrightarrow M_x(t) \defin \begin{bmatrix}
    I_{r_1} & 0 & B(t) x \\
    0 & (1-t^2) I_{r_2} & (1-t^2) C(t) x \\
    x^T B(t)^T & (1- t^2) x^T C(t)^T & 1
\end{bmatrix} \succeq 0
$$
for $d$ even, or 
$$
x^T P(t) x \leq 1 \Leftrightarrow M_x(t) \defin \begin{bmatrix}
    (1+t)I_{r_1} & 0 & (1+t)B(t) x \\
    0 & (1-t) I_{r_2} & (1-t) C(t) x \\
    (1+t)x^T B(t)^T & (1- t) x^T C(t)^T & 1
\end{bmatrix} \succeq 0
$$
for $d$ odd.
Then, we can write
$$
\mathcal{E}_d = \{x \in \R^n \: | \: M_x(t) \succeq 0 \:\: \forall t \in [-1,1] \}.
$$
Observe that in both cases, $M_x(t)$ is a univariate polynomial matrix whose coefficients depend affinely on $x$.
Thus, in view of Proposition~\ref{prop:sdp_transformation}, the constraint that $M_x(t) \succeq 0 \:\: \forall t \in [-1,1]$ can be reduced to affine and semidefinite constraints on $x$.
It follows that $\mathcal{E}_d$ is an SDP-representable set.
% Using Theorem~\ref{thm:univarite_decomp} again, we can write:
% $$
% \mathcal{E}_d = \left \{ x \in \R^n \middle| \begin{matrix}
%     D(t),E(t) \quad (4n+1) \times (4n+1) \\
%     D(t),E(t) \quad \text{degree} \leq \lceil \frac{d}{2} \rceil + 2 \\
%     D(t),E(t) \quad \text{sos} \\
%     M(t,x) = D(t) + (1-t^2) E(t) \quad \forall t
% \end{matrix} \right \}
% $$
% if the degree of $M$ is even or if it is odd:
% $$
% \mathcal{E}_d = \left \{ x \in \R^n  \middle| \begin{matrix}
%     D(t),E(t) \quad (4n+1) \times (4n+1) \\
%     D(t),E(t) \quad \text{degree} \leq \lceil \frac{d}{2} \rceil + 2 \\
%     D(t),E(t) \quad \text{sos} \\
%     M(t,x) = (1+t)D(t) + (1-t) E(t) \quad \forall t
% \end{matrix} \right \}
% $$
% which is evidently an SDP-representable set.

\vspace{0.2cm}

\noindent \emph{SDP representation 2:} Let $\mathcal{E}_d \subset \R^n$ be a GE-$d$ defined by the $n\times n$ polynomial matrix $P(t)$ of degree $d$. We claim that $\mathcal{E}_d = \bar{\mathcal{E}}_d$, where 
$$
\bar{\mathcal{E}}_d \defin \left \{ x \in \R^n \:\: \middle | \:\: \exists X \in S^n \text{ s.t. } \begin{bmatrix} X & x \\ x^T & 1\end{bmatrix} \succeq 0, \quad \text{Tr}(X P(t)) \leq 1 \:\:\: \forall t \in [-1,1] \right\}.
$$
The set $\bar{\mathcal{E}}_d$ is clearly SDP-representable since the constraint that $\text{Tr}(X P(t)) \leq 1$ for all $ t \in [-1,1]$ can be reformulated as an SDP constraint in view of Proposition~\ref{prop:sdp_transformation} (applied to the scalar-valued polynomial $1 - \text{Tr}(X P(t))$).

To see that $\mathcal{E}_d = \bar{\mathcal{E}}_d$, first observe that any $x \in \mathcal{E}_d$ also belongs to $\bar{\mathcal{E}}_d$ since the vector $x$ and the matrix $X = xx^T$ satisfy the constraints of $\bar{\mathcal{E}}_d$.
Conversely, for any $x \in \bar{\mathcal{E}}_d$, fix a matrix $X \in S^n$ such that $\begin{bmatrix} X & x \\ x^T & 1\end{bmatrix} \succeq 0$ and $\text{Tr}(X P(t)) \leq 1$ for all $t \in [-1,1]$.
By taking a Schur complement, the first constraint implies that $X \succeq xx^T$.
For all $t \in [-1,1]$, since $P(t)\succeq 0$, we have
$$
x^T P(t)x = \text{Tr}(xx^T P(t))
\leq \text{Tr}(X P(t)) \leq 1,
$$
and therefore $x \in \mathcal{E}_d$.
\end{proof}

%\textcolor{red}{[*Add remark. To do about $mxm$ in 4 cases. And why that's a good thing.* ]}

\begin{remark}\label{rem:size_comparison}
Following the proof of Theorem~\ref{thm:sdp_representable} and Proposition~\ref{prop:sdp_transformation}, one can check that the size of the semidefinite constraints necessary to represent a GE-$d$ in dimension $n$ via the first construction in the proof of Theorem~\ref{thm:sdp_representable} is 
$$
\begin{cases}
(\frac{d}{4}+1)(r_1+r_2+1) & d \equiv 0\mod 4\\
(\frac{d+3}{4})(r_1+r_2+1) & d \equiv 1\mod 4\\
(\frac{d+2}{4}+1)(r_1+r_2+1) & d \equiv 2\mod 4\\
(\frac{d+1}{4}+1)(r_1+r_2+1) & d \equiv 3\mod 4.
\end{cases}
$$
Given that by \cite[Theorem~7.1]{ChoiLamRez80} we can always take $r_1,r_2 \leq 2n$, the size of the semidefinite constraints grows only linearly with $d$ (resp. with $n$) for fixed $n$ (resp. fixed $d$). 

Similarly, one can check that the size of the semidefinite constraints necessary to represent a GE-$d$ in dimension $n$ via the second construction in the proof of Theorem~\ref{thm:sdp_representable} is 
$$
\begin{cases}
\max\left \{n + 1, \frac{d}{2} + 1\right \} & d \text{ even}\\
\max\left \{n + 1, \frac{d+1}{2} \right\} & d \text{ odd}.
\end{cases}
$$
Thus, the size of the semidefinite constraints grows only linearly with $\max\{n, d\}$.

By contrast, to our knowledge, sublevel sets of $n$-variate homogeneous convex polynomials of degree $d$ (which would also serve as a natural generalization ellipsoids) are not known to have a semidefinite representation~\cite{HeltonNie}.
% If one instead works with the stronger requirement of ``sos-convexity'', the sublevel set would have an exact semidefinite representation (see, e.g.,~\cite{Lasserre2}), but the size of the semidefinite constraint to represent the sublevel set of an sos-convex homogeneous polynomial of degree $d$ in $n$ variables would be $ \left( \substack{n +\frac{d}{2}\\\frac{d}{2}} \right)$.
If one instead works with ``sos-convex'' polynomials (a stronger condition), then the sublevel sets do have a semidefinite representation (see, e.g.,~\cite{HeltonNie,Lasserre2}), but the size of the semidefinite constraint would be $ \left( \substack{n +\frac{d}{2}\\\frac{d}{2}} \right)$.
\end{remark}

% \textcolor{red}{Options to get B and C from P(t):
% \begin{itemize}
% \item Transformation to global case: either strongly poly proof or SDP.
% \item Factorization: either following the construction in the proof of * in Reznick or using the algorithm of Parrilo (in the monic case). Or simply by a factorization of the psd matrix Q as Q=LL' and v(t,y)'*L'*L*v(t,y) --> B,C. Indeed, dimension would be rank of Q.
% \end{itemize}
% }

\begin{remark}
Note that the second SDP representation in the proof of Theorem~\ref{thm:sdp_representable} is directly in terms of the polynomial matrix $P(t)$.
The first SDP representation, however, is in terms of polynomial matrices $B(t)$ and $C(t)$ that appear in the decomposition of the polynomial matrix $P(t)$. In some situations (e.g., the application in Section~\ref{subsec:shift regression}, the construction in the proof of Theorem~\ref{thm: not socp}, or when dealing with factor models), the matrix $P(t)$ appears already in decomposed form. In situations where this is not the case, there are multiple ways of obtaining polynomial matrices $B(t)$ and $C(t)$ from the polynomial matrix $P(t)$.

One option is to solve the SDP in Proposition~\ref{prop:sdp_transformation} to find psd matrices $Q_1$ and $Q_2$ in a sum of squares decomposition of $y^TX_i(t)y$, for $i=1,2$, associated with the sos-matrices $X_1(t),X_2(t)$ that appear in Theorem~\ref{thm:univarite_decomp}.
By performing a matrix decomposition of the form $Q_i=L_i^TL_i$, $i=1,2$, we can readily get an expression for $B(t)$ and $C(t)$; see, e.g.,~\cite[Lemma 1]{SchererHol}. We note that there are alternative ways of factoring univariate sos-matrices; e.g., by using the algorithm proposed in~\cite{aylward2007explicit}, or by following the proof of~\cite[Theorem~7.1]{ChoiLamRez80}.

Another option is to directly find a decomposition of the polynomial matrix $Q(t) \defin (t^2+1)^d P(\frac{t^2-1}{t^2+1})$ as $R(t)^T R(t)$ for a polynomial matrix $R(t)$.
Such a decomposition must exist (since $Q(t) \succeq 0$ for all $t \in \R$) and can be found by methods mentioned in the previous paragraph.
One can then convert the decomposition of $Q(t)$ into a suitable decomposition of $P(t)$, e.g., by following the proof of~\cite[Theorem~6.11]{papp2013semidefinite}.
\end{remark}

%%%%%%%%%%%%%%%%%%

\subsubsection{Distance between two GEs}\label{subsec:distance_comp}

As an application of Theorem~\ref{thm:sdp_representable}, we show here how one can compute the distance between two GEs. The problem of computing distances between two convex sets has applications in many areas, for example robotics, computer-aided design, and computer graphics~\cite{GJK, AHMS}.
% determining the spatial relationship between two sets of objects is crucial. Whether for collision detection, path planning, or modeling interactions, knowing the precise distance between objects provides valuable insights. The Euclidean distance, defined as the shortest line segment connecting two points from each set, serves as the most natural and intuitive measure of proximity. Accurately computing this distance is fundamental for ensuring efficient and reliable performance in these applications. In this subsection, we consider the problem of finding the distance between two GEs. Using results from \Cref{subsec:sdp_representable} (more specifically, the representation given in the proof of \Cref{thm:sdp_representable}), we formulate this problem as a semidefinite program. 

As a concrete example, consider the following two GEs 
$$\mathcal{E}^y = \{x \in \R^n \: | \:  (x-c_y)^TP_y(t)(x-c_y) \leq 1 \:\: \forall t \in [-1,1]\},$$ 
$$\mathcal{E}^z = \{x \in \R^n \: | \:  (x-c_z)^TP_z(t)(x-c_z) \leq 1 \:\: \forall t \in [-1,1]\},$$ 
where
$$P_y(t) = \begin{bmatrix} 6-4t & 2-3t & 1-t \\ 2-3t & 3-t & -t \\ 1-t & -t & 2 \end{bmatrix}, 
\hspace{0.5cm} 
P_z(t) = \begin{bmatrix} 2-t^2 & t & 0 \\ t & 3-t^2 & 0 \\ 0 & 0 & 1 \end{bmatrix},
\hspace{0.5cm} 
c_y = \begin{bmatrix} 0 \\ 0 \\ 0 \end{bmatrix},
\hspace{0.5cm}
c_z = \begin{bmatrix} 1 \\ -1 \\ 1 \end{bmatrix}.$$
% \begin{figure}[ht]
% \centering
% \begin{subfigure}{.34\textwidth}
%   \centering
%   \includegraphics[width=.9\linewidth]{images/example_3d.png}
%   \label{fig:sub3}
% \end{subfigure}
% \begin{subfigure}{.34\textwidth}
%   \centering
%   \includegraphics[width=.9\linewidth]{images/example_3d_2.png}
%   \label{fig:sub4}
% \end{subfigure}
% \vspace{-0.7cm}
% \caption{The GEs $S_1$ and $S_2$}
% \label{fig:GE_examples}
% \end{figure}
The distance between $\mathcal{E}^y$ and $\mathcal{E}^z$ is equal to
% \begin{equation}
% \hspace{4cm}
% \label{eq:distance_opt}
% \begin{aligned}
% &\min_{x_1,x_2} &&||x_1-x_2||_2\\
% &\text{s.t.} && x_i \in S_i \quad i=1,2.
% \end{aligned}
% \end{equation}
% \begin{equation}
% \hspace{4cm}
% \label{eq:distance_opt2}
% \begin{aligned}
% &\min_{x_1,x_2} &&||x_1-x_2||_2\\
% &\text{s.t.} && (x_i-c_i)^TP_i(t)(x_i-c_i) \leq 1 \:\: \forall t \in [-1,1] \quad i=1,2.
% \end{aligned}
% \end{equation}
\vspace{-0.5cm}
\begin{multicols}{3}
\begin{equation*}
\hspace{-0.2cm}
\begin{aligned}
& \min_{y,z} &&||y-z||_2 \\
& \text{s.t.} && y \in \mathcal{E}^y \\
&&& z \in \mathcal{E}^z
\end{aligned}
\end{equation*}
\hspace{0.8cm}
\begin{equation*}
\hspace{-1.5cm}
\begin{aligned}
& \quad = \quad \\
&
\end{aligned}
\end{equation*}
\hspace{0.8cm}
\begin{equation*}
\hspace{-5.2cm}
\begin{aligned}
& \min_{y,z} &&||y-z||_2\\
& \text{s.t.} && (y-c_y)^T P_y(t)(y-c_y) \leq 1 \:\: \forall t \in [-1,1] \\
&&& (z-c_z)^T P_z(t)(z-c_z) \leq 1 \:\: \forall t \in [-1,1].
\end{aligned}
\end{equation*}
\end{multicols}

\noindent As $\mathcal{E}^y$ is a GE-1, from the proof of Proposition~\ref{prop:finite_intersec}, we have $y+c_y \in \mathcal{E}^y$ if and only if
\begin{equation}\label{eq: y constraint}
\hspace{3.1cm}    
y^T \begin{bmatrix}
    10&5&2\\
       5&4&1\\
       2&1&2
\end{bmatrix} y  \leq 1 \quad \text{and} \quad
y^T \begin{bmatrix}
    2&-1& 0\\
      -1& 2&-1\\
       0&-1& 2\\
\end{bmatrix} y
 \leq 1.
\end{equation}

\noindent Since $P_z(t) \succeq 0 $ for all $t\in[-1,1]$ and its degree is even, by Theorem~\ref{thm:univarite_decomp}, we can find a representation of the form
$$
P_z(t) = B(t)^TB(t) + (1-t^2) C(t)^TC(t).
$$
We can take for example
$$
B(t) = \begin{bmatrix}
    \frac{1}{\sqrt{2}} & 0 & 0 \\
    0 & 0 & -1 \\
    \frac{t}{\sqrt{2}} & \sqrt{2} & 0
\end{bmatrix} \quad \text{and} \quad 
C(t) =  \begin{bmatrix}
    0 & 1 & 0 \\
    \sqrt{\frac{3}{2}} & 0 & 0
\end{bmatrix}.
$$

% $$
% P_z(t) = \begin{bmatrix}
%     \frac{1}{\sqrt{2}} & 0 & \frac{t}{\sqrt{2}} \\
%     0 & 0 & \sqrt{2} \\
%     0 & -1 & 0
% \end{bmatrix}
% \begin{bmatrix}
%     \frac{1}{\sqrt{2}} & 0 & \frac{t}{\sqrt{2}} \\
%     0 & 0 & \sqrt{2} \\
%     0 & -1 & 0
% \end{bmatrix}^T
% + (1-t^2)\begin{bmatrix}
%     0 & \sqrt{\frac{3}{2}} \\
%     1 & 0 \\
%     0 & 0
% \end{bmatrix}
% \begin{bmatrix}
%     0 & \sqrt{\frac{3}{2}} \\
%     1 & 0 \\
%     0 & 0
% \end{bmatrix}^T.
% $$

\noindent Following the first construction in the proof of Theorem~\ref{thm:sdp_representable}, we have that $z+c_z \in \mathcal{E}^z$ if and only if
\begin{equation}\label{eq: z constraint}
\begin{bmatrix}
    I_{3} & 0 & \begin{bmatrix}
    \frac{1}{\sqrt{2}} & 0 & 0 \\
    0 & 0 & -1 \\
    \frac{t}{\sqrt{2}} & \sqrt{2} & 0
\end{bmatrix} z \\
    0 & (1-t^2) I_{2} & (1-t^2) \begin{bmatrix}
    0 & 1 & 0 \\
    \sqrt{\frac{3}{2}} & 0 & 0
\end{bmatrix} z \\
    z^T \begin{bmatrix}
    \frac{1}{\sqrt{2}} & 0 & 0 \\
    0 & 0 & -1 \\
    \frac{t}{\sqrt{2}} & \sqrt{2} & 0
\end{bmatrix}^T & (1- t^2) z^T \begin{bmatrix}
    0 & 1 & 0 \\
    \sqrt{\frac{3}{2}} & 0 & 0
\end{bmatrix}^T & 1
\end{bmatrix} \succeq 0 \quad \forall t\in [-1,1].
\end{equation}

\noindent Thus, the problem of computing the distance between $\mathcal{E}_y$ and $\mathcal{E}_z$ can be reformulated as 
\begin{equation}\label{eq: distance 1}
\begin{aligned}
\hspace{5.4cm}
\min_{y,z} \quad & \|(y+c_y)-(z+c_z) \| \\
\text{s.t.} \quad & \eqref{eq: y constraint}, \eqref{eq: z constraint}.
\end{aligned}
\end{equation}

Alternatively, by the second construction in the proof of Theorem~\ref{thm:sdp_representable}, we have that $z+c_z \in \mathcal{E}^z$ if and only if there is a matrix $Z \in S^3$ such that 
\begin{equation}\label{eq:Z_constraints}
\begin{aligned}
\hspace{4cm}
\begin{bmatrix}Z & z \\ z^T & 1 \end{bmatrix} \succeq 0, \quad \quad \quad \text{Tr}(Z P_z(t)) \leq 1 \:\:\: \forall t \in [-1,1].
\end{aligned}
\end{equation}

\noindent Thus, the problem of computing the distance between $\mathcal{E}_y$ and $\mathcal{E}_z$ can be reformulated as
\begin{equation}\label{eq: distance 2}
\begin{aligned}
\hspace{5.4cm}
\min_{y,z,Z} \quad & \|(y+c_y)-(z+c_z) \| \\
\text{s.t.} \quad & \eqref{eq: y constraint}, \eqref{eq:Z_constraints}. 
%\\
%& \begin{bmatrix}Z & z \\ z^T & 1 \end{bmatrix} \succeq 0 \\
%& \text{Tr}(Z P_z(t)) \leq 1 \quad \forall t \in [-1,1].
\end{aligned}
\end{equation}

In view of Proposition~\ref{prop:sdp_transformation}, both \eqref{eq: distance 1} and \eqref{eq: distance 2} are semidefinite programming problems.
By solving either numerically, we find the distance between $\mathcal{E}_y$ and $\mathcal{E}_z$ to be 0.4635 to four digits of accuracy.
The two GEs and a line segment connecting points in each set that attain this minimum distance are plotted in Figure~\ref{fig:distance_fig}.

% $\min_{y,z} \|(y+c_y)-(z+c_z) \|$ subject to \eqref{eq: y constraint} and \eqref{eq: z constraint}.
% \begin{align*}
% \hspace{0.5cm}
% \min_{y,z} \quad & \|(y+c_y)-(z+c_z) \| \\
% \text{s.t.} \quad & y^T \begin{bmatrix}
%     10&5&2\\
%        5&4&1\\
%        2&1&2
% \end{bmatrix} y \leq 1 \\
% & y^T \begin{bmatrix}
%     2&-1& 0\\
%       -1& 2&-1\\
%        0&-1& 2\\
% \end{bmatrix} y \leq 1 \\
% & \begin{bmatrix}
%     I_{3} & 0 & \begin{bmatrix}
%     \frac{1}{\sqrt{2}} & 0 & \frac{t}{\sqrt{2}} \\
%     0 & 0 & \sqrt{2} \\
%     0 & -1 & 0
% \end{bmatrix} z \\
%     0 & (1-t^2) I_{2} & (1-t^2) \begin{bmatrix}
%     0 & \sqrt{\frac{3}{2}} \\
%     1 & 0 \\
%     0 & 0
% \end{bmatrix} z \\
%     z^T \begin{bmatrix}
%     \frac{1}{\sqrt{2}} & 0 & \frac{t}{\sqrt{2}} \\
%     0 & 0 & \sqrt{2} \\
%     0 & -1 & 0
% \end{bmatrix}^T & (1- t^2) z^T \begin{bmatrix}
%     0 & \sqrt{\frac{3}{2}} \\
%     1 & 0 \\
%     0 & 0
% \end{bmatrix}^T & 1
% \end{bmatrix} \succeq 0 \quad \forall t\in [-1,1].
% \end{align*}

\begin{figure}[ht]
\centering
\includegraphics[width=.5\linewidth]{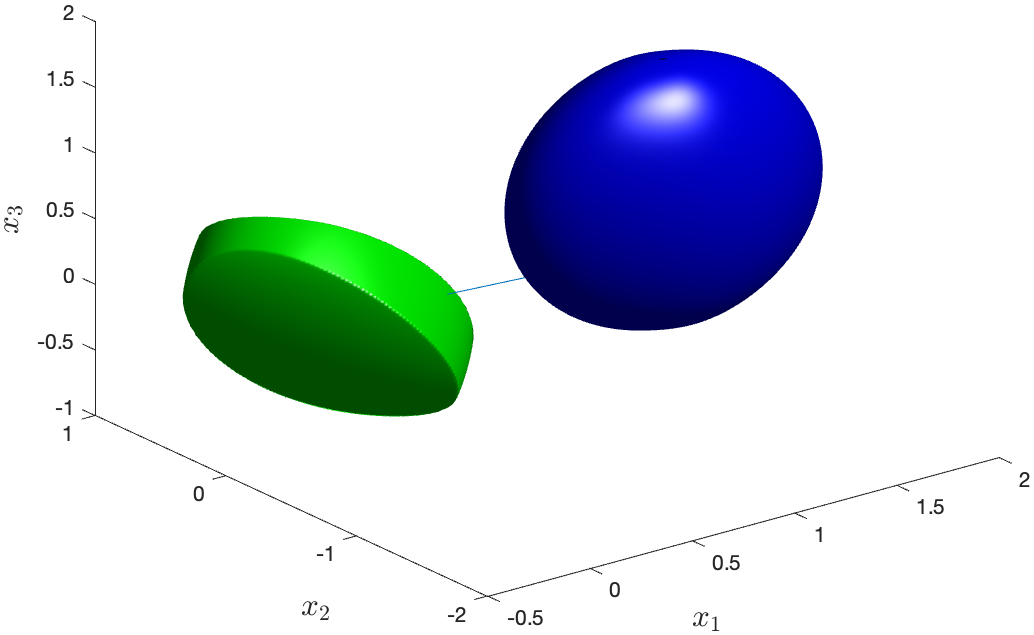}
\caption{The distance between the GEs $\mathcal{E}^y$ and $\mathcal{E}^z$ in Section~\ref{subsec:distance_comp}.}
\label{fig:distance_fig}
\end{figure}

%%%%%%%%%%%%%%%%%%%%%%%%%%%%%%
\section{How Expressive are GEs?}
\label{sec:rep_power}

We have already seen that GEs can express some nontrivial convex sets; e.g., the set of univariate polynomials $p:\R \rightarrow \R$ such that $|p(t)|\leq 1$ for $t\in[-1,1]$ (see the proof of Theorem~\ref{thm: not socp}). In this section, we show that every symmetric
%\footnote{A set $\Omega \subseteq \mathbb{R}^n$ is symmetric if (after possible translation) we have $x \in \Omega$ if and only if $-x \in \Omega$ for every point $x \in \R^n$.} 
full-dimensional polytope and every finite intersection of co-centered ellipsoids can be represented exactly as a GE (Corollary~\ref{cor:ellipsoids and polytopes}, following from Theorem~\ref{thm:semiellipsoid}). We then use this result to show that every convex body
%\footnote{Recall that a convex body is a compact convex set with nonempty interior.} 
can be approximated arbitrarily well by a GE. We also quantify the quality of the approximation as a function of the degree of the GE (Theorem~\ref{thm:approximation_by_GE}).
We end this section by describing how GEs can behave under polar duality.

Let us begin with a definition. We call a set $T \subseteq \R^n$ a \emph{semiellipsoid} if it can be written as $T = \{x \in \R^n \mid x^T P x \leq 1\}$ for some positive semidefinite matrix $P$. Note that a compact semiellipsoid is always a GE-0. 

\begin{theorem}
\label{thm:semiellipsoid}
For every integer $m \geq 2$, a compact intersection of $m$ semiellipsoids is a GE-$d$ with $d \leq 2m-3$.
\end{theorem}
The proof of Theorem~\ref{thm:semiellipsoid} relies on the following lemma, which could be of independent interest. The lemma states that in any dimension $m$, there exists a polynomial curve that stays within the unit simplex in $\mathbb{R}^m$ and visits every corner.

\begin{lemma}[Polynomial tour of the simplex]
\label{lem:poly simplex}
For every integer $m \geq 2$, there exist univariate polynomials $p_1,\dots,p_m$ of degree at most $2m-3$ satisfying
\begin{enumerate}
    \item $p_i(t) \geq 0 \quad \forall t \in [-1,1], \quad i=1,\dots,m$,
    \item $\sum_{i=1}^m p_i(t) = 1 \quad \forall t \in [-1,1]$, and
    \item for every $i = 1,\dots,m, \exists t_{i} \in [-1,1]$ such that $p_i(t_{i}) = 1$.
\end{enumerate}
\end{lemma}
The proof of this lemma utilizes a nonconstructive argument inspired by game theory and may be of independent interest as a proof technique. With some additional arguments, we can also make the proof of this lemma constructive (see Remark~\ref{rem:explicit}). We note that all but two of the $m$ polynomials can be taken to have degree at most $2m-4$, but the degree bound of $2m-3$ in Lemma~\ref{lem:poly simplex} is the lowest possible (see Lemma~\ref{lem:min_degree}). We first recall the following result from game theory.

\begin{theorem}[Debreu, Glicksberg, Fan; see, e.g., \cite{debreu}]\label{thm:nash}
Consider a game with $N$ players indexed by $i=1,\dots,N$.
Suppose that for each $i$, player $i$ chooses an action $a_i$ from a nonempty, compact, and convex set $A_i \subseteq \R^M$ and receives a payoff of $u_i(a_1,\dots,a_N)$.
If for each $i$, the function $u_i$ is continuous in $a_1,\dots,a_N$ and quasiconcave in $a_i$ (over $A_i$), then the game possesses a pure-strategy Nash equilibrium; i.e., there exist actions $\bar{a}_1 \in A_1,\dots, \bar{a}_N \in A_N$, %for $i = 1,\dots,N,$ 
such that for every $i=1,\dots,N$, we have
\begin{equation}\label{eq: nash}
\hspace{4.8cm}
u_i(\bar{a}_1,\dots,\bar{a}_N) = \max_{a_i\in A_i} u_i(a_i,\bar{a}_{-i}),
\end{equation}
where the index $-i$ represents all players besides player $i$.
\end{theorem}

\begin{proof} [Proof of Lemma~\ref{lem:poly simplex}]
We set up a game whose pure-strategy Nash equilibria correspond to roots of polynomials that satisfy the three conditions in the lemma. Consider the following game with players indexed by $i=1,\dots,m$ and action sets $A_i \subseteq [-1,1]$.
Players $1$ and $m$ must respectively play $-1$ and $1$ (i.e., $A_1 = \{-1\}$ and $A_m = \{1\}$), and their payoffs are taken to be $0$.
For $i=2,\dots,m-1$, player $i$ chooses an action $a_i \in [-1,1]$ and receives a payoff of
% $$
% u_i(a) = -f(a_{i-1}-a_i) - f(a_i - a_{i+1}) +f(a_i-a_{i-1}) f(a_{i+1}-a_i) (1-a_i^2) \prod_{j\not\in \{1,i-1,i,i+1,m\}} (a_i-a_j)^2 ,
% $$
\begin{equation}\label{eq: u_i def}
\hspace{0.8cm}
u_i(a) = -f(a_{i-1}-a_i) - f(a_i - a_{i+1}) + (1-a_i^2) \prod_{1<j<i} f(a_i-a_j) \prod_{i<j<m} f(a_j-a_i),
\end{equation}
where $a \defin (a_1,\dots,a_m) $ and the function $f$ is defined as
$$
f(x) = \begin{cases} 
      0 & x< 0 \\
      x^2 & x\geq 0 
   \end{cases}.
$$ 
The first two terms in the payoff function incentivize player $i$ to take an action $a_i$ that belongs to the interval $(a_{i-1},a_{i+1})$.
The third term essentially incentivizes player $i$ to increase the geometric mean of the deviations between her action and the actions of the others. This third term will soon be used when we construct the polynomials desired by the lemma. As an example, a plot of $u_3(a_3,a_{-3})$ is shown in Figure~\ref{fig:simplex_plot} for the case $m=5$ and where $a_{-3}$ corresponds to the actions of players $1,2,4,5$ at a pure-strategy Nash equilibrium.
\begin{figure}[ht]
\centering
\includegraphics[width=.45\linewidth]{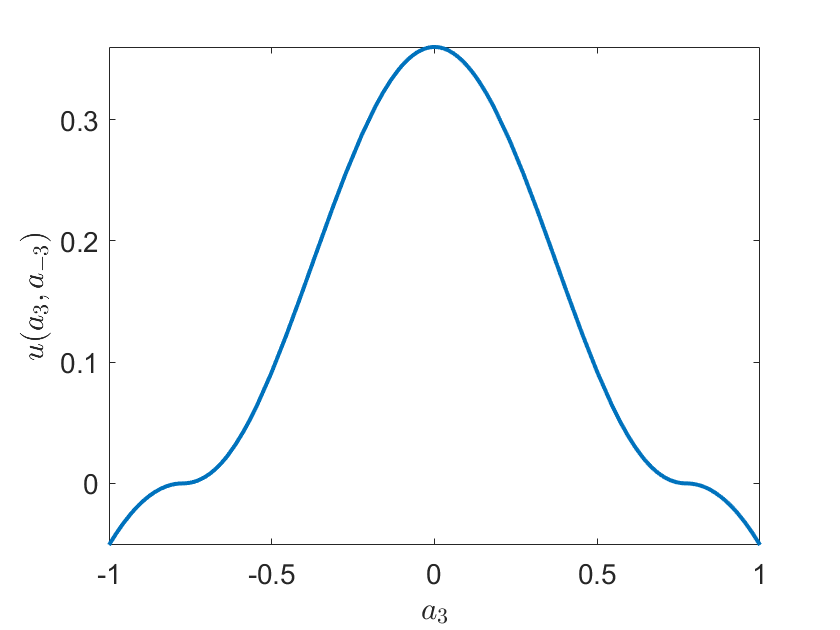}
\caption{The payoff function $u_3(a_3,a_{-3})$ of the third player for the case $m=5$ and where $a_{-3} = (a_1,a_2,a_4,a_5)= (-1, -\sqrt{0.6},\sqrt{0.6},1)$ corresponds to the actions of the other players at a pure-strategy Nash equilibrium for the game appearing in the proof of Lemma~\ref{lem:poly simplex}.}
\label{fig:simplex_plot}
\end{figure}

The functions $u_i$ are continuous in $a$, since the function $f$ is continuous.
% Therefore, the third summand in $u_i$ is a product of continuous functions.
% Thus, $u_i$ is a continuous function of $a$.
% When $a_i \notin [a_{i-1}, a_{i+1}]$, $u_i$ is equal to the constant $0$ and is therefore continuous.
% As $a_i$ approaches $a_{i-1}$ or $a_{i+1}$, from either side, the function tends to $0$.
% Thus the functions are continuous in $a$.
% [\textcolor{red}{*New argument to go in.*}]
Next, we claim that for each $i=2,\dots,m-1$, the function $u_i$ is quasiconcave in $a_i$. Fix some $i$ and $a_{-i}$, and consider $u_i$ as a function of only $a_i$. Consider three cases:
\begin{itemize}
\item $a_{i-1} \geq a_{i+1}$: In this case, the second derivative of $u_i$ with respect to $a_i$ is $-4$ in the interval $(a_{i+1},a_{i-1})$ and $-2$ outside of the interval $[a_{i+1},a_{i-1}]$.
Thus, $u_i$ is strictly concave with respect to $a_i$.

\item $a_{i-1} < a_{i+1}$ and $\max\limits_{j<i} a_j \geq \min\limits_{j>i} a_j$: In this case, $u_i$ is equal to $-f(a_{i-1}-a_i) - f(a_i - a_{i+1})$.
Given that $-f$ is concave, it follows that $u_i$ is concave.

\item $a_{i-1} < a_{i+1}$ and $\max\limits_{j<i} a_j < \min\limits_{j>i} a_j$: In this case, we first claim that $u_i$ is quasiconcave when $a_i \in [\max\limits_{j<i} a_j,\min\limits_{j>i} a_j]$ where $u_i$ is equal to $ (1-a_i^2) \prod_{j\not\in \{1,i,m\}} (a_i-a_j)^2 $.
Since this polynomial is real-rooted, by interlacing, it must have a root between any two roots of its derivative.
% As the polynomial has no root in the interval $(\max\limits_{j<i} a_j,\min\limits_{j>i} a_j)$, over the same interval, its derivative has at most one root and can change sign at most once.
Thus, in this interval, $u_i$ is either quasiconvex or quasiconcave in $a_i$ (see, e.g.,~\cite[Section~3.4.2]{boyd_book}).
Since $u_i$ vanishes at $\max\limits_{j<i} a_j$ and $\min\limits_{j>i} a_j$ and is positive in between, it must be quasiconcave over the interval $[\max\limits_{j<i} a_j,\min\limits_{j>i} a_j]$.
To extend the quasiconcavity argument to all of $[-1,1]$, observe that $u_i$ vanishes in the intervals $[a_{i-1},\max\limits_{j<i} a_j]$ and $[\min\limits_{j>i} a_j,a_{i+1}]$.
As $u_i$ is nonpositive outside of the interval $(a_{i-1},a_{i+1})$, our previous argument implies that for any $\alpha > 0$, the $\alpha$-superlevel set of $u_i$ is convex. In addition, for any $\alpha\leq 0$, we have $\{a_i \in \R \mid u_i(a_i,a_{-i}) \geq \alpha \} = [a_{i-1} - \sqrt{-\alpha},a_{i+1} + \sqrt{-\alpha}]$.
Thus, $u_i$ is quasiconcave in $a_i$.
\end{itemize}

\noindent By Theorem~\ref{thm:nash} (applied with $N=m,M=1$), there exist actions $\bar{a}_1,\dots,\bar{a}_{m} \in [-1,1]$, such that for every $i=1,\dots,m$ we have
\begin{equation}
\hspace{4.5cm}
u_i(\bar{a}_1,\dots,\bar{a}_{m}) = \max_{a_i\in [-1,1]} u_i(a_i,\bar{a}_{-i}).
\label{eq:we_got_Nash}
\end{equation}

%\textcolor{red}{declare what is left to do and why we need them}

%We first make some observations on how the Nash equilibrium conditions affect the relative actions of players.

\noindent Fix such actions $\bar{a} \defin (\bar{a}_1,\dots,\bar{a}_{m})$. We next claim that $\bar{a}_1 < \bar{a}_2 < \dots < \bar{a}_m$.
First observe that by the definition of the payoff function $u_i$ in \eqref{eq: u_i def} and in view of~\eqref{eq:we_got_Nash}, for every $i=2,\dots,m-1$, we have
\begin{itemize}
\itemsep0em
% \item if $\bar{a}_{i-1} = \bar{a}_{i+1}$, then $\bar{a}_i = \bar{a}_{i-1}=\bar{a}_{i+1}$,
\item if $\bar{a}_{i-1} > \bar{a}_{i+1}$, then $\bar{a}_i \in (\bar{a}_{i+1},\bar{a}_{i-1})$,
\item if $\bar{a}_{i-1} \leq \bar{a}_{i+1}$, then $\bar{a}_i \in [\bar{a}_{i-1}, \bar{a}_{i+1}]$,
\begin{itemize}
\item[$\circ$] if we further have $\max\limits_{j<i} \bar{a}_j = \bar{a}_{i-1} < \bar{a}_{i+1} =  \min\limits_{j>i} \bar{a}_j$, then $\bar{a}_i \in (\bar{a}_{i-1}, \bar{a}_{i+1})$.
\end{itemize}
\end{itemize}

\noindent We first show that $\bar{a}_1\leq \bar{a}_2 \leq \dots \leq \bar{a}_m$ by induction. Since $\bar{a}_1 = -1$ and $\bar{a}_2 \in [-1,1]$, we have $\bar{a}_1 \leq \bar{a}_2$.
Now assume by induction that $\bar{a}_{i-1} \leq \bar{a}_i$ for some $i < m$.
Suppose for the sake of contradiction that $\bar{a}_i > \bar{a}_{i+1}$.
If $\bar{a}_{i-1} > \bar{a}_{i+1}$, we have $\bar{a}_i \in (\bar{a}_{i+1},\bar{a}_{i-1})$ and in particular $\bar{a}_i < \bar{a}_{i-1}$, contradicting the inductive hypothesis.
If $\bar{a}_{i-1} \leq \bar{a}_{i+1}$, we have $\bar{a}_i \in [\bar{a}_{i-1},\bar{a}_{i+1}]$ and in particular, $\bar{a}_i \leq \bar{a}_{i+1}$, a contradiction.
% If $\bar{a}_{i-1} = \bar{a}_{i+1}$, we have by optimality that $\bar{a}_i = \bar{a}_{i+1}$ and in particular $\bar{a}_i \leq \bar{a}_{i+1}$, a contradiction.
Thus, we have $\bar{a}_i \leq \bar{a}_{i+1}$ and by induction $\bar{a}_1\leq \dots \leq \bar{a}_m$.

% Now we show that $\bar{a}_1 < \bar{a}_2 < \dots < \bar{a}_m$.
Suppose for the sake of contradiction that (at least) two actions are equal.
% Let $\{\bar{a}_i\}$ be a set of equal actions of maximum cardinality, where $i$ is the smallest index corresponding to an action in this set.
Let $k \in \{1,\ldots,m\}$ be the smallest index corresponding to an action in a set of equal actions of maximum cardinality.
Either $\bar{a}_k \neq -1$ or $\bar{a}_k \neq 1$. Assume without loss of generality that $\bar{a}_k \neq -1$.
Since $\bar{a}_1 = -1$, we must have $k > 1$.
By minimality of $k$, we must have $\bar{a}_{k-1} < \bar{a}_k$.
Since we have $\bar{a}_1\leq \bar{a}_2 \leq \dots \leq \bar{a}_m$, we have that $\max\limits_{j<k} \bar{a}_j = \bar{a}_{k-1}$ and $\bar{a}_{k+1} =\min\limits_{j>k} \bar{a}_j$. Since we further have $\bar{a}_{k-1} < \bar{a}_k = \bar{a}_{k+1}$, by the last (sub)-bullet above, we must have that $\bar{a}_k \in (\bar{a}_{k-1},\bar{a}_{k+1})$, and in particular $\bar{a}_k < \bar{a}_{k+1}$, contradicting the fact that the equal action set has size at least two.
Thus, $\bar{a}_1 < \dots < \bar{a}_m$.

% First we claim that $a_2 > -1$.
% Suppose for contradiction that $a_2 = -1$.
% If $a_3 > -1$, then by optimality we would have $a_2 > -1$, a contradiction, thus suppose that $a_3 = -1$.
% Continuing this logic we must have that $a_i = -1$ for all $i$, however we must have $a_m = 1$, a contraction.
% Therefore $a_2 > -1$.

% We claim that no two $\bar{a}_i$ can be equal. Fix a maximal set of identical actions. Suppose for the sake of contradiction that this set has size at least two. 
% % The identical actions cannot all be equal to $-1$ and $1$. 
% If the identical actions are not equal $-1$, the first identical player (i.e., the player with the smallest index) can improve their payoff by playing a smaller action. If the identical actions are not equal to $-1$, the last identical player can improve their payoff by playing a larger action. Thus no two $\bar{a}_i$ are equal. A similar argument also shows that no $\bar{a}_i$ can be equal to $-1$ or $1$.

Going back to the statement of the lemma, we define, for $i=1,\dots,m$, $t_i = \bar{a}_i$ and
$$p_i(t) = \frac{q_i(t)}{q_i(t_i)},$$
where
$$q_1(t) = (1-t)\prod_{j\not \in \{1,m\}} (t-t_j)^2, \quad \quad \quad q_m(t) = (1+t)\prod_{j\not \in \{1,m\}} (t-t_j)^2,$$ $$q_i(t) = (1-t^2)\prod_{j\not \in \{1,i,m\}} (t-t_j)^2 \quad \quad i=2,\dots,m-1.$$

Note that since no two members of $\{\bar{a}_1,\dots,\bar{a}_m\}$ coincide, we have $q_i(t_i) \neq 0$ for $i=1,\dots,m$, and therefore the polynomials $p_i(t)$ are well defined. With this construction, property 3 of Lemma~\ref{lem:poly simplex} is immediate from the definition of $p_i(t)$ as $p_i(t_i) = 1$. Property 1 is also straightforward to check since each $p_i(t)$ is a positive scaling of a product of squares multiplied by either $(1-t)$,$(1+t)$, or $(1-t^2)$ which are all nonnegative for $t\in[-1,1]$.

It remains to verify property 2. Fix an arbitrary index $i \in \{2,\dots, m-1 \}$.
First observe that since $\bar{a}_1 \leq \bar{a}_2 \leq \dots \leq \bar{a}_m$, for $a_i \in (\bar{a}_{i-1},\bar{a}_{i+1})$, the payoff function $u_i(a_i,\bar{a}_{-i}) = q_i(a_i)$.
Then by~\eqref{eq:we_got_Nash}, the point $t_i=\bar{a}_i$ is a local maximum of $q_i(t)$ and thus also of $p_i(t)$.
Combined with the fact that $p_i(t_i)=1$, it follows that the point $t_i$ is a double root of $1-p_i(t)$.
Additionally, for $k\in \{1,\dots,i-1,i+1,\dots,m \}$, $t_i$ is a double root of $p_k(t)$ by construction.
Hence, $t_i$ is a double root of $1-\sum_{\ell=1}^m p_\ell(t)$.
The polynomial $1-\sum_{\ell=1}^m p_\ell(t)$ also has roots at $-1$ and $1$.
Thus, counting with multiplicity, $1-\sum_{\ell=1}^m p_\ell (t)$ has $2m-2$ roots.
However, since for $\ell=1,\dots,m$, each $p_\ell(t)$ has degree at most $2m-3$, the degree of $1-\sum_{\ell=1}^m p_\ell (t)$ is at most $2m-3$.
Since a degree-$d$ nonzero polynomial has at most $d$ roots, it follows that $1-\sum_{\ell=1}^m p_\ell(t)$ is identically zero.
Therefore, $\sum_{\ell=1}^m p_\ell(t) = 1$ for all $t$.
% To verify property 2, first observe that since $\bar{a}_1 \leq \bar{a}_2 \leq \dots \leq \bar{a}_m$, for $i=2,\dots,m-1$ and $a_i \in (\bar{a}_{i-1},\bar{a}_{i+1})$, the payoff function $u_i(a_i,\bar{a}_{-i}) = q_i(a_i)$.
% Then by~\eqref{eq:we_got_Nash}, for each $i=2,\dots,m-1$, the point $t_i=\bar{a}_i$ is a local maximum of $q_i(t)$ and therefore also $p_i(t)$.
% Combined with the fact that $p_i(t_i)=1$, it follows that for each $i=2,\dots,m-1$, the point $t_i$ is a double root of $1-p_i(t)$.
% Additionally, for each $i=2,\dots,m-1$ and each $j\not \in \{1,i,m\}$, $t_j$ is clearly a double root of $p_i(t)$.
% Thus for each $i=2,\dots,m-1$, $t_i$ is a double root of $1-\sum_{i=1}^m p_i(t)$.
% The polynomial $1-\sum_{i=1}^m p_i(t)$ also has roots at $-1$ and $1$.
% Thus, including multiplicity, $1-\sum_{i=1}^m p_i$ has $2m-2$ roots.
% By construction $\max_i \deg p_i = 2m-3$, thus $\deg (1-\sum_{i=1}^m p_i(t)) = 2m-3$.
% Since a degree $d$ nonzero polynomial has at most $d$ roots, it follows that $1-\sum_{i=1}^m p_i(t)$ is identically zero and thus $\sum_{i=1}^m p_i(t) = 1$ for all $t$ and therefore the constructed polynomials satisfy property 2 of Lemma~\ref{lem:poly simplex}.
\end{proof}
Figure~\ref{fig:simplex_plot2} demonstrates polynomials $p_i(t)$ that satisfy the three conditions of Lemma~\ref{lem:poly simplex} for the case $m=5$. These polynomials were found by a numerical search for a Nash equilibrium for the game set up in the proof of Lemma~\ref{lem:poly simplex}. Interestingly, this equilibrium corresponds to the roots of the Legendre polynomial of degree 3. The next lemma proves that the roots of Legendre polynomials~(see, e.g., \cite[Chapter 22]{abramowitz_stegun} for a definition) always provide a Nash equilibrium to our game.\footnote{In view of the explicit construction of a Nash equilibrium in Lemma~\ref{lem:explicit_Nash}, we do not actually need to invoke Theorem~\ref{thm:nash} for our purposes. However, we include Theorem~\ref{thm:nash} to highlight that the game theory approach could potentially be used more generally to prove existence of polynomials with certain desired properties even when a Nash equilibrium cannot be constructed explicitly.}
%coincide with the polynomials that appear in the constructive proof of this lemma, which we present next.
\begin{figure}[ht]
\centering
\includegraphics[width=.45\linewidth]{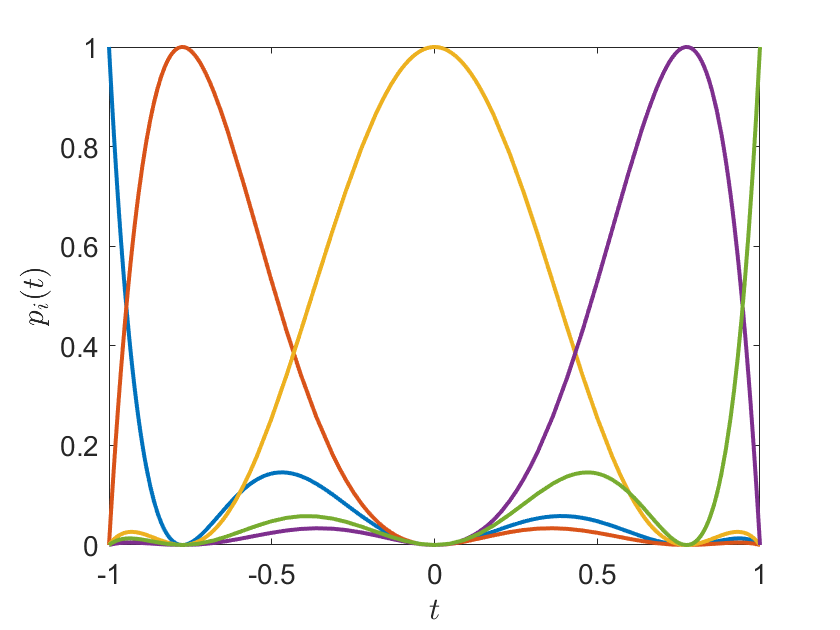}
\caption{Polynomials $p_i(t)$ satisfying the three conditions of Lemma~\ref{lem:poly simplex} for $m=5$.}
\label{fig:simplex_plot2}
\end{figure}

\begin{lemma}
\label{lem:explicit_Nash}
Let $\bar{a}_1=-1$, $\bar{a}_m=1$, and let $\bar{a}_2,\dots,\bar{a}_{m-1}$ be the roots of the Legendre polynomial of degree $m-2$, arranged in ascending order.
Then the actions $\bar{a}_1,\dots,\bar{a}_m$ form a pure-strategy Nash equilibrium of the game defined in the proof of Lemma~\ref{lem:poly simplex}.
\end{lemma}
\begin{proof}
Since the action sets of players $1$ and $m$ are singletons, it is trivial that~\eqref{eq: nash} holds for $i =1,m$. Fix an arbitrary index $i \in \{2,\dots, m-1 \}$.
From the definition of the payoff function $u_i$ in~\eqref{eq: u_i def}, $u_i(a_i,\bar{a}_{-i})$ is only positive in the interval $(\bar{a}_{i-1},\bar{a}_{i+1})$, where it is equal to the polynomial function $(1-a_i^2) \prod_{j\not\in \{1,i,m\}} (a_i-\bar{a}_j)^2 $.
Thus, 
$$ \max_{a_i\in [-1,1]} u_i(a_i,\bar{a}_{-i}) = \max_{a_i\in (\bar{a}_{i-1},\bar{a}_{i+1})} u_i(a_i,\bar{a}_{-i}) = \max_{a_i\in (\bar{a}_{i-1},\bar{a}_{i+1})} (1-a_i^2) \prod_{j\not\in \{1,i,m\}} (a_i-\bar{a}_j)^2. $$
% From the definition of the payoff function $u_i$ in \eqref{eq: u_i def}, 
% $u_i(a_i,\bar{a}_{-i})$ is only positive in the interval $(\bar{a}_{i-1},\bar{a}_{i+1})$.
% Thus, we have $$ \max_{a_i\in [-1,1]} u_i(a_i,\bar{a}_{-i}) = \max_{a_i\in (\bar{a}_{i-1},\bar{a}_{i+1})} u_i(a_i,\bar{a}_{-i}). $$
% As shown in the proof of Lemma~\ref{lem:poly simplex}, when $a_i$ is in the interval $(\bar{a}_{i-1},\bar{a}_{i+1})$, $u_i(a_i,\bar{a}_{-i})$ is equal to the polynomial $(1-a_i^2) \prod_{j\not\in \{1,i,m\}} (a_i-\bar{a}_j)^2 $.
% Thus, $$\max_{a_i\in (\bar{a}_{i-1},\bar{a}_{i+1})} u_i(a_i,\bar{a}_{-i}) = \max_{a_i\in (\bar{a}_{i-1},\bar{a}_{i+1})} (1-a_i^2) \prod_{j\not\in \{1,i,m\}} (a_i-\bar{a}_j)^2. $$
This polynomial is zero at $\bar{a}_{i-1}$ and $\bar{a}_{i+1}$ and positive in the interval $(\bar{a}_{i-1},\bar{a}_{i+1})$.
As argued in the proof of Lemma~\ref{lem:poly simplex}, the derivative of this polynomial has at most one root over $(\bar{a}_{i-1},\bar{a}_{i+1})$.
% It follows that this polynomial attains its maximum over the interval $(\bar{a}_{i-1},\bar{a}_{i+1})$ at the unique critical point in the same interval.
Therefore, any root of the derivative in this interval attains the maximum.
Thus, in order to show that $$
u_i(\bar{a}_i,\bar{a}_{-i}) = \max_{a_i\in (\bar{a}_{i-1},\bar{a}_{i+1})} (1-a_i^2) \prod_{j\not\in \{1,i,m\}} (a_i-\bar{a}_j)^2,
$$
if suffices to show that
$$
\frac{d}{da_i} \left( (1-a_i^2) \prod_{j\not\in \{1,i,m\}} (a_i-\bar{a}_j)^2\middle) \right|_{a_i=\bar{a}_i} = 0.
$$
To this end, we recall that due to the properties of the Legendre polynomials, the polynomial function $\ell_{m-2}(t) \defin \prod_{j \notin \{1,m\}} (t-\bar{a}_j)$ satisfies the Legendre differential equation (see, e.g.,~\cite[22.6.13]{abramowitz_stegun}):
$$
 (1-t^2)\ell_{m-2}''(t) - 2t\ell_{m-2}'(t)+(m-2)(m-1)\ell_{m-2}(t)=0 \quad \forall t.
$$
Observe that
\begin{align*}
\ell_{m-2}'(t) &= \prod_{j \notin \{1,2,m\}} (t-\bar{a}_j) + \prod_{j \notin \{1,3,m\}} (t-\bar{a}_j) + \quad \cdots \quad + \prod_{j \notin \{1,m-1,m\}} (t-\bar{a}_j) = \sum_{k=2}^{m-1} \prod_{j \notin \{1,k,m\}} (t-\bar{a}_j),\\
\ell_{m-2}''(t) &=  2 \sum_{k_1=2}^{m-2} \:\: \sum_{k_2=k_1+1}^{m-1} \:\: \prod_{j \notin \{1,k_1,k_2,m\}} (t-\bar{a}_j).
\end{align*}
Hence, we have
$$\ell_{m-2}'(\bar{a}_i) = \prod_{j \notin \{1,i,m\}} (\bar{a}_i-\bar{a}_j) \quad \text{and} \quad \ell_{m-2}''(\bar{a}_i) = 2 \sum_{k \notin \{1,i,m\}} \prod_{j \notin \{1,i,k,m\}} (\bar{a}_i-\bar{a}_j).$$
One can then check that
\begin{align*}
& \frac{d}{da_i} \left( (1-a_i^2) \prod_{j\not\in \{1,i,m\}} (a_i-\bar{a}_j)^2\middle) \right|_{a_i=\bar{a}_i}\\
&= -2 \bar{a}_i \left( \prod_{j\not \in \{1,i,m\}} (\bar{a}_i - \bar{a}_j)^2 \right) + (1-\bar{a}_i^2) \left( \sum_{k\not \in \{1,i,m\}} 2(\bar{a}_i-\bar{a}_k) \prod_{j\not \in \{1,i,k,m\}} (\bar{a}_i-\bar{a}_j)^2 \right) \\
&= \left( \prod_{j\not \in \{1,i,m\}} (\bar{a}_i-\bar{a}_j) \right) \left( -2\bar{a}_i \prod_{j\not \in \{1,i,m\}} (\bar{a}_i - \bar{a}_j) + 2(1-\bar{a}_i^2) \sum_{k\not \in \{1,i,m\}} \prod_{j\not \in \{1,i,k,m\}} (\bar{a}_i-\bar{a}_j) \right)\\
&= \left( \prod_{j\not \in \{1,i,m\}} (\bar{a}_i-\bar{a}_j) \right) \left( - 2\bar{a}_i\ell_{m-2}'(\bar{a}_i) + (1-\bar{a}_i^2)\ell_{m-2}''(\bar{a}_i) \right) \\
&= \left( \prod_{j\not \in \{1,i,m\}} (\bar{a}_i-\bar{a}_j) \right) \left( (1-\bar{a}_i^2)\ell_{m-2}''(\bar{a}_i) - 2\bar{a}_i\ell_{m-2}'(\bar{a}_i)+(m-2)(m-1)\ell_{m-2}(\bar{a}_i) \right) = 0,
\end{align*}
where the last equality follows from the Legendre differential equation stated above.
\end{proof}

\begin{remark}\label{rem:explicit}
Combining Lemma~\ref{lem:explicit_Nash} and the arguments from the proof of Lemma~\ref{lem:poly simplex}, one can explicitly construct polynomials $p_1,\dots,p_m$ satisfying the requirements of Lemma~\ref{lem:poly simplex} as follows.
Let $t_1 = -1, t_m = 1$, and let $t_2,\dots,t_{m-1}$ be the roots of the Legendre polynomial of degree $m-2$.
Then, we can have
\[
p_1(t) = \left(\frac{1-t}{2}\right)\prod_{j\not \in \{1,m\}} \left(\frac{t-t_j}{1 + t_j}\right)^2, \quad
p_m(t) = \left(\frac{1+t}{2}\right)\prod_{j\not \in \{1,m\}} \left(\frac{t-t_j}{1 - t_j}\right)^2, \quad \text{and}
\]
\[
p_i(t) = \left(\frac{1-t^2}{1-t_i^2}\right)\prod_{j\not\in \{1,i,m\}} \left(\frac{t-t_j}{t_i - t_j}\right)^2 \quad i=2,\dots,m-1.
\]
\end{remark}

\begin{lemma}\label{lem:min_degree}
For $m \geq 2$, let $\{p_1,\dots,p_m\}$ be any set of polynomials satisfying the three properties of Lemma~\ref{lem:poly simplex}.
Then at least one of these polynomials must have degree at least $2m-3$.
\end{lemma}
\begin{proof}
The claim is trivial for $m=2$, so assume $m \geq 3$.
Let $t_1,\dots,t_m\in [-1,1]$ be any set of points such that $p_i(t_i) =1$ for $i=1,\dots,m$. 
The properties of Lemma~\ref{lem:poly simplex} imply that the points $t_1,\dots,t_m$ are distinct and that $p_i(t_j)=0$ for $i\neq j$.
If $-1$ or $1$ are in the set $\{t_1,\dots,t_m\}$, fix $i$ such that $t_i \in \{-1,1\}$;
otherwise choose $i$ arbitrarily.
If both $-1$ and $1$ are in the set $\{t_1,\dots,t_m\}$, fix $j$ such that $\{t_i,t_j\} = \{-1,1\}$; otherwise choose any $j \neq i$.
For each $k \not \in \{i,j\}$, since $p_i(t_k) = 0$, $p_i(t) \geq 0$ for $t \in [-1,1]$, and $t_k \not \in \{-1,1\}$, $t_k$ must be a root of $p_i(t)$ of multiplicity at least two.
% Furthermore, since $p_i(t_j) = 0$, $t_j$ must be a root of $p_i(t)$.
Furthermore, since $t_j$ is a root of $p_i(t)$, counting with multiplicities, $p_i(t)$ must have at least $2(m-2) + 1 = 2m-3$ roots.
\end{proof}

We now shift our focus back to GEs and show that they can represent any compact finite intersection of semiellipsoids.

\begin{proof}[Proof of Theorem~\ref{thm:semiellipsoid}]
Let $T \subset \R^n$ be a compact finite intersection of $m$ semiellipsoids.
Then there exist matrices $P_1,\dots, P_m \in S_+^{n}$ such that $T = \bigcap_{i=1}^m \{x \in \R^n \mid x^T P_i x \leq 1\}$.
% Then there are semiellipsoids $T_i \subset \R^n$ for $i=1,\dots,m$ such that $T = \bigcap_{i=1}^m T_i$ and psd matrices $P_i \in S^{n}$ such that $T_i = \{x \in \R^n \mid x^T P_i x \leq 1\}$.
Let $\{p_1,\dots,p_m\}$ be polynomials of degree at most $2m-3$ satisfying the three properties in Lemma~\ref{lem:poly simplex}.
Define the polynomial matrix $P(t)$ and the set $\bar{T}$ as follows:
$$P(t) = \sum_{i=1}^m p_i(t) P_i \quad \text{and} \quad \bar{T} = \bigcap_{t \in [-1,1]} \{x \in \R^n \mid x^TP(t)x \leq 1\}.$$
We claim that $\bar{T} = T$ and that $\bar{T}$ is a GE-$d$ (with $d \leq 2m-3$).
To see that $\bar{T} \subseteq T$, recall first that, by Lemma~\ref{lem:poly simplex}, for each $i = 1,\dots,m, \exists t_{i} \in [-1,1]$ such that $p_i(t_{i}) = 1$.
Then we have
$$\bar{T} = \bigcap_{t \in [-1,1]} \{x \in \R^n \mid x^TP(t)x \leq 1\} \\
 \subseteq \bigcap_{i=1}^m \{x \in \R^n \mid x^TP(t_i)x \leq 1\} \\
 = \bigcap_{i=1}^m \{x \in \R^n \mid x^TP_ix \leq 1\} = T.$$
% \begin{align*}
% \hspace{4.5cm}
% \bar{T} &= \bigcap_{t \in [-1,1]} \{x \in \R^n \mid x^TP(t)x \leq 1\} \\
% & \subseteq \bigcap_{i=1}^m \{x \in \R^n \mid x^TP(t_i)x \leq 1\} \\
% & \subseteq \bigcap_{i=1}^m \{x \in \R^n \mid x^TP_ix \leq 1\} = T.
% \end{align*}
To see that $T \subseteq \bar{T}$, take $x \in T$ and $t \in [-1,1]$.
Recall by Lemma~\ref{lem:poly simplex} that $\sum_{i=1}^m p_i(t) = 1$.
Then we have
$$
x^TP(t)x = \sum_{i=1}^m p_i(t) x^T P_i x \leq \max_{i=1,\dots,m} x^T P_i x \leq 1.
$$
Thus, $x\in \bar{T}$. Finally, we note that $\bar{T}$ is a valid GE-$d$.
The psd condition holds since for $i = 1,\dots,m$, $p_i(t) \geq 0$ for all $t \in [-1,1]$, and $P_i \succeq 0$.
The kernel condition holds since $\bar{T}$ is compact (as $T$ is compact).
\end{proof}

\begin{corollary}
\label{cor:ellipsoids and polytopes}
Every symmetric full-dimensional polytope and every finite intersection of co-centered ellipsoids is a GE.
\end{corollary}

\begin{proof}
The claim for a finite intersection of co-centered ellipsoids is immediate from Theorem~\ref{thm:semiellipsoid}.
Let $T \subset \R^n$ be a symmetric full-dimensional polytope.
By translation, we may assume $T$ is symmetric around the origin; from this and full-dimensionality, it follows that $T$ contains the origin in its interior. Thus, we may write 
$$T = \{ x \in \R^n \mid |a_i^T x| \leq 1 \quad i=1,\dots,m\}$$
for some vectors $a_1,\dots,a_m \in \R^n$.
Therefore, the description of $T$ can be rewritten as
$$T = \bigcap_{i=1}^m \{ x \in \R^n \mid x^T (a_i a_i^T) x \leq 1 \}.$$
As $T$ is evidently a (compact) finite intersection of semiellipsoids, the claim follows from Theorem~\ref{thm:semiellipsoid}.
\end{proof}

We now show that every symmetric convex body can be approximated arbitrarily well by a GE.

\begin{theorem}
\label{thm:approximation_by_GE}
There exists $\varepsilon_0 > 0$ such that for any $\varepsilon \in (0,\varepsilon_0)$ and for any symmetric convex body $C \subset \mathbb{R}^n$, there is a GE-$d$ $\mathcal{E}_d$, with $d \leq 2\left( \frac{1}{2\sqrt{\varepsilon}} \log(\frac{1}{\varepsilon}) \right)^n - 3$, satisfying
\begin{equation}\label{eq:convex.body.approximation}
\hspace{5.7cm}
\mathcal{E}_d \: \subseteq \: C \: \subseteq \: (1+\varepsilon) \mathcal{E}_d.
\end{equation}
\end{theorem}
\begin{proof}
Let $C^o \defin \{y\in \R^n \mid \langle x,y \rangle \leq 1\}$ be the polar dual of $C$. Since $C^o$ is a convex body, by \cite[Corollary 1.2]{barvinok}, there exists $\varepsilon_0 >0$ ($\varepsilon_0(\frac{1}{2})$ in the language of that paper) such that for any $\varepsilon \in (0,\varepsilon_0)$, there exists a symmetric (full-dimensional) polytope $T \subset \R^n$ with at most $\left(\frac{1}{2\sqrt{\varepsilon}} \log(\frac{1}{\varepsilon})\right)^n$ vertices that satisfies $T \subseteq C^o \subseteq (1+\varepsilon) T$.
By duality, we have $\frac{1}{1+\varepsilon} T^o \subseteq C \subseteq T^o$.
Then, the set $\mathcal{E}_d \defin \frac{1}{1+\varepsilon} T^o$ satisfies~\eqref{eq:convex.body.approximation}.
Since $\mathcal{E}_d$ is a scaling of the polar dual of $T$, it is symmetric, full-dimensional, and the number of its facets is at most $\left(\frac{1}{2\sqrt{\varepsilon}} \log(\frac{1}{\varepsilon})\right)^n$ (see, e.g., \cite[Chap. 4]{barvinok_book}). By~Corollary~\ref{cor:ellipsoids and polytopes}, $\mathcal{E}_d$ is a GE-$d$ with $d \leq 2\left(\frac{1}{2\sqrt{\varepsilon}} \log(\frac{1}{\varepsilon})\right)^n - 3$.
\end{proof}

\subsection{GEs under polar duality}\label{subsec:GEPloarity}
Considering that the polar dual of a polytope (resp. an ellipsoid) is always a polytope (resp. an ellipsoid), it is natural to wonder what one can say about the polar dual of a GE.
In particular, for a GE-$d$, is its polar dual also a GE-$d$?
Is its polar dual a GE at all? For $d=0$, the answer to the first question (and hence the second question) is clearly positive. In this section, we show that already when $d=1$, three different scenarios may arise.

%%%%%%%%%%%%%%%%%%%%%%%%%%%%%%%%
\subsubsection{A GE-$1$ whose polar dual is a GE-$1$}
Recall that the $\ell_1$ and $\ell_\infty$ balls are polar dual to each other.
In dimension two, both of these sets are GE-$1$s, thus they are a pair of GE-$1$s which are polar dual to each other.
\begin{figure}[ht]
\centering
\begin{subfigure}{.4\textwidth}
  \centering
  \includegraphics[width=.45\linewidth]{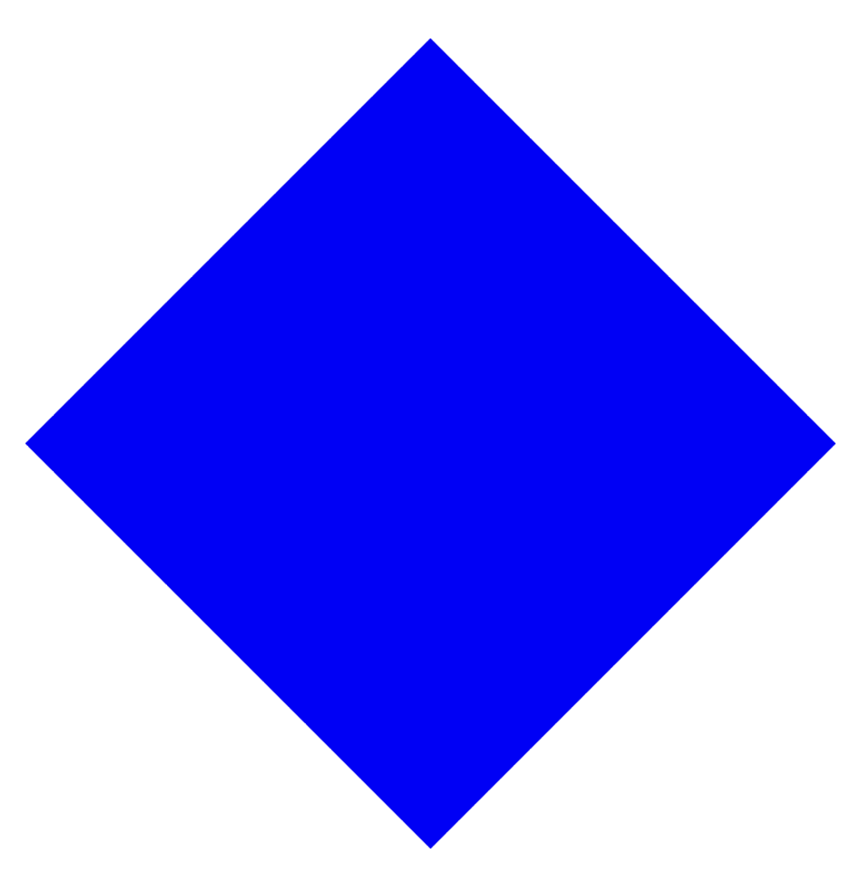}
  \hspace{-0.8cm}
%  \caption{}
\end{subfigure}
\begin{subfigure}{.4\textwidth}
  \centering
  \includegraphics[width=.45\linewidth]{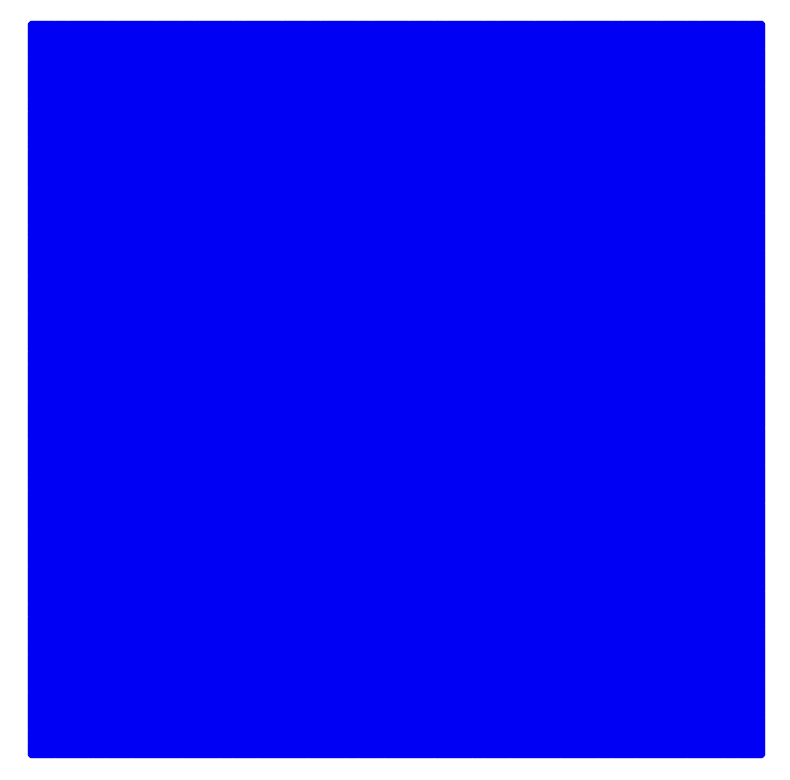}
  \hspace{-1.5cm}
%  \caption{}
\end{subfigure}
\caption{A GE-$1$ whose polar dual is a GE-$1$.}
\end{figure}

%%%%%%%%%%%%%%%%%%%%%%%%%%%%%%%%
\subsubsection{A GE-$1$ whose polar dual is a GE but not a GE-$1$}
Consider the following GE-1:
%The following exhibits a GE-$1$ whose polar dual is a GE-$2$, but not a GE-$1$:
$$
\Omega = \left\{ x\in \R^2 \,\middle\vert\, x^T \begin{bmatrix} \frac{3-t}{2} & 0 \\ 0 & \frac{1+t}{2}\end{bmatrix} x \leq 1 \quad \forall t \in [-1,1] \right\}.
$$
One can check that the polar dual is given by
$$
\Omega^\circ = \left\{ x\in \R^2 \,\middle\vert\, x^T \begin{bmatrix} \frac{1}{2} & \frac{t}{2} \\ \frac{t}{2} & \frac{2-t^2}{2}\end{bmatrix} x \leq 1 \quad \forall t \in [-1,1] \right\},$$
and hence it is a GE-2. However, $\Omega^\circ$ is not the intersection of two semiellipsoids and therefore, following the proof of Proposition~\ref{prop:finite_intersec}, it is not a GE-$1$.
%One can check that the polar dual is not the intersection of two semiellipsoids and therefore, following the proof of Proposition~\ref{prop:finite_intersec}, it is not a GE-$1$.
\begin{figure}[ht]
\centering
\begin{subfigure}{.4\textwidth}
  \centering
  \includegraphics[width=.35\linewidth]{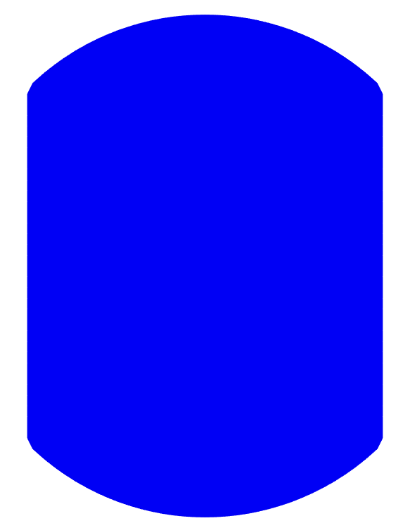}
  \hspace{-0.8cm}
%  \caption{}
\end{subfigure}
\begin{subfigure}{.4\textwidth}
  \centering
  \includegraphics[width=.5\linewidth]{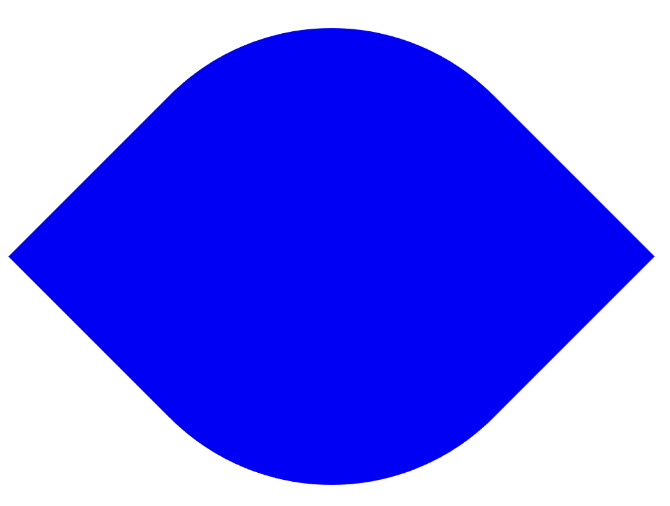}
  \hspace{-1.5cm}
%  \caption{}
\end{subfigure}
\caption{A GE-$1$ (left) whose polar dual (right) is a GE but not a GE-$1$.}
\end{figure}

%%%%%%%%%%%%%%%%%%%%%%%%%%%%%%%%
\subsubsection{A GE-$1$ whose polar dual is not a GE}
\begin{theorem}\label{thm:not GE}
The polar dual of the GE-1
$$
\Omega = \left\{ x\in \R^2 \,\middle\vert\, x^T \begin{bmatrix} \frac{5-3t}{2} & 0 \\ 0 & \frac{5+3t}{2}\end{bmatrix} x \leq 1 \quad \forall t \in [-1,1] \right\},
$$
is not a GE.
\end{theorem}
\begin{figure}[ht]
\centering
\begin{subfigure}{.4\textwidth}
  \centering
  \includegraphics[width=.4\linewidth]{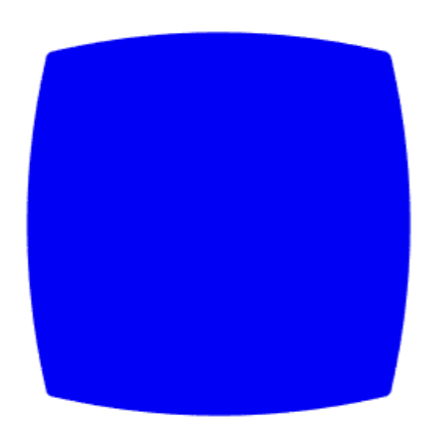}
  \hspace{-0.8cm}
%  \caption{}
\end{subfigure}
\begin{subfigure}{.4\textwidth}
  \centering
  \includegraphics[width=.45\linewidth]{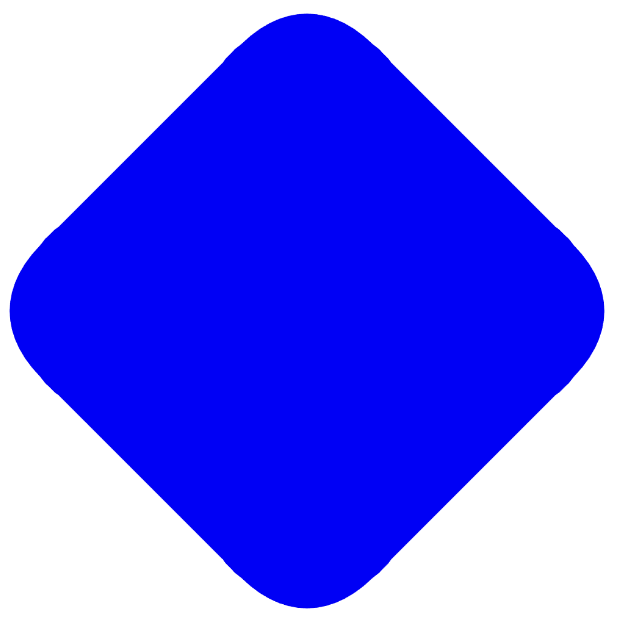}
  \hspace{-1.5cm}
%  \caption{}
\end{subfigure}
\caption{A GE-$1$ (left) whose polar dual (right) is not a GE.}
\end{figure}

\begin{proof}
Note that
\begin{equation}\label{eq: omega polar}
\begin{aligned}
\Omega^\circ & = \left\{ x\in \R^2 \,\middle\vert\, x^T \begin{bmatrix} \frac{5-3t}{2} & 0 \\ 0 & \frac{5+3t}{2}\end{bmatrix} x \leq 1 \quad \forall t \in [-1,1] \right\}^\circ \\
& = \left( \left\{ x\in \R^2 \,\middle\vert\, x^T \begin{bmatrix} 4 & 0 \\ 0 & 1\end{bmatrix} x \leq 1 \right\} \bigcap \left\{ x\in \R^2 \,\middle\vert\, x^T \begin{bmatrix} 1 & 0 \\ 0 & 4\end{bmatrix} x \leq 1 \right\} \right)^\circ \\
& = \text{conv}\left( \left\{ x\in \R^2 \,\middle\vert\, x^T \begin{bmatrix} \frac{1}{4} & 0 \\ 0 & 1\end{bmatrix} x \leq 1 \right\} \bigcup \left\{ x\in \R^2 \,\middle\vert\, x^T \begin{bmatrix} 1 & 0 \\ 0 & \frac{1}{4}\end{bmatrix} x \leq 1 \right\} \right),
\end{aligned}
\end{equation}
where the last equality follows from the fact that the polar dual of the intersection of two closed sets is the convex hull of the union of their polar duals.
Now suppose for the sake of contradiction that $\Omega^\circ$ is a GE defined by some polynomial matrix $P(t) = \begin{bmatrix}
    p_1(t) & p_2(t) \\ p_2(t) & p_3(t)
\end{bmatrix}$.
Let us define a function $f: S^2 \mapsto \R$ as
$$
f(P) = \begin{bmatrix}
    \max_{x\in \R^2} \quad & x^T P x \\
    \text{s.t.} \quad &x^T \begin{bmatrix} \frac{1}{4} & 0 \\ 0 & 1\end{bmatrix} x \leq 1
\end{bmatrix}.
$$
We claim that $f(P(t)) = 1$ for infinitely many values of $t \in [-1,1]$.
First observe that for any $t \in [-1,1]$, we have $f(P(t)) \leq 1$.
Indeed, if there was a vector $\bar{x}$ such that $\bar{x}^T P(t) \bar{x} > 1$ and $\bar{x}^T \begin{bmatrix} \frac{1}{4} & 0 \\ 0 & 1\end{bmatrix} \bar{x} \leq 1$, we would have that $ \left\{ x\in \R^2 \,\middle\vert\, x^T \begin{bmatrix} \frac{1}{4} & 0 \\ 0 & 1\end{bmatrix} x \leq 1 \right\} \not \subseteq \Omega^\circ$, a contradiction in view of~\eqref{eq: omega polar}.
Therefore, to prove the claim, it suffices to show that for infinitely many values of $t \in [-1,1]$, there exists an $x$ such that $x^T \begin{bmatrix} \frac{1}{4} & 0 \\ 0 & 1\end{bmatrix} x \leq 1$ and $x^T P(t) x = 1$.
Let us define the set $A = \left\{ x\in \Omega^\circ \,\middle\vert\, x^T \begin{bmatrix} \frac{1}{4} & 0 \\ 0 & 1\end{bmatrix} x = 1 \right\}$, which has infinite cardinality.
Since $A$ is a subset of the boundary of $\Omega^\circ$, for any $x \in A$, there must exist some $t(x) \in [-1,1]$ such that $x^T P(t(x)) x = 1$.
If the set $\{t(x) \mid x\in A\}$ was finite, then there would be at least one $\bar{t} \in [-1,1]$ such that $\bar{t} = t(x)$ for infinitely many $x \in A$.
%If infinitely many $x \in A$ share the same $\bar{t}=t(x)$, t
Then we would have that the curves $\left\{ x\in \R^2 \,\middle\vert\, x^T \begin{bmatrix} \frac{1}{4} & 0 \\ 0 & 1\end{bmatrix} x = 1 \right\}$ and $\left\{ x\in \R^2 \,\middle\vert\, x^T P(\bar{t}) x = 1 \right\}$ share an infinite number of points, which would imply that $P(\bar{t}) = \begin{bmatrix} \frac{1}{4} & 0 \\ 0 & 1\end{bmatrix}$.
This would imply that 
$$\Omega^\circ = \bigcap_{t \in [-1,1]} \left\{ x\in \R^2 \,\middle\vert\, x^T P(t) x \leq 1 \right\} \subseteq \left\{ x\in \R^2 \,\middle\vert\, x^T \begin{bmatrix} \frac{1}{4} & 0 \\ 0 & 1\end{bmatrix} x \leq 1 \right\},$$ which in view of~\eqref{eq: omega polar} is a contradiction.
Since we have established that the set $\{t(x) \mid x\in A\}$ is infinite, it follows that for infinitely many values $t \in [-1,1]$, there exists an $x$ such that $x^T \begin{bmatrix} \frac{1}{4} & 0 \\ 0 & 1\end{bmatrix} x \leq 1$ and $x^T P(t) x = 1$.
As mentioned above, this proves the claim that for infinitely many values of $t \in [-1,1]$, we have $f(P(t)) = 1$.

Now let us observe by the change of variable $y = \begin{bmatrix}
    \frac{1}{2} & 0 \\ 0 & 1
\end{bmatrix} x$ that we can write
$$
f(P) = \begin{bmatrix}
    \max_{y\in \R^2} \quad & y^T \begin{bmatrix}
    2 & 0 \\ 0 & 1
\end{bmatrix} P \begin{bmatrix}
    2 & 0 \\ 0 & 1
\end{bmatrix} y \\
    \text{s.t.} \quad &y^Ty \leq 1
\end{bmatrix} = \left\|\begin{bmatrix}
    2 & 0 \\ 0 & 1
\end{bmatrix} P \begin{bmatrix}
    2 & 0 \\ 0 & 1
\end{bmatrix} \right\|.
$$
For $t \in [-1,1]$, since $P(t) \succeq 0$, we have that $$f(P(t)) = \left\|\begin{bmatrix}
    2 & 0 \\ 0 & 1
\end{bmatrix} P(t) \begin{bmatrix}
    2 & 0 \\ 0 & 1
\end{bmatrix} \right\| = \left\|\begin{bmatrix}
    4p_1(t) & 2p_2(t) \\ 2p_2(t) & p_3(t)
\end{bmatrix} \right\| = \lambda_{\max} \left( \begin{bmatrix}
    4p_1(t) & 2p_2(t) \\ 2p_2(t) & p_3(t)
\end{bmatrix}\right).$$
Therefore, if $f(P(t))=1$, then $1$ is an eigenvalue of $\begin{bmatrix}
    4p_1(t) & 2p_2(t) \\ 2p_2(t) & p_3(t)
\end{bmatrix}$ and thus a root of its characteristic equation.
Then we have $\begin{vmatrix}
    4p_1(t) -1 & 2p_2(t) \\ 2p_2(t) & p_3(t) -1
\end{vmatrix} = (4p_1(t) -1)(p_3(t) -1) - (2p_2(t))^2=0$.
Since this occurs for infinitely many $t$, we must have that the polynomial 
\begin{equation}
\label{eq: zero}
\hspace{4.7cm}
(4p_1(t) -1)(p_3(t) -1) - (2p_2(t))^2
\end{equation} 
is identically zero.

By defining a function $g: S^2 \mapsto \R$ as
$$
g(P) = \begin{bmatrix}
    \max_{x\in \R^2} \quad & x^T P x \\
    \text{s.t.} \quad &x^T \begin{bmatrix} 1 & 0 \\ 0 & \frac{1}{4}\end{bmatrix} x \leq 1
\end{bmatrix}
$$
and repeating the same arguments as before we can conclude that the polynomial 
\begin{equation}
\label{eq: zero 2}
\hspace{4.7cm}
(p_1(t) -1)(4p_3(t) -1) - (2p_2(t))^2
\end{equation} 
is also identically zero.

Equating \eqref{eq: zero} and \eqref{eq: zero 2}, we observe that $(4p_1(t) -1)(p_3(t) -1) \equiv (p_1(t) -1)(4p_3(t) -1)$ and thus $p_1(t) \equiv p_3(t)$.
Now replacing $p_3(t)$ with $p_1(t)$ in \eqref{eq: zero} or \eqref{eq: zero 2}, we have $(4p_1(t) -1)(p_1(t) -1) \equiv (2p_2(t))^2$.
After some manipulation, we have $(p_1(t)-\frac{5}{8}+p_2(t))(p_1(t)-\frac{5}{8}-p_2(t))\equiv \frac{9}{64}$.
Since the product of two polynomials is a constant only if both polynomials are constants, there are must exist scalars $k_1,k_2$ such that $p_1(t)-\frac{5}{8}+p_2(t) \equiv k_1$ and $p_1(t)-\frac{5}{8}-p_2(t) \equiv k_2$.
From this we have that $p_2(t) \equiv \frac{k_1-k_2}{2}$ and $p_1(t) \equiv \frac{5}{8} + \frac{k_1+k_2}{2}$.
Thus we have that the matrix $P(t)$ is a constant which would imply that $\Omega^\circ$ is an ellipsoid, a contradiction.
\end{proof}

We remark that since the polar dual of an SDP-representable set is SDP-representable~(see, e.g., \cite[Theorem 3.5]{NetzerPlaumann}), Theorem~\ref{thm:not GE} also implies that there are SDP-representable symmetric convex bodies that are not GEs.

%%%%%%%%%%%%%%%%%%%%%%%%%%%%%%%%%%%%%%%%%%%%%
\section{Applications}
\label{sec:applications}

In this section, we present four potential applications involving GEs.

\subsection{Minimum-variance portfolio optimization with time-varying covariance}\label{subsec:markowitz}

In its simplest form, the minimum-variance portfolio optimization problem in finance takes the form 
\begin{align*}
\hspace{6.2cm}
\min_{x \in \R^n} \quad & x^T \Sigma x \\
\text{s.t.} \quad & \mathbf{1}^Tx = 1, \quad x \geq 0,
\end{align*}
where, for $i=1,\ldots,n$, the $i^{\text{th}}$ entry of the portfolio $x \in \R^n$ determines the fraction of our wealth that we invest in asset $i$. Here, $\mathbf{1}$ denotes the vector of all ones, the nonnegativity constraint on $x$ is entrywise, and $\Sigma \in S_+^n$ is the covariance matrix of the underlying asset returns\footnote{The \emph{return} of asset $i$ over a period is given by $\frac{p^i_{\text{end}} - p^i_{\text{beg}}}{p^i_{\text{beg}}}$, where $p^i_{\text{beg}}$ (resp. $p^i_{\text{end}}$) is the price of the asset at the beginning (resp. the end) of the period.} over a fixed time period and is assumed to be known.\footnote{In the Markowitz variant of this problem, one has an additional linear constraint that imposes a lower bound on the expected return of the portfolio. Such a constraint can be easily incorporated into our framework. However, our focus here is on the variance of the portfolio since we want to highlight how time-varying versions of convex quadratic programs can give rise to GEs.}

We consider a generalization of this problem where we commit to a portfolio at the start of a period (say at time $t=-1$), but allow ourselves to liquidate the portfolio at any time within a given horizon (say at anytime $t\in[-1,1]$).
In this setting, the covariance matrix is no longer constant and depends on the liquidation time.
In other words, there is a function $\Sigma: [-1,1] \rightarrow S_+^n$, such that $\Sigma(t)$ is the covariance matrix of the asset returns at time $t$.
Then to find a portfolio that minimizes the worst-case variance of the returns over all possible liquidation times, we must solve the problem
\begin{equation*}
\begin{aligned}
\hspace{6cm}
\min_{x \in \R^n} \quad & \max_{t \in [-1,1]} \quad x^T \Sigma(t) x \\
\text{s.t.} \quad & \mathbf{1}^Tx = 1, \quad x \geq 0.
\end{aligned}
\end{equation*}
It is unreasonable to assume access to the function $\Sigma(t)$.
Instead, we assume we have noisy measurements $\Sigma_1,\dots,\Sigma_m \in S^n$ of this function at times $t_1,\dots,t_m \in [-1,1]$ during similar past periods.
% In practice, however, it is unknown and must be estimated from historical data. 
% Consider a scenario where we invest at the beginning of some time interval, but we are unsure of the optimal time to sell. Suppose we have access to a collection of covariance matrices $\Sigma_1,\dots, \Sigma_m$ obtained from samples of unknown distributions of possible covariance matrices at different times during that interval.
Given this data, one might consider minimizing the worst-case variance with respect to the measurements. This corresponds to solving
\begin{equation}\label{eq:sigma_i_formulation}
\hspace{6cm}
\begin{aligned}
\min_{x \in \R^n} \quad & \max_{i=1,\dots,m} \quad x^T \hat{\Sigma}_i x \\
\text{s.t.} \quad & \mathbf{1}^Tx = 1, \quad x \geq 0,
\end{aligned}
\end{equation}
where $\hat{\Sigma}_i$ is the nearest (e.g., in Frobenius distance) psd matrix to $\Sigma_i$. 

As a model-based alternative to \eqref{eq:sigma_i_formulation}, we propose to fit a polynomial matrix $P(t)$ to the measurements $\Sigma_1,\dots, \Sigma_m$ by solving
\begin{equation}\label{eq:fit_Pt}
\hspace{5cm}
\begin{aligned}
\min_{P \in S_d^n[t]} \quad & \sum_{i=1}^{m} \left\| P\left( t_i \right) - \Sigma_i \right\|^2_F \\
\text{s.t.} \quad & P(t) \succeq 0 \quad \forall t \in [-1,1],
\end{aligned}
\end{equation}
where $S_d^n[t]$ denotes the set of symmetric $n \times n$ univariate polynomial matrices of degree (at most) $d$ and  $\|\cdot\|_F$ denotes the Frobenius norm. Note that the constraint in~\eqref{eq:fit_Pt} ensures that our model produces a valid covariance matrix at all times. In view of Proposition~\ref{prop:sdp_transformation}, this constrained regression problem can be formulated as an SDP.

Let $P^*(t)$ be an optimal solution to~\eqref{eq:fit_Pt}. To find our portfolio, we propose to solve the following problem:
\begin{equation}\label{eq:GE_formulation}
\begin{aligned}
\hspace{6cm}
\min_{x \in \R^n} \quad & \max_{t \in [-1,1]} \quad x^T P^*(t) x \\
\text{s.t.} \quad & \mathbf{1}^Tx = 1, \quad x \geq 0.
\end{aligned}
\end{equation}
This problem searches for a portfolio which has minimum GE-$d$-norm defined by $P^*(t)$ (see~\eqref{eq:GE_d_norm} in Section~\ref{sec:GE_definition} for a definition). By Theorem~\ref{thm:sdp_representable}, problem~\eqref{eq:GE_formulation} can be formulated as an SDP.

\subsubsection{Numerical example}\label{subsubsec:markowitz_example}

As a numerical example, we consider a universe of $n=10$ assets such that the covariance matrix of their returns at time $t$ is given by the (non-polynomial) function
\begin{equation}\label{eq: sigma t}
\hspace{2.7cm}
\Sigma(t) = 6\sin{(t+1)} A_1A_1^T + 2(1-t^2) A_2A_2^T + (t+1)^2 A_3A_3^T,
\end{equation}
where the entries of the matrices $A_1, A_2, A_3 \in \R^{10 \times 2}$ were generated independently and according to the standard Gaussian distribution. Note that $\Sigma(t)\succeq 0$ for $t\in [-1,1]$ and that $\Sigma(-1)=0$ as there is no uncertainty in the return at the beginning of the period. As described before, we do not assume access to the function $\Sigma(t)$, but instead to $m=500$ noisy measurements $\Sigma_1,\dots, \Sigma_{500}$ of it at equally spaced times $t_1,\ldots, t_{500}$ in the interval $[-1,1]$.
More specifically, we let $\Sigma_i = \Sigma(t_i) + Z_i$, where $Z_i$ is a $10 \times 10$ symmetric matrix with upper triangular entries drawn independently from a Gaussian with mean zero and standard deviation $30$.

In Figure~\ref{fig:markowitz_example}, we compare the variance of four different portfolios with respect to the true covariance matrix $\Sigma(t)$ in~\eqref{eq: sigma t}. The curves correspond to $x_*^T\Sigma(t)x_*$, where $x_*\in\{ x_*^\eqref{eq:sigma_i_formulation}, x_*^{\text{GE}-0}, x_*^{\text{GE}-1}, x_*^{\text{GE}-2} \}$. Here, $x_*^\eqref{eq:sigma_i_formulation}$ is optimal to \eqref{eq:sigma_i_formulation}, and $x_*^{\text{GE}-0}, x_*^{\text{GE}-1}, x_*^{\text{GE}-2}$ are optimal to~\eqref{eq:GE_formulation} with $d=0,1,2$, respectively. Note that to find the latter three portfolios, we first solve~\eqref{eq:fit_Pt} for $d=0,1,2$ to obtain the optimal matrix $P^*(t)$ that goes as input to \eqref{eq:GE_formulation}.
We observe that the GE-based portfolios have lower variance throughout time than the portfolio coming from the solution to~\eqref{eq:sigma_i_formulation}.
In this example, the improvement seems to saturate at degree $d=2$. The worst-case variances of the four portfolios (with respect to the true covariance matrix $\Sigma(t)$) are reported in Table~\ref{table: markowitz}.
\begin{figure}[ht]
\centering
\includegraphics[width=.7\textwidth]{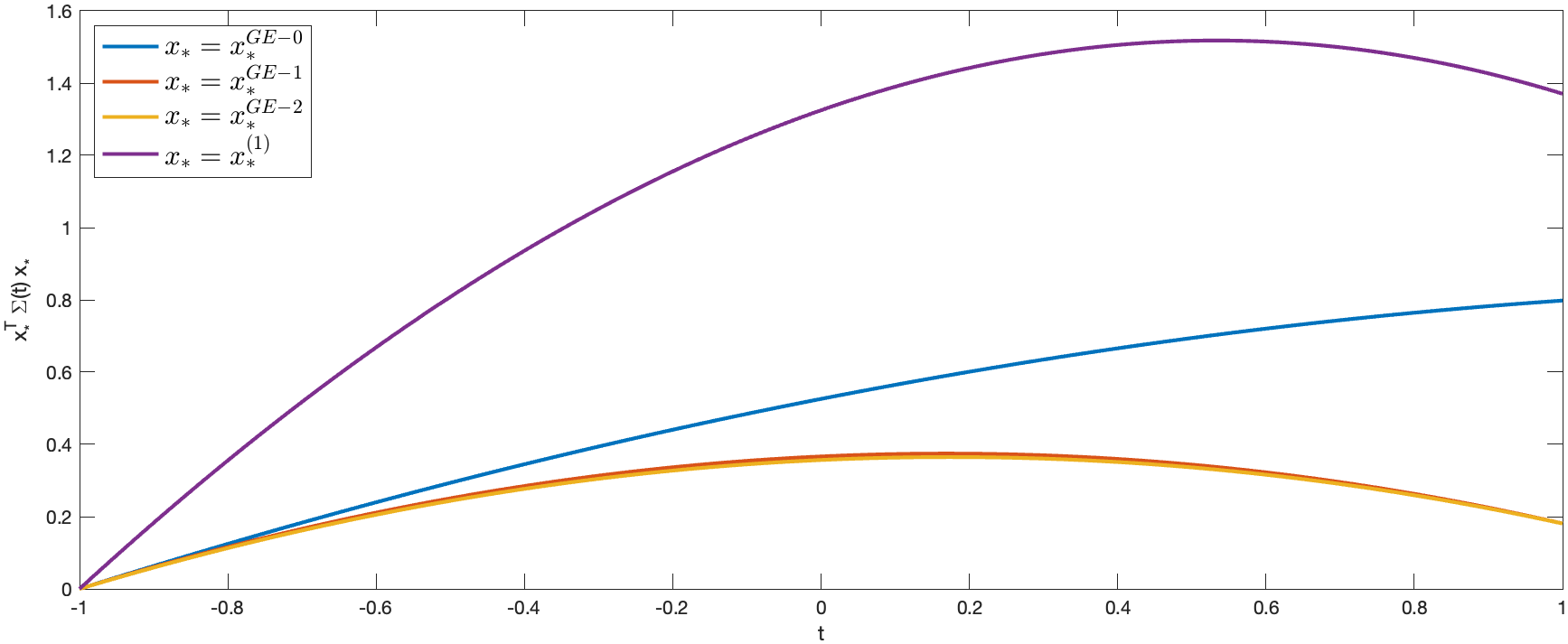}
\caption{Comparing the variance of the four portfolios discussed in Section~\ref{subsubsec:markowitz_example} with respect to the true covariance matrix $\Sigma(t)$.}
\label{fig:markowitz_example}
\end{figure}

\begin{table}[ht]
\centering
\begin{tabular}{| m{2.8cm} || m{2.1cm} | m{2.1cm} | m{2.1cm} | m{2.1cm} ||} 
 \hline
 & \vspace{0.1cm} \hspace{0.15cm} $x_* = x_*^\eqref{eq:sigma_i_formulation}$ & \vspace{0.1cm} \hspace{0.04cm} $x_* = x_*^{\text{GE}-0}$ & \vspace{0.1cm} \hspace{0.03cm} $x_* = x_*^{\text{GE}-1}$ & \vspace{0.1cm} \hspace{0.01cm} $x_* = x_*^{\text{GE}-2}$ \\ [0.3ex] 
 \hline\hline
\vspace{0.16cm} $\max\limits_{t \in [-1,1]} x_*^T \Sigma(t) x_*$ & \vspace{0.26cm} \hspace{0.38cm} 1.5179 & \vspace{0.26cm} \hspace{0.38cm} 0.7981 & \vspace{0.26cm} \hspace{0.38cm} 0.3745 & \vspace{0.26cm} \hspace{0.38cm} 0.3649 \\  [2.4ex] 
 \hline
\end{tabular}
\caption{The worst-case variance $\max\limits_{t \in [-1,1]} x_*^T \Sigma(t) x_*$ for four different portfolios $x_*$.}
\label{table: markowitz}
\end{table}

%%%%%%%%%%%%%%%%%%%%%%%%%%%%%%%%%%%%%%%%%%%%%
\subsection{Joint spectral radius and stability of switched linear systems}\label{subsec:jsr}

% \begin{itemize}
% \itemsep0em
% \item use the same style as in the other paper
% \item if a single matrix, then ellipsoid
% \item let’s also give the quantitative version, combining with the semiellipsoid result – first keep the two theorems
% \item then say this can be quantified
% \end{itemize}

In this section, we show that asymptotically stable switched linear systems always admit a GE as an invariant set. Equivalently, we show that if the joint spectral radius of a set of matrices is less than one, then there must exist a ``contracting'' GE-$d$-norm.

%Equivalently, we show that GEs provide arbitrarily tight upper bounds on the joint spectral radius of a set of matrices. 
%we use GE-$d$-norms to show stability of switched linear systems. This task is equivalent to showing that the joint spectral radius of a family of matrices is less than one. 

Recall that the \emph{spectral radius} $\rho(A)$ of a matrix $A \in \R^{n \times n}$ is defined as 
$$\rho(A)=\lim_{k \rightarrow \infty} ||A^k||^{1/k},$$
where $||\cdot||$ is any matrix norm.
This quantity coincides with the maximum of the absolute values of the eigenvalues of $A$. The discrete-time linear dynamical system $x_{k+1}=Ax_k$, where $x_k \in \R^n$ is the state of the system at time $k\in\mathbb{N}$, is said to be \emph{asymptotically stable} if for any starting state $x_0 \in \mathbb{R}^n$, $x_k \rightarrow 0$ as $k \rightarrow \infty.$ It is straightforward to establish that a linear system is asymptotically stable if and only if $\rho(A)<1$. 
%In 1960, Rota and Strang introduced a generalization of the spectral radius to a \emph{set} of matrices. 

The \emph{joint spectral radius} $\rho(\mathcal{A})$ of a set of matrices $\mathcal{A} \defin \{A_1,\ldots,A_m\} \subseteq \R^{n \times n}$ is defined as
$$\rho(\mathcal{A})\defin\lim_{k\rightarrow \infty} \max_{\sigma \in \{1,\ldots,m\}^k}||A_{\sigma_k} \ldots A_{\sigma_1}||^{1/k},$$
where $||\cdot||$ is any matrix norm~\cite{RotaStrang}. Note that when $\mathcal{A}$ contains a single matrix, the definition of the joint spectral radius (JSR) simplifies to that of the spectral radius. However, computing the JSR is significantly more challenging than computing the spectral radius; for example the problem of testing if $\rho(\mathcal{A}) \leq 1$ is undecidable already when $m=2$~\cite{Blondel2,Blondel1}. The JSR has a close connection to stability of a discrete-time \emph{switched} linear system, i.e., a dynamical system of the type $x_{k+1}=A_{k}x_k$, where the matrix $A_k \in \R^{n \times n}$ can vary arbitrarily in each iteration within the set $\mathcal{A}$. We say that a switched linear system is asymptotically stable if for any starting state $x_0 \in \mathbb{R}^n$ and any sequence of products of matrices in $\mathcal{A}$, $x_k \rightarrow 0$ as $k \rightarrow \infty$. One can show that this property holds if and only if $\rho(\mathcal{A})<1$; see, e.g.,~\cite{JungBook}. Therefore, much research has focused on providing conditions that guarantee the JSR is less than one. Theorem~\ref{thm:GEd_norms_JSR} below shows that GEs can always provide such a condition. This theorem can be seen as a direct generalization of the following classical result in linear systems theory.

\begin{theorem}[see, e.g., Theorem 8.4 in \cite{Hespanha}]
\label{thm:AAA.Jungers}
For a matrix $A \in \mathbb{R}^{n \times n}$, we have $\rho(A)<1$ if and only if there exists a contracting quadratic norm; i.e., a function $V:\mathbb{R}^n \rightarrow \mathbb{R}$ of the form $V(x)=\sqrt{x^TQx}$, with $Q\succ 0$, such that $V(Ax)<V(x)$ $\forall x\neq 0.$
\end{theorem}

Geometrically, the above theorem implies that if a linear system is asymptotically stable, then there is an ellipsoid (given by any sublevel set of the quadratic norm) that is invariant under the trajectories, i.e., all trajectories starting in this ellipsoid remain in the ellipsoid for all time. %It is well known that the existence of such an ellipsoid is no longer necessary for asymptotic stability of switched linear systems involving at least two matrices. The following theorem implies that the existence of an invariant GE is however a necessary condition. 
It is well known, however, that for switched linear systems involving at least two matrices, existence of an invariant ellipsoid is not a necessary condition for asymptotic stability. By contrast, the following theorem implies that existence of an invariant GE is a necessary condition for asymptotic stability of switched linear systems.
The theorem is stated in the language of GE-$d$-norms (recall the definition from~\eqref{eq:GE_d_norm} in Section~\ref{sec:GE_definition}).

\begin{theorem}
\label{thm:GEd_norms_JSR}
For a set of matrices $\mathcal{A} \defin \{A_1,\dots, A_m \} \subseteq \R^{n \times n}$, we have $\rho(\mathcal{A})<1$ if and only if there exists a contracting GE-$d$-norm; i.e., 
%a function $V(x)= \max_{t \in [-1,1]} \sqrt{x^TP(t)x}$, where $P(t)$ is a univariate polynomial matrix of degree $d$ with $P(t) \succeq 0 \:\: \forall t \in [-1,1]$ and $\bigcap_{t \in [-1,1]} \text{Ker}(P(t)) = \{0\}$, such that $V(A_ix)<V(x),~\forall x\neq 0$ and $\forall i=1,\ldots,m.$
%a GE-$d$-norm $V(x)= \max_{t \in [-1,1]} \sqrt{x^TP(t)x}$ such that $V(A_ix)<V(x),~\forall x\neq 0$ and $\forall i=1,\ldots,m.$
a function $V:\mathbb{R}^n \rightarrow \mathbb{R}$ of the form $V(x)= \max_{t \in [-1,1]} \sqrt{x^TP(t)x}$, where $P(t)$ is a univariate polynomial matrix of degree $d$ satisfying the psd and the kernel conditions in Definition~\ref{def:GE_d_defn}, such that $V(A x)<V(x)$ $\forall x\neq 0$ and $\forall A \in \mathcal{A}$.
\end{theorem}

The ``if'' direction of Theorem~\ref{thm:GEd_norms_JSR} follows from Lyapunov's global stability theorem; see, e.g.,~\cite[Section 1]{PathComplete} and references therein (in the language of that paper, our GE-$d$-norm is a ``common'' or ``simultaneous'' Lyapunov function). The ``only if'' direction of Theorem~\ref{thm:GEd_norms_JSR} can be shown by first invoking a nonconstructive converse Lyapunov theorem which states that if $\rho(\mathcal{A}) < 1$, then there exists a contracting norm; see, e.g.,~\cite{RotaStrang} or~\cite[page 24]{JungBook}. This abstract norm however can be approximated arbitrarily well by a GE-$d$-norm. This is a consequence of our Theorem~\ref{thm:approximation_by_GE}, which proves that any symmetric convex body can be approximated arbitrarily well by a GE-$d$. The ``only if'' direction of Theorem~\ref{thm:GEd_norms_JSR} also follows from Theorem~\ref{thm:quantitative_version}, which provides a quantitative version of the statement. This theorem generalizes the main result of~\cite{AndoShih,BlondelNestTheys} (which corresponds to $l=1$) from ellipsoids to generalized ellipsoids.

\begin{theorem}
\label{thm:quantitative_version}
Let $\mathcal{A} \defin \{A_1,\ldots,A_m\} \subseteq \R^{n \times n}$. For a positive integer $\ell$, if $\rho(\mathcal{A}) < \frac{1}{\sqrt[2\ell]{n}}$, then there exists a contracting GE-$d$-norm with $d \leq \max \{ 2m^{\ell-1}-3,0 \}$.
\end{theorem}

\begin{proof}
Suppose $\rho(\mathcal{A}) < \frac{1}{\sqrt[2\ell]{n}}$.
Then it follows from \cite[Theorem~6.1]{PathComplete} that there exist $m^{\ell-1}$ matrices $P_1,\dots,P_{m^{\ell-1}} \in S^n_{++}$ such that the function $W(x) = \max_{i\in \{1,\dots,m^{\ell-1} \}} \sqrt{x^T P_i x}$ satisfies $W(Ax) < W(x)$, $\forall x \neq 0$ and $\forall A \in \mathcal{A}$.
If $m=1$ or $\ell=1$, the GE-$0$-norm $V(x) = \sqrt{x^T P_1 x}$ is evidently contracting.
Now assume $m,\ell \geq 2$.
It follows from Corollary~\ref{cor:ellipsoids and polytopes} that there exists a polynomial matrix $P(t)$ of degree $d \leq 2m^{\ell-1}-3$ such that the GE-$d$-norm $V(x) = \max_{t \in [-1,1]} \sqrt{x^T P(t) x} = W(x)$ for all $x \in \R^n$.
The claim follows.
\end{proof}

%%%%%%%%%%%%%%%%%%%%%%%%%%%
\subsubsection{An example}\label{subsubsec:jsr_example}

Consider the set of matrices $\mathcal{A}_{\gamma} =\{\gamma A_1, \gamma A_2\} \subseteq \R^{2 \times 2}$ parameterized by a scalar $\gamma\geq 0$ and with 
$$A_1=\begin{bmatrix} 1 & 0 \\ 1 & 0\end{bmatrix}, \:\: A_2=\begin{bmatrix} 0 & 1 \\ 0 & -1 \end{bmatrix}.$$
%$\mathcal{A} =\{A_1,\ldots,A_m\}$

% and its scaling $\mathcal{A}_{\gamma} =\{\gamma A_1, \dots, \gamma A_m\}$ parameterized by a positive scalar $\gamma$. 
% Observe that $\rho(\mathcal{A}_{\gamma}) = \gamma \rho(\mathcal{A})$. 
% Hence, if we can show that $\mathcal{A}_{\gamma} \leq 1$, then we would have $\rho(\mathcal{A}) \leq \frac{1}{\gamma}$. 
% Consider now the set of matrices $\mathcal{A} = \{A_1, A_2\}$, with
% $$A_1=\begin{bmatrix} 1 & 0 \\ 1 & 0\end{bmatrix}, \:\: A_2=\begin{bmatrix} 0 & 1 \\ 0 & -1 \end{bmatrix}.$$
These matrices that have been studied e.g. in \cite{AndoShih} to demonstrate that the JSR can be less than one without existence of a contracting quadratic norm (recall the definition from the statement of Theorem~\ref{thm:AAA.Jungers}). Indeed, one can show that $\rho(\mathcal{A_{\gamma}}) < 1$ for any $\gamma < 1$, while a contracting quadratic norm exists only when $\gamma < \frac{1}{\sqrt{2}}$. By contrast, we observe that a contracting GE-$1$-norm exists for any $\gamma < 1$.
 
% Let $\gamma = \frac{1}{1+\epsilon}$ for some $\epsilon > 0$, and let 
% $$P(t) = \frac{1}{2} \begin{bmatrix} 1-t & 0 \\ 0 & 1+t\end{bmatrix}.$$
% It is straightforward to verify that  $P(t) \succeq 0 \:\: \forall t \in [-1,1]$ and that $\bigcap\limits_{t \in [-1,1]} \text{Ker}(P(t)) = \{0\}$. We have
% \begin{equation*}
% \begin{aligned}
% V(x) &= \max\limits_{t \in [-1,1]} \sqrt{x^TP(t)x} = \max\limits_{t \in [-1,1]} \sqrt{\frac{1}{2} \left ((1-t)x_1^2+(1+t)x_2^2 \right)} = \begin{cases} 
%       |x_1| & \text{if} \hspace{0.2cm} |x_1| \geq |x_2|, \\
%       |x_2| & \text{if} \hspace{0.2cm} |x_2| \geq |x_1|.
%    \end{cases} \\
% V(\gamma A_1 x) &= \max\limits_{t \in [-1,1]} \gamma \sqrt{x^TA_1^TP(t)A_1x} =  \frac{1}{1+\epsilon} |x_1|. \\
% V(\gamma A_2 x) &= \max\limits_{t \in [-1,1]} \gamma \sqrt{x^TA_2^TP(t)A_2x} =  \frac{1}{1+\epsilon} |x_2|.
% \end{aligned}
% \end{equation*}
% Now, let $x \neq 0$ and assume without loss of generality that $|x_1| \geq |x_2|$. (Note that $|x_1| > 0$.) Then,
% \begin{equation*}
% \begin{aligned}
% V(x) - V(\gamma A_1 x) &= |x_1| \left(1 - \frac{1}{1+\epsilon} \right) > 0, \\
% V(x) - V(\gamma A_2 x) &= \frac{(1+\epsilon)|x_1| - |x_2|}{1+\epsilon} = \frac{|x_1|-|x_2|+\epsilon |x_1|}{1+\epsilon} > 0.
% \end{aligned}
% \end{equation*}
% This proves stability of $\mathcal{A}_{\gamma}$, i.e., $\rho(\gamma A_1, \gamma A_2) < 1$. Therefore, $\rho(A_1, A_2) < 1+\epsilon$ for any $\epsilon > 0$.

Let
$$P(t) = \frac{1}{2} \begin{bmatrix} 1-t & 0 \\ 0 & 1+t\end{bmatrix}.$$
It is straightforward to verify that  $P(t) \succeq 0 \:\: \forall t \in [-1,1]$ and that $\bigcap\limits_{t \in [-1,1]} \text{Ker}(P(t)) = \{0\}$. Let
$$
V(x) \defin \max\limits_{t \in [-1,1]} \sqrt{x^TP(t)x} = \max\limits_{t \in [-1,1]} \sqrt{ \frac{1-t}{2}x_1^2+\frac{1+t}{2}x_2^2} = \max \{ |x_1|,|x_2|\}
$$
be the associated GE-$1$-norm. We have $V(\gamma A_1 x) = \gamma |x_1|$ and $V(\gamma A_2 x) = \gamma |x_2|$.
% \begin{equation*}
% \begin{aligned}
% V(\gamma A_1 x) &= \max\limits_{t \in [-1,1]} \gamma \sqrt{x^TA_1^TP(t)A_1x} =  \gamma |x_1|. \\
% V(\gamma A_2 x) &= \max\limits_{t \in [-1,1]} \gamma \sqrt{x^TA_2^TP(t)A_2x} =  \gamma |x_2|.
% \end{aligned}
% \end{equation*}
% Now, let $x \neq 0$ and assume without loss of generality that $|x_1| \geq |x_2|$. (Note that $|x_1| > 0$.) Then,
% \begin{equation*}
% \begin{aligned}
% V(x) - V(\gamma A_1 x) &= |x_1| \left(1 - \gamma \right) > 0, \\
% V(x) - V(\gamma A_2 x) &= \frac{\frac{1}{\gamma}|x_1| - |x_2|}{\frac{1}{\gamma}} > 0.
% \end{aligned}
% \end{equation*}
%This proves $\rho(\mathcal{A}_{\gamma}) < 1$. Therefore, $\rho(\mathcal{A}) < \frac{1}{\gamma}$.
It follows that $V(\gamma A_i x) < V(x)$, $\forall \gamma < 1$, $\forall x \neq 0$, and for $i=1,2$.

%This shows that the GE-$1$-norm $V(x)$ is contracting for any $\gamma < 1$.

%%%%%%%%%%%%%%%%%%%%%%%%%%%%%%%%%%%%%%%%%%%%%

\subsection{Robust-to-dynamics optimization}
\label{subsec:RDO}
A robust-to-dynamics optimization (RDO) problem is an optimization problem of the form
%consists of (i) an optimization problem with an objective function $f: \R^n \rightarrow \R$ and a feasible region $\Omega \subseteq \R^n$, and (ii) a dynamical system encoded by a map $g: \R^n \rightarrow \R^n$ \cite{RDO}. With these inputs one wishes to solve the problem
\begin{equation}\label{eq: rdo}
\hspace{5.3cm}
\begin{aligned}
\min_{x \in \R^n} \quad & f(x)\\
\text{s.t.}\quad & x,g(x), g(g(x)),\ldots \in \Omega,
\end{aligned}    
\end{equation}
where $f: \R^n \rightarrow \R$, $\Omega\subseteq \R^n$, and $g: \R^n \rightarrow \R^n$ is a map that represents a dynamical system $x_{k+1}=g(x_k)$ with $k=0,1,2,\ldots$ denoting the index of time. In words, the goal of the RDO problem is to optimize $f$ over the set $\mathcal{S} \subseteq \Omega$ of initial conditions that forever remain in $\Omega$ under $g$. We refer the reader to~\cite{RDO} for more context and to~\cite{ACST} for applications of this problem to safe learning.

In~\cite{RDO}, algorithms that provide tractable inner and outer approximations to the feasible set $\mathcal{S}$ of~\eqref{eq: rdo} are provided for certain subclasses of the RDO problem. A particular focus is on the case where $\Omega$ is a polyhedron and $g$ is a linear map. %; i.e., $g(x)=Ax$ for some matrix $A \in \R^{n\times n}$.
More specifically, in this setting, $\Omega =  \{ x \in \R^n \mid Hx \leq 1\}$ where $H \in \R^{m \times n}$ is a given matrix and $g(x) = Ax$ where $A \in \R^{n \times n}$ is \emph{stable}\footnote{A terminology more consistent with Section~\ref{subsec:jsr} would have been ``asymptotically stable'', but in this section we drop the word asymptotic for simplicity.}, i.e. has spectral radius less than one.
%all of its eigenvalues have absolute value at most one.
Note that any polytope with the origin in its interior can be written in the form of $\Omega$.
We refer the reader to \cite[Section~2.1.1]{RDO} (and also \cite[Proposition~16]{ACST}) to see why the assumptions that $\Omega$ contains the origin in its interior and that $A$ is stable are made. These assumptions are only slightly stronger than the natural requirement that $\mathcal{S}$ is not a measure-zero set.

In this section,  we extend this setting to the case where the matrix $A$ is unknown, but must belong to the convex hull of two given matrices $\hat{A}$ and $\check{A}$. 
The input to our problem is a matrix $H \in \R^{m \times n}$ representing the polytope $\Omega =  \{ x \in \R^n \mid Hx \leq 1\}$ and two matrices $\hat{A},\check{A} \in \R^{n \times n}$.
%with the assumption that every matrix in their convex hull is stable.
Given this input, we wish to characterize the set
\begin{equation}\label{eq:uncertain_RDO}
\hspace{2.2cm}
\mathcal{S} \defin \{x \in \R^n \mid HA^k x \leq 1 \quad k=0,1,\dots, \quad \forall A \in \text{conv}(\hat{A},\check{A}) \}.
\end{equation}
The approach that we present will also work in the more general setting where the matrix $A$ is only known to belong to a given polynomial curve in matrix space.
As is done in \cite{RDO}, outer approximations to $\mathcal{S}$ can be obtained by truncating the infinite time horizon, and in our case sampling from $\text{conv}(\hat{A},\check{A})$; we carry out an outer approximation of this form in the numerical example below.
As is the case in~\cite{RDO}, finding inner approximations to $\mathcal{S}$ are more challenging however.
The following theorem shows how one can inner approximate $\mathcal{S}$ with a GE via semidefinite programming.

\begin{theorem}\label{thm:rdo ge}
Let $\mathcal{S}$ be as in \eqref{eq:uncertain_RDO}. Let $A(t) \defin \frac{1+t}{2}\hat{A} + \frac{1-t}{2} \check{A}$ and $h_i^T$ be the $i^{\text{th}}$ row of $H$. If $\mathcal{E}_d$ is a GE-$d$ defined by a polynomial matrix $P(t)$ of degree $d$ which satisfies the constraints
\begin{equation}\label{eq:rdo contraints}
\begin{aligned}
\hspace{4cm}
% &P(t) \succeq 0 \quad \forall t \in [-1,1]\\
&P(t) - A(t)^TP(t)A(t) \succeq 0 \quad \forall t \in [-1,1]\\
&P(t) \succeq h_i h_i^T \quad \forall t \in [-1,1] \quad i=1,\dots,m,
\end{aligned}
\end{equation}
then $\mathcal{E}_d \subseteq \mathcal{S}$.
\end{theorem}
\begin{proof}
For each $t \in [-1,1]$, define $\mathcal{E}_d(t) \defin \{x \in \R^n \mid x^T P(t) x \leq 1\}$.
By the first constraint in \eqref{eq:rdo contraints}, we have $x \in \mathcal{E}_d(t) \Rightarrow A(t)x \in \mathcal{E}_d(t)$.
From this, it follows that $x \in \mathcal{E}_d(t) \Rightarrow A^k(t) x \in \mathcal{E}_d(t)$ for all $k\geq 0$.
By the second constraint in \eqref{eq:rdo contraints}, we have $x \in \mathcal{E}_d(t) \Rightarrow Hx \leq 1$.
Then we have
$$
x \in \mathcal{E}_d = \bigcap_{t \in [-1,1]}\mathcal{E}_d(t) \Rightarrow HA^k(t) x \leq 1 \quad \forall t \in[-1,1] \quad k=0,1,\dots,
$$
which gives the desired result since the set $\{A(t) \mid t \in [-1,1]\} = \text{conv}(\hat{A},\check{A})$.
\end{proof}

We note that if $\Omega$ is compact and if any matrix in $\text{conv}(\hat{A},\check{A})$ has spectral radius more than one, then the set $\mathcal{S}$ in \eqref{eq:uncertain_RDO} will have measure zero \cite[Proposition~16]{ACST}.
To avoid this situation, similarly to what is done in \cite{RDO}, we work with the assumption that all matrices in $\text{conv}(\hat{A},\check{A})$ are stable. Under this assumption, the following lemma ensures that there is always a suitable polynomial matrix $P(t)$ which satisfies the constraints of~\eqref{eq:rdo contraints}. The second constraint in~\eqref{eq:rdo contraints} is not mentioned in this lemma since it can always be satisfied simply by scaling up the matrix $P(t)$.
In the language of dynamical systems, this lemma states that if all matrices in the convex hull are stable, then there must exist a polynomially-varying quadratic Lyapunov function $x^TP(t)x$ for the associated linear dynamical systems.

%Note that this lemma, in fact, ensures the existence of a positive definite matrix $P(t)$ which can then be scaled up to satisfy the second constraint of \eqref{eq:rdo contraints}.
% Note that if $P(t)$ satisfies the first constraint of \eqref{eq:rdo contraints} and is positive definite for all $t \in [-1,1]$, then it can be scaled up to also satisfy the second constraint.
% \begin{lemma}[Special case of Lemma~24 of \cite{ACST}]
% Let $A: \R \rightarrow \R^{n \times n}$ be a polynomial matrix.
% Then, every matrix in the set $\{A(t) \mid t \in [-1,1] \}$ is stable if and only if there exists an $n \times n$ sos-matrix $P:\R \rightarrow S^{n}$ such that
% \begin{enumerate}
%     \item $P(t) \succ 0 \quad \forall t \in [-1,1]$,
%     \item $P(t) - A(t)^T  P(t) A(t) \succ 0 \quad \forall t \in [-1,1]$.
% \end{enumerate}
% \end{lemma}

\begin{lemma}[Special case of Lemma~24 of \cite{ACST}]\label{lem: safelearning stable}
For two matrices $\hat{A},\check{A} \in \R^{n \times n}$,
every matrix in the set $\text{conv}(\hat{A},\check{A})$ is stable if and only if there exists a polynomial matrix $P:\R \rightarrow S^{n}$ such that
\begin{enumerate}
    \item $P(t) \succ 0 \quad \forall t \in [-1,1]$,
    \item $P(t) - A(t)^T  P(t) A(t) \succ 0 \quad \forall t \in [-1,1]$,
\end{enumerate}
where $A(t) \defin \frac{1+t}{2}\hat{A} + \frac{1-t}{2} \check{A}$.
\end{lemma}

% Note that we need not distinguish between polynomials of $A$ and polynomials of $t$, since in the setting of Theorem~\ref{thm:rdo ge}, $A$ itself varies polynomially with $t$.

% Observe that we can search for a polynomial matrix $P(t)$ of degree $d$ satisfying \eqref{eq:rdo contraints} via semidefinite programming.

We have just shown that the following optimization problem
\begin{equation}\label{eq: rdo radius}
\hspace{3.2cm}
\begin{aligned}
\min_{P \in S_d^n[t],\: \gamma \in \R} \quad & \gamma \\
\text{s.t.} \quad & P(t) \preceq \gamma I  \quad \forall t \in [-1,1] \\
&P(t) \succeq 0 \quad \forall t \in [-1,1]\\
&P(t) - A(t)^TP(t)A(t) \succeq 0 \quad \forall t \in [-1,1]\\
&P(t) \succeq h_i h_i^T \quad \forall t \in [-1,1] \quad i=1,\dots,m,
\end{aligned}
\end{equation}
where $A(t) \defin \frac{1+t}{2}\hat{A} + \frac{1-t}{2} \check{A}$,
is feasible for a polynomial matrix $P(t)$ of sufficiently large degree.
Note that~\eqref{eq: rdo radius} is a semidefinite program; see Section~\ref{subsec:efficient_search}.
The GE-$d$ associated with the polynomial matrix $P(t)$ is an inner approximation to the set $\mathcal{S}$ defined in \eqref{eq:uncertain_RDO}. The objective and the first constraint in this SDP are maximizing the radius of a ball contained in this GE-$d$.
%For example, to maximize the radius of a ball contained in the GE-$d$ defined by $P(t)$, one can solve the following semidefinite program

% One can check that this program maximizes the radius of a ball contained in the GE-$d$ defined by $P(t)$.
\subsubsection{Numerical example}\label{sec: rdo example}
As an example, we seek to characterize the set $\mathcal{S}$ as defined in \eqref{eq:uncertain_RDO} with the following input:
$$
H = \begin{bmatrix}
    -1 & 0\\1 & 0\\0&-1\\0&1
\end{bmatrix},\quad \hat{A} = \begin{bmatrix}
       -0.9  &  0.6 \\
   -1.6  &  1.1
\end{bmatrix},\quad \check{A} = \begin{bmatrix}
    1.1  &  0.6\\
   -1.6  & -0.9
\end{bmatrix}.
$$
We solve the SDP in \eqref{eq: rdo radius} to find a polynomial matrix $P(t)$ of degree $d$.
For $d=0$, the problem is infeasible.
For $d=1,2$, the problem is feasible and we denote the GEs defined by the optimal solutions by $\mathcal{E}_1,\mathcal{E}_2$, respectively.
In Figure~\ref{fig:rdo}, we plot the sets $\Omega,\mathcal{E}_1,\mathcal{E}_2$, as well as the set
$$
S_{10} \defin \left \{x \in \R^n \mid HA^k(t) x \leq 1 \quad \forall t \in \left\{ -1,-\frac{9}{10},\dots,\frac{9}{10},1 \right\} \quad k=0,1,\dots 10 \right \},
$$
which is clearly an outer approximation of the set $\mathcal{S}$.
Since $\mathcal{S}$ must be sandwiched between the sets $\mathcal{E}_2$ and $S_{10}$, we can conclude that $\mathcal{E}_2$ provides a good inner approximation to $\mathcal{S}$.
%We observe that $\mathcal{E}_2$ is not much smaller than $S_{10}$; thus $\mathcal{E}_2$ must be a good inner approximation of the set $\mathcal{S}$.
\begin{figure}[ht]
    \centering
    \includegraphics[width = .5\textwidth]{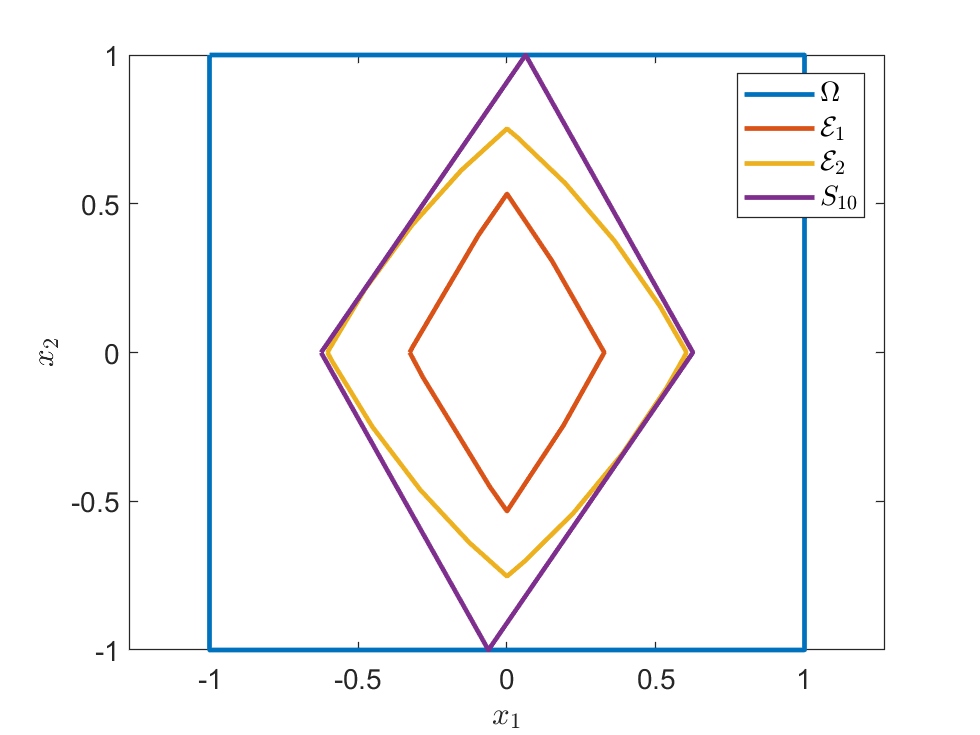}
    \caption{The sets associated with the numerical example in Section~\ref{sec: rdo example}. The GE-2 denoted by $\mathcal{E}_2$ provides a good inner approximation of the set $\mathcal{S}$ defined in \eqref{eq:uncertain_RDO}.}
    \label{fig:rdo}
\end{figure}

%%%%%%%%%%%%%%%%%%%%%%%%%%%%%%%%%%%%%%%%%%%%%

\subsection{Polynomial regression robust to a shift}\label{subsec:shift regression}

A classic task in statistics is to fit a polynomial function $p: \R \rightarrow \R$ to observations $(x_i,y_i)$ for $i=1,\dots,m$. A standard approach to find a polynomial fit of degree (at most) $d$ is that of least-squares polynomial regression, which solves the problem
\begin{equation}\label{eq: least squares}
\hspace{5.5cm}
\min_{c \in \R^{d+1}} \left\| \Phi(x)c - y \right\|^2,
\end{equation}
where $c \in \R^{d+1}$ is the vector of coefficients of $p$, the vectors $x,y \in \R^m$ have their $i^{\text{th}}$ entry equal to $x_i,y_i$, respectively, and $\Phi: \R^m \rightarrow \R^{m \times (d+1)}$ is the polynomial matrix with the $i^{\text{th}}$ row of $\Phi(z)$ equal to $(1,z_i,z_i^2,\dots,z_i^d)$.
In some applications, due to metrological limitations, one may have some error in measuring the points $x_i$.
In particular, one may wish to find a function which provides a good fit to the observations even if the points $x_i$ were slightly shifted.
In this section, we describe how GEs arise when solving this problem.

More concretely, to find a polynomial that fits the observations well even if the points $x_i$ are shifted by up to $\varepsilon$ units to the left or right, one can write
\begin{equation}\label{eq: robust regression}
\hspace{5cm}
\min_{c \in \R^{d+1}} \max_{t \in [-1,1]} \left\| \Phi(x+\varepsilon t)c - y \right\|^2.
\end{equation}
This problem can be reformulated as
\begin{equation}
\hspace{3.5cm}
\begin{aligned}
\min_{c \in \R^{d+1}, \gamma \in \R} \quad &\gamma\\
\text{s.t.} \quad & \begin{bmatrix}
    c \\ 1
\end{bmatrix}^T P(t) \begin{bmatrix}
    c \\ 1
\end{bmatrix} \leq \gamma \quad \forall t\in [-1,1],
\end{aligned}
\end{equation}
for $$P(t) = \begin{bmatrix}
    \Phi(x+\varepsilon t)^T\Phi(x+\varepsilon t) & -\Phi(x+\varepsilon t)^T y \\ -y^T\Phi(x+\varepsilon t) & y^Ty
\end{bmatrix}.$$
Note that $P(t) \succeq 0$ for all $t \in [-1,1]$ and it is straightforward to check that $P(t)$ satisfies the kernel condition under the mild assumption that there are a pair of observations $(x_i,y_i)$ and $(x_j,y_j)$ in the dataset such that $x_i\neq x_j$ and $y_i \neq y_j$.
Therefore, problem \eqref{eq: robust regression} corresponds to finding a coefficient vector $c$ such that the appended vector $[c,1]^T$ is minimal with respect to the GE-$2d$-norm defined by $P(t)$ (see~\eqref{eq:GE_d_norm} in Section~\ref{sec:GE_definition}).
By Theorem~\ref{thm:sdp_representable}, this problem can be reformulated as an SDP.

\subsubsection{Numerical example}\label{sec:regression_example}
For our experiment, we take $m=10$, the points $x_i$ to be uniformly spaced between $-1$ and $1$, and $y_i = f(x_i)$ where $f$ is the Runge function $f(x) \defin \frac{1}{1+25x^2}$.
We fit a degree-$9$ polynomial to these observations both with the standard least-squares approach and with the shift-robust approach.
We take the shift tolerance $\varepsilon$ to be equal to $0.05$.
In Figure~\ref{fig:shift regression}, we plot the observations $(x_i,y_i)$ as well as $(x_i \pm 0.05, y_i)$.
We also plot the polynomials corresponding to the solutions of \eqref{eq: least squares} and \eqref{eq: robust regression}, labelled as $p_{LS}$ and $p_{GE}$, respectively.
We can calculate the worst-case errors for these polynomials:
\begin{align*}
\hspace{4.1cm}
\max_{t \in [-1,1]} \left\| p_{LS}(x+0.05 t) - y \right\|^2 &= 0.8086,\\
\max_{t \in [-1,1]} \left\| p_{GE}(x+0.05 t) - y \right\|^2 &= 0.0447.
\end{align*}
Here we see that $p_{GE}$ is significantly more robust to a small shift of the points and is overall much smoother than $p_{LS}$.
\begin{figure}[ht]
\centering
\includegraphics[width=.5\textwidth]{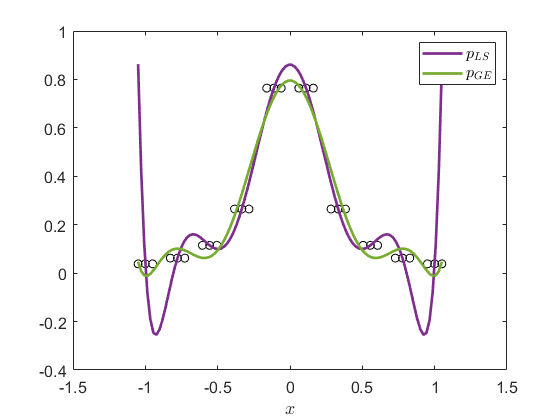}
\caption{Observations and fitted polynomials associated with the numerical example in Section~\ref{sec:regression_example}.}
\label{fig:shift regression}
\end{figure}

%%%%%%%%%%%%%%%%%%%%%

\section{Future Research Directions}\label{sec:future_directions}
We conclude with a few questions for future research.
% In this work, we extend the concept of ellipsoids to a broader family of sets, maintaining essential geometric and algebraic characteristics of ellipsoids while ensuring they remain viable for efficient recognition and optimization.
% In Section~\ref{sec:rep_power}, we show that every convex body can be approximated arbitrarily well by a GE. A related question that we do not address in this paper is: How well can an arbitrary convex body be approximated using a finite intersection of ellipsoids, instead of a GE? 
Our first two questions concern extensions of some results from Section~\ref{sec:rep_power}.

\begin{itemize}
\item We showed that every symmetric convex body can be approximated arbitrarily well by a GE. Our approximation factor relies on results on polytopic approximation of convex bodies. A related question is: how well can one approximate an $n$-dimensional symmetric convex body with a finite intersection of $m$ (co-centered) ellipsoids, with $m$ growing potentially with~$n$? Since we have shown that a GE-$d$, with $d=2m-3$, can exactly represent an intersection of $m$ co-centered ellipsoids (see~Corollary~\ref{cor:ellipsoids and polytopes} and~Theorem~\ref{thm:semiellipsoid}), progress on this question can potentially improve our approximation factor in Theorem~\ref{thm:approximation_by_GE}.

%GEs can represent finite intersections of ellipsoids

%by using results on polytopic approximation.A related question is: how well can one approximate an arbitrary symmetric convex body with a finite intersection of ellipsoids? In particular, can one approximate an arbitrary $n$-dimensional symmetric convex body better with an intersection of $m$ ellipsoids rather than an intersection of $m$ halfspaces? An affirmative answer to this question would immediately lead to a better approximation result for GEs in view of Corollary~\ref{cor:ellipsoids and polytopes}.

% \item In Theorem~\ref{thm:sdp_representable}, we showed that every GE is SDP-representable.
% It is natural to ask if the converse holds.
% Namely, is every SDP-representable symmetric convex body a GE?
% If not, it is also natural to ask: under which conditions is an SDP-representable symmetric convex body a GE?

\item Let us call a set $\Omega \subseteq \R^n$ \emph{GE-$d$-representable} if for some nonnegative integers $k,m$, a GE-$d$ $\mathcal{E}_d \subset \R^m$, some matrices $A \in \R^{m \times n}$ and $B \in \R^{m \times k}$, and some vector $b \in \R^m$, one can write
$$
\Omega = \{x \in \R^n \mid \exists u \in \R^k \text{ s.t. } Ax + Bu + b \in \mathcal{E}_d \}.
$$
We say that a set is \emph{GE-representable} if it is GE-$d$-representable for some nonnegative integer $d$.
It would be interesting to study the expressiveness of GE-representable sets; in particular, do GE-representable sets fall in between SOCP and SDP-representable sets?

\end{itemize}

% \noindent Finally, we believe it would be interesting to study questions in optimization and about ellipsoids also for generalized ellipsoids. 

% Finally, extend theorems about ellipsoids in different contexts to generalized ellipsoids. We give two examples.

\noindent We also highlight some results (among many) about ellipsoids which we believe might be interesting to extend to generalized ellipsoids.

\begin{itemize}

\item Motivated by problems in subspace identification and factor analysis, the ellipsoid fitting conjecture~\cite{SCPW12, SPW13} concerns the maximum number of independent standard Gaussian vectors in $\R^n$ such that with high probability, there exists an ellipsoid (i.e., a GE-$0$), passing through them. Recently, great progress has been made on this problem which resolves the conjecture up to a constant~\cite{TulWu,BMMP,HKPX}.
%give a proof based on concentration of sample covariance matrices, that with high probability, it is possible to fit an ellipsoid through $n^2/C$ random Gaussian points, where $C > 4$ is a universal constant.
How does this maximum number change when one replaces a GE-$0$ with a GE-$d$ for a fixed value of $d$?
%As a generalization to higher degrees, we consider for a given fixed $d$, what is the maximum number of independent standard Gaussian vectors in $\R^n$ such that, with high probability, there exists a GE-$d$ passing through all the points?
% Note that $f_0(n)$ corresponds to the ellipsoid case, and we look for higher degrees, i.e., $f_1(n)$, $f_2(n)$, etc.
% Of course, we have $f_i(n) \leq f_{i+1}(n)$.

\item In \cite{NRT}, it is shown that the standard SDP relaxation for the (nonconvex) problem of maximizing an arbitrary homogeneous quadratic function over the intersection of $m$ ellipsoids provides an approximation ratio of $\frac{1}{2\log (2m^2)}$. Can one derive a similar result for maximization of quadratic functions over a GE-$d$ and obtain an approximation ratio in terms of $d$? Note that one cannot directly apply the result of \cite{NRT} since for $d \geq 2$, a GE-$d$ can be the intersection of an infinite number of ellipsoids.

%one can approximate the maximum norm over the intersection of ellipsoids with an SDP relaxation; the approximation ratio proven is $O(\frac{1}{\log (m)})$ where $m$ is the number of ellipsoids. Can one derive a similar method to approximate the maximum norm over a GE?
\end{itemize}

\noindent Finally, it would be interesting to study generalizations of GEs defined by a matrix function which is not necessarily polynomial. For example, one could study rational functions:

\begin{itemize}
\item If we consider sets of the form  $\Omega = \{x \in \R^2 \mid x^T R(t) x \leq 1 \:\: \forall t \in [-1,1]\},$ where the matrix $R(t)$ is positive semidefinite over $[-1,1]$ and has entries which are \emph{rational functions}, then such sets would be generalizations of GEs. One can extend our proof of Theorem~\ref{thm:sdp_representable} to show that these sets are still SDP-representable. Indeed, for $R(t)$ to be well-defined over $[-1,1]$, the entries of this matrix should have no poles in this domain. Moreover, after possible removal of common terms between the numerators and denominators, we can assume (without loss of generality) that all denominators are positive definite over this domain. 
%and given the expressivity of rational functions, they might be advantageous for some applications.
%While many of our theorems will not readily extend to this case,
%we can show that such sets would be SDP-representable.
% Let $R(t)$ be a matrix of rational functions such that $R(t) \succeq 0$ for all $t \in [-1,1]$.
% First, for all entries of $R(t)$, let us cancel any common terms between the numerator and denominator.
% Since $R(t)$ must be defined for all $t \in [-1,1]$, we must have that the entries of $R(t)$ have no poles over this same domain.
% Thus, the polynomial denominators of each entry of $R(t)$ are nonzero for all $t \in [-1,1]$.
% By flipping the sign of both the numerator and denominator, we can ensure that all of the denominators are positive for $t \in [-1,1]$.
Hence, after taking a common denominator, one can write $R(t) = \frac{P(t)}{q(t)}$, where the polynomial $q(t)$ is positive over $[-1,1]$.
% we can find a common polynomial denominator, $q(t)$, for all entries of $R(t)$ such that $q(t) >0, \:\: \forall t \in [-1,1]$.
% In this way we can find a polynomial matrix $P(t)$ such that $R(t) = \frac{P(t)}{q(t)}$.
Hence,
$$
\Omega = \{x \in \R^2 \mid x^T P(t) x \leq q(t) \:\: \forall t \in [-1,1]\}.
$$
Since $P(t)$ is positive semidefinite over $t \in [-1,1]$, by a similar argument as in the proof of Theorem~\ref{thm:sdp_representable}, it follows that
$$
\Omega = \left \{ x \in \R^n \:\: \middle | \:\: \exists X \in S^n \text{ s.t. } \begin{bmatrix} X & x \\ x^T & 1\end{bmatrix} \succeq 0, \quad \text{Tr}(X P(t)) \leq q(t) \:\:\: \forall t \in [-1,1] \right\}.
$$
In view of Proposition~\ref{prop:sdp_transformation}, this shows that $\Omega$ is SDP-representable. It would be interesting to quantify how well convex bodies can be approximated by such a generalization of GEs.

%Which other results will generalize to the rational function case and what sort of advantages can the generality of rational functions provide?

\end{itemize}

%%%%%%%%%%%%%%%%%%%%%

\section*{Acknowledgements}
We are grateful to Pablo Parrilo for insightful discussions around Lemma~\ref{lem:explicit_Nash}. We also extend our gratitude to two anonymous referees and the Associate Editor for many insightful comments and questions.

% \section*{Appendix}

% \subsection*{Equivalent descriptions of a GE}

% $$S = \bigcap\limits_{t \in [-1,1]} \{x \in \R^n \: | \:  x^TP(t)x \leq 1\}$$

% $$S = \{x \in \R^n \: | \:  \max\limits_{t \in [-1,1]} x^TP(t)x \leq 1\}$$

% $$S = \{x \in \R^n \: | \:  \max\limits_{t \in [-1,1]} \sqrt{x^TP(t)x} \leq 1\}$$

% $$S = \bigcap\limits_{t \in [-1,1]} \{x \in \R^n \: | \:  \sqrt{x^TP(t)x} \leq 1\}$$

% $$S = \{x \in \R^n \: | \: x^TP(t)x \leq 1 \quad \forall t \in [-1,1]\}$$

% $$S = \{x \in \R^n \: | \:  1 - x^TP(t)x \geq 0 \quad \forall t \in [-1,1]\}$$

% $$S = \{x \in \R^n \: | \: 
% \begin{bmatrix}
%      1 & x^T \\
%      x & P(t)^{-1}
% \end{bmatrix} \succeq 0 \quad \forall t \in [-1,1]\}$$

% $$S = \{x \in \R^n \: | \: \left |\left | \sqrt{P(t)}x \right | \right| \leq 1 \quad \forall t \in [-1,1]\}$$

%%%%%%%%%%%%%%%%%%%%%%%%%%%%%%%%%%%%%%%%%%%%%%%


\begin{thebibliography}{100}

\bibitem{abramowitz_stegun}
M.~Abramowitz, I.~Stegun, ``Handbook of Mathematical Functions with Formulas, Graphs, and Mathematical Tables'', {\it Dover Publications, New York}, (1974).

\bibitem{ACST}
A.A.~Ahmadi, A.~Chaudhry, V.~Sindhwani, S.~Tu, ``Safely learning dynamical systems'', preprint, \url{https://arxiv.org/abs/2305.12284}, (2024).

\bibitem{PolyNorms}
A.A.~Ahmadi, E.~De Klerk, G.~Hall, ``Polynomial norms'', {\it SIAM Journal on Optimization}, 29(1) (2019), 399--422.

\bibitem{TV-SDP}
A.A.~Ahmadi, B.~El Khadir, ``Time-varying semidefinite programs'', {\it Mathematics of Operations Research}, 46(3) (2021), 1054--1080.

\bibitem{RDO}
A.A.~Ahmadi, O.~Günlük, ``Robust-to-dynamics optimization'', To appear in {\it Mathematics of Operations Research}, \url{https://doi.org/10.1287/moor.2023.0116}, (2024).

\bibitem{AHMS}
A.A.~Ahmadi, G.~Hall, A.~Makadia, V.~Sindhwani, ``Geometry of 3D environments and sum of squares polynomials'', In {\it Robotics: Science and Systems}, (2017).

%\bibitem{AhmJung}
%A.A.~Ahmadi, R.M.~Jungers, ``Switched stability of nonlinear systems via SOS-convex Lyapunov functions and semidefinite programming'', In {\it Proceedings of the 52nd IEEE Conference on Decision and Control}, IEEE Press, Piscataway, NJ, (2013), 727--732.

\bibitem{PathComplete}
A.A.~Ahmadi, R.M.~Jungers, P.A.~Parrilo, M.~Roozbehani, ``Joint spectral radius and path-complete graph Lyapunov functions'', {\it SIAM Journal on Control and Optimization}, 52(1) (2014), 687--717.

\bibitem{AOPT}
A.A.~Ahmadi, A.~Olshevsky, P.A.~Parrilo, J.N.~Tsitsiklis, ``NP-hardness of deciding convexity of quartic polynomials and related problems'', {\it Mathematical Programming}, 137(1-2) (2013), 453--476.

\bibitem{AndoShih}
T.~Ando, M.-H.~Shih, ``Simultaneous contractibility'', {\it SIAM Journal on Matrix Analysis and Applications}, 19(2) (1998), 487--498.

\bibitem{aylward2007explicit}
E.~Aylward, S.~Itani, P.A.~Parrilo, ``Explicit SOS decompositions of univariate polynomial matrices and the Kalman-Yakubovich-Popov lemma'', In {\it Proceedings of the 46th IEEE Conference on Decision and Control}, (2007), 5660--5665.

\bibitem{BMMP}
A.S.~Bandeira, A.~Maillard, S.~Mendelson, E.~Paquette, ``Fitting an ellipsoid to a quadratic number of random points'', preprint, \url{https://arxiv.org/abs/2307.01181}, (2023).

\bibitem{barvinok_book}
A.~Barvinok, ``A Course in Convexity'', Graduate Studies in Mathematics, Volume 54, {\it American Mathematical Society}, 2002.

\bibitem{barvinok}
A.~Barvinok, ``Thrifty approximations of convex bodies by polytopes'', In {\it International Mathematics Research Notices}, (2014), 4341--4356.

\bibitem{BTNem}
A.~Ben-Tal, A.~Nemirovski, ``Lecture Notes on Modern Convex Optimization'', {\it Society for Industrial and Applied Mathematics}, 2001.

\bibitem{Blondel1}
V.D.~Blondel, V.~Canterini, ``Undecidable problems for probabilistic automata of fixed dimension'', {\it Theory of Computing Systems}, 36 (2003), 231--245.

\bibitem{BlondelNestTheys}
V.D.~Blondel, Y.~Nesterov, J.Theys, ``On the accuracy of the ellipsoid norm approximation of the joint spectral radius'', {\it Linear Algebra and its Applications}, 394 (2005), 91--107.

\bibitem{Blondel2}
V.D.~Blondel, J.N.~Tsitsiklis, ``The boundedness of all products of a pair of matrices is undecidable'', {\it Systems and Control Letters}, 41 (2000), 135--140.

\bibitem{boyd_book}
S.~Boyd, L.~Vandenberghe, ``Convex Optimization'', {\it Cambridge University Press, New York}, (2004).

\bibitem{choi}
M.~D.~Choi, ``Positive semidefinite biquadratic forms'', {\it Linear Algebra and its Applications}, 12 (1975), 95--100.

\bibitem{ChoiLamRez80}
M.D.~Choi, T.Y.~Lam, B.~Reznick, ``Real zeros of positive semidefinite forms. I'', {\it Mathematische Zeitschrift}, 171 (1) (1980), 1--26.

\bibitem{ChoiLamRez}
M.D.~Choi, T.Y.~Lam, B.~Reznick, ``Sums of squares of real polynomials'',  In {\it Proceedings of Symposia in Pure Mathematics}, 58 (1995), 103--126.

\bibitem{debreu}
G.~Debreu, ``A social equilibrium existence theorem'', In {\it Proceedings of the National Academy of Sciences of the United States of America}, 38(10) (1952), 886--893.

\bibitem{DetStud}
H.~Dette, W.J.~Studden, ``Matrix measures, moment spaces and Favard’s theorem for the interval $[0,1]$ and $[0, \infty)$'', {\it Linear Algebra and its Applications}, 345(1-3) (2002), 169--193.

\bibitem{fawzi}
H.~Fawzi, ``On representing the positive semidefinite cone using the second-order cone'', {\it Mathematical Programming}, 175 (2019), 109--118.

% \bibitem{robust_markowitz}
% L.~Ghaoui, M.~Oks, F.Oustry, ``Worst-Case Value-at-Risk and Robust Portfolio Optimization: A Conic Programming Approach'', {\it Operations Research}, 51(4) (2003), 543--556.

\bibitem{GJK}
E.G.~Gilbert, D.W.~Johnson, S.S.~Keerthi, ``A fast procedure for computing the distance between complex objects in three-dimensional space'', {\it IEEE Journal on Robotics and Automation}, 4(2) (1988), 193--203.

\bibitem{Grotschel1988}
M.~Gr\"otschel, L.~Lov\'asz, A.~Schrijver, ``Geometric Algorithms and Combinatorial Optimization'', {\it Springer-Verlag, Berlin}, (1988).

\bibitem{HarrisParrilo}
M.T.~Harris, P.A.~Parrilo, ``Improved nonnegativity testing in the Bernstein basis via geometric means'', preprint, \url{https://arxiv.org/abs/2309.10675}, (2023).

\bibitem{HeltonNie}
J.W.~Helton, J.~Nie, ``Semidefinite representation of convex sets'', {\it Mathematical Programming}, 122(1) (2010), 21--64.

\bibitem{Hespanha}
J.P.~Hespanha, ``Linear Systems Theory'', {\it Princeton University Press}, 2009.

\bibitem{HKPX}
J.T.~Hsieh, P.K.~Kothari, A.~Potechin, J.~Xu, ``Ellipsoid fitting up to a constant'', In {\it 50th International Colloquium on Automata, Languages, and Programming (ICALP 2023), Leibniz International Proceedings in Informatics (LIPIcs)}, Schloss Dagstuhl – Leibniz-Zentrum für Informatik, 261 (2023), 78:1--78:20.

\bibitem{JungBook}
R.~Jungers, ``The Joint Spectral Radius: Theory and Applications'', Lecture Notes in Control and Information Sciences 385, {\it Springer-Verlag, Berlin}, (2009).

% \bibitem{Lasserre}
% J.B.~Lasserre, ``Global optimization with polynomials and the problem of moments'', {\it SIAM Journal on Optimization}, 11(3) (2001), 796--817.

\bibitem{Lasserre2}
J.B.~Lasserre, ``Convexity in semialgebraic geometry and polynomial optimization'', {\it SIAM Journal on Optimization}, 19(4) (2009), 1995--2014.

\bibitem{LYP}
B.~Legat, C.~Yuan, P.A.~Parrilo, ``Low-rank univariate sum of squares has no spurious local minima'', {\it SIAM Journal on Optimization}, 33(3) (2023), 2041--2061.

\bibitem{Yalmip}
J.~L\"{o}fberg, ``YALMIP: A toolbox for modeling and optimization in MATLAB'', In {\it Proceedings of the IEEE International Conference on Robotics and Automation}, (2004), 284--289.

\bibitem{LofPar}
J.~L\"{o}fberg, P.A.~Parrilo, ``From coefficients to samples: A new approach to SOS optimization'', In {\it Proceedings of the 43rd IEEE Conference on Decision and Control}, (2004), 3154--3159.

\bibitem{NRT}
A.~Nemirovski, C.~Roos, T.~Terlaky, ``On maximization of quadratic form over intersection of ellipsoids with common center'', {\it Mathematical Programming}, 86 (1999), 463--473.

\bibitem{NetzerPlaumann}
T.~Netzer, D.~Plaumann, ``Geometry of Linear Matrix Inequalities'', {\it Springer}, (2023).
%A Course in Convexity and Real Algebraic Geometry with a View Towards Optimization

% \bibitem{NieParSturm}
% J.~Nie, P.A.~Parrilo, B.~Sturmfels, ``Semidefinite Representation of the $k$-Ellipse'', Algorithms in Algebraic Geometry, {\it Springer New York, New York, NY}, (2008), 117--132.

\bibitem{Papp}
D.~Papp, ``Semi-infinite programming using high-degree polynomial interpolants and semidefinite programming'', {\it SIAM Journal on Optimization}, 27(3) (2017), 1858--1879.

\bibitem{papp2013semidefinite}
D.~Papp, F.~Alizadeh, ``Semidefinite characterization of sum-of-squares cones in algebras'', {\it SIAM Journal on Optimization}, 23(3) (2013), 1398--1423.

\bibitem{PappYildiz}
D.~Papp, S.~Yildiz, ``Sum-of-squares optimization without semidefinite programming'', {\it SIAM Journal on Optimization}, 29(1) (2019), 822--851.

\bibitem{ParriloThesis}
P.A.~Parrilo, ``Structured semidefinite programs and semialgebraic geometry methods in robustness and optimization'', Ph.D. Thesis, California Institute of Technology, (2000).

% \bibitem{Parrilo}
% P.A.~Parrilo, ``Semidefinite programming relaxations for semialgebraic problems'', {\it Mathematical Programming}, 96(2) (2003), 293--320.

\bibitem{real_stability_testing}
P.~Raghavendra, N.~Ryder, N.~Srivastava, ``Real stability testing'', In {\it 8th Innovations in Theoretical Computer Science Conference (ITCS 2017), Leibniz International Proceedings in Informatics (LIPIcs)}, Schloss Dagstuhl – Leibniz-Zentrum für Informatik, 67 (2017), 5:1--5:15.

\bibitem{RotaStrang}
G.C.~Rota, W.G.~Strang, ``A note on the joint spectral radius'', {\it Indag. Math.}, 22 (1960), 379--381.

\bibitem{SCPW12}
J.~Saunderson, V.~Chandrasekaran, P.A.~Parrilo, A.S.~Willsky, ``Diagonal and low-rank matrix decompositions, correlation matrices, and ellipsoid fitting'', {\it SIAM Journal on Matrix Analysis and Applications}, 33(4) (2012), 1395--1416.

\bibitem{SPW13}
J.~Saunderson, P.A.~Parrilo, A.S.~Willsky, ``Diagonal and low-rank decompositions and fitting ellipsoids to random point'', In {\it Proceedings of the 52nd IEEE Conference on Decision and Control}, (2013), 6031--6036.

\bibitem{SchererHol}
C.~Scherer, C.~Hol, ``Matrix sum-of-squares relaxations for robust semi-definite programs'', {\it Mathematical Programming}, 107 (2006), 189--211.

\bibitem{Todd}
M.J.~Todd, ``Minimum-Volume Ellipsoids: Theory and Algorithms'', {\it Society for Industrial and Applied Mathematics, Philadelphia, PA, USA}, (2016).

\bibitem{TulWu}
M.~Tulsiani, J.~Wu, ``Ellipsoid fitting up to constant via empirical covariance estimation'', preprint, \url{https://arxiv.org/abs/2307.10941}, (2023).

\bibitem{BoydSurvey}
L.~Vandenberghe, S.~Boyd, ``Semidefinite programming'', {\it SIAM Review}, 38(1) (1996), 49--95.

\bibitem{Julia}
T.~Weisser, B.~Legat, C.~Coey, L.~Kapelevich, V.J.~Pablo, ``Polynomial and moment optimization in Julia and JuMP'', \url{https://pretalx.com/juliacon2019/talk/QZBKAU/}, (2019).

\bibitem{Yakubovich}
V.A.~Yakubovich, ``Factorization of symmetric matrix polynomials'', {\it Dokl. Akad. Nauk SSSR}, 194(3) (1970), 532--535.

\bibitem{Youla}
D.~Youla, ``On the factorization of rational matrices'', {\it IRE Transactions on Information Theory}, 7(3) (1961), 172--189.

\end{thebibliography}
\end{document}